\documentclass[12pt]{article}

\usepackage[numbers]{natbib}
\usepackage{graphicx}
\usepackage{amssymb}
\usepackage{amsmath,amsthm,enumerate}
\usepackage{cancel}

\usepackage[citecolor=blue,colorlinks=true]{hyperref} 

\usepackage[left=2cm,right=2cm,top=3.5cm,bottom=3.5cm]{geometry}
\usepackage{color}

\renewcommand{\footnoterule}{%
  \kern -3pt
  \hrule width \textwidth height 0.4pt
  \kern 2pt
}

\usepackage{mdframed}
\usepackage{extarrows}
\usepackage{stmaryrd}
\usepackage{import}
\usepackage{yfonts}
\usepackage{mathtools}
\usepackage{upgreek}
\usepackage{lmodern}
\usepackage{inputenc}
\usepackage{microtype}
\usepackage{float}
\usepackage{enumitem}
\setitemize{leftmargin=*}
\usepackage{tikz}
\usetikzlibrary{shapes,arrows}
\usepackage{chngcntr}
\usepackage[linesnumbered,german]{algorithm2e}
\usepackage{csquotes}
\usepackage{array}
\usepackage{multirow}
\usepackage{multicol}
\usepackage{cite}
\usepackage{accents}
\usepackage{marvosym}
\usepackage{wasysym}
\usepackage{bm}
\usepackage{dsfont}
\usepackage{xfrac}
\usepackage{hyperref}
\usepackage{xcolor}
\hypersetup{linktoc=all, colorlinks=true, linkcolor=blue, citecolor=blue}
\usepackage{stackengine}

\stackMath

\DeclareFontShape{OMX}{cmex}{m}{b}{<-> cmexb10}{}
\DeclareSymbolFont{boldlargesymbols}{OMX}{cmex}{m}{b}
\DeclareMathAccent{\bwidetilde}{\mathord}{boldlargesymbols}{"65}

\newtheoremstyle{break}
{14pt}{20pt}%
{}{}%
{\bfseries}{\vspace{0.5mm}}%
{\newline}{}%

\theoremstyle{break}

\theoremstyle{plain}
\newtheorem{theorem}{Theorem}[section]

\newtheorem{lemma}[theorem]{Lemma}
\newtheorem{corollary}[theorem]{Corollary}

\theoremstyle{definition}
\newtheorem{definition}[theorem]{Definition}

\theoremstyle{remark}
\newtheorem{remark}[theorem]{Remark}

\catcode`\@=11 \@addtoreset{equation}{section}

\catcode`\@=12

\allowdisplaybreaks

\definecolor{Cgrey}{rgb}{0.85,0.85,0.85}
\definecolor{Cblue}{rgb}{0.50,0.85,0.85}
\definecolor{Cred}{rgb}{1,0,0}
\definecolor{fancy}{rgb}{0.10,0.85,0.10}

\newcommand\Cbox[2]{%
    \newbox\contentbox%
    \newbox\bkgdbox%
    \setbox\contentbox\hbox to \hsize{%
        \vtop{
            \kern\columnsep
            \hbox to \hsize{%
                \kern\columnsep%
                \advance\hsize by -2\columnsep%
                \setlength{\textwidth}{\hsize}%
                \vbox{
                    \parskip=\baselineskip
                    \parindent=0bp
                    #2
                }%
                \kern\columnsep%
            }%
            \kern\columnsep%
        }%
    }%
    \setbox\bkgdbox\vbox{
        \color{#1}
        \hrule width  \wd\contentbox %
               height \ht\contentbox %
               depth  \dp\contentbox
        \color{black}
    }%
    \wd\bkgdbox=0bp%
    \vbox{\hbox to \hsize{\box\bkgdbox\box\contentbox}}%
    \vskip\baselineskip%
}

\DeclareMathAlphabet{\mathbbm}{U}{bbm}{m}{n}
\newcommand{\softd}{d\hspace{-0.2mm}'}
\renewcommand{\S}{\mathbb{S}}
\newcommand{\Id}{\mathbb{I}}
\newcommand{\Lp}[2]{L^{\raisebox{#2}{\scalebox{0.75}{$#1$}}}}
\newcommand{\dy}{\mathrm{d}\mathbf{y}}

\newcommand{\bs}{\backslash\hspace{0.2mm}}

\renewcommand{\div}{\mathrm{div}}

\newcommand{\hb}{h^\beta}
\newcommand{\hd}{h^\delta}
\newcommand{\he}{h^\varepsilon}

\newcommand{\hee}{\frac{\he}{2}}

\newcommand{\reals}{\mathbb{R}}
\newcommand{\fatu}{\bm{u}}
\newcommand{\fatv}{\bm{v}}
\newcommand{\fatw}{\bm{w}}
\newcommand{\fatx}{\bm{x}}
\newcommand{\faty}{\bm{y}}

\newcommand{\fatf}{\bm{f}}
\newcommand{\fatr}{\bm{r}}
\newcommand{\fats}{\bm{s}}
\newcommand{\fatg}{\bm{g}}
\newcommand{\bvarrho}{\bwidetilde{\varrho}}
\newcommand{\bfatu}{\bwidetilde{\fatu}}
\newcommand{\btheta}{\bwidetilde{\theta}}

\newcommand{\ei}[1]{\mathbf{e}_{#1}}

\newcommand{\fatn}{\bm{n}}
\newcommand{\const}{\mathrm{const}}
\newcommand{\naturals}{\mathbb{N}}

\newcommand{\conv}{\mathrm{conv}}

\newcommand{\p}{p\hspace{0.2mm}}

\newcommand{\sk}{\sum_{K\scalebox{0.75}{$\,\in\,$} \mathcal{T}_h}} 

\newcommand{\gh}{\mathcal{E}_{h}} 
\newcommand{\ghint}{\mathcal{E}_{h,\mathrm{int}}}
\newcommand{\gint}{\mathcal{E}_{\mathrm{int}}}

\newcommand{\ghext}{\mathcal{E}_{h,\mathrm{ext}}} 

\newcommand{\sginth}{\sum_{\sigma\,\in\,\mathcal{E}_{h,\mathrm{int}}}}

\newcommand{\sgkh}{\sum_{\sigma\scalebox{0.75}{$\,\in\,$}\mathcal{E}_h(K)}}

\newcommand{\po}{\partial \hspace{0.1mm} \Omega}
\newcommand{\poh}{\partial \Omega_h} 

\newcommand{\ngamma}{\fatn_\sigma}
\newcommand{\nk}{\fatn_{\scalebox{0.7}{$K$}}} 

\newcommand{\deltat}{\Delta t}
\newcommand{\thetakh}{\theta_{\hspace{-0.3mm}h}^{k}}
\newcommand{\thetakmh}{\theta_{\hspace{-0.3mm}h}^{\hspace{0.3mm}k-1}}
\newcommand{\thetah}{\theta_{\hspace{-0.3mm}h}}

\newcommand{\thetazeroh}{\theta_{\hspace{-0.3mm}h}^{0}}
\newcommand{\thetazero}{\theta_{\hspace{-0.3mm}0}}

\newcommand{\inttinteint}{\intt\!\!\inteint}
\newcommand{\inteint}{\int_{\mathcal{E}_\mathrm{int}}}

\newcommand{\intek}{\int_{\mathcal{E}(K)}}

\newcommand{\intt}{\int^{T}_{0}}
\newcommand{\inttau}{\int_{0}^{\hspace{0.3mm}\tau}}

\newcommand{\intk}{\int_{\hspace{0.2mm}\raisebox{-0.3mm}{\scalebox{0.75}{$\!K$}}}} 
\newcommand{\into}{\int_{\hspace{0.2mm}\raisebox{-0.3mm}{\scalebox{0.75}{$\Omega$}}}}
\newcommand{\intoh}{\int_{\Omega_h}} 
\newcommand{\intgc}{\int_{\raisebox{-0.3mm}{\scalebox{0.75}{$\sigma$}}}} 
\newcommand{\intg}{\intgc} 
\newcommand{\inttinto}{\intt\!\!\into}
\newcommand{\inttintoh}{\intt\!\!\intoh} 

\newcommand{\inttauinto}{\inttau\!\!\into}
\newcommand{\dx}{\mathrm{d}\bm{x}}
\newcommand{\dsx}{\mathrm{d}S_{\bm{x}}}
\newcommand{\dt}{\mathrm{d}t}

\newcommand{\dxdt}{\dx\hspace{0.3mm}\dt}
\newcommand{\dsxdt}{\dsx\hspace{0.3mm}\dt}
\newcommand{\dd}{\mathrm{d}}

\newcommand{\divh}{\mathrm{div}_{\raisebox{-0.3mm}{\hspace{-0.3mm}\scalebox{0.75}{$h$}}}} 
\newcommand{\Dt}{D_{t}}

\newcommand{\gradh}{\nabla_{\!h}} 
\newcommand{\pt}{\partial_t}
\newcommand{\gradx}{\nabla_{\hspace{-0.7mm}\fatx}}
\newcommand{\divx}{\mathrm{div}_{\fatx}}

\newcommand{\chip}{\chi^{\hspace{0.3mm}\prime}}
\newcommand{\bp}{b^{\hspace{0.3mm}\prime}}
\newcommand{\bpp}{b^{\hspace{0.3mm}\prime\prime}}
\newcommand{\Pp}{P^{\hspace{0.3mm}\prime}}
\newcommand{\Ppp}{P^{\hspace{0.3mm}\prime\prime}}

\newcommand{\jump}[1]{\llbracket\hspace{0.2mm}#1\hspace{0.2mm}\rrbracket}
\newcommand{\jumpG}[1]{\llbracket\hspace{0.2mm}#1\hspace{0.2mm}\rrbracket_\sigma}
\newcommand{\Jump}[1]{\left\llbracket\hspace{0.2mm}#1\hspace{0.2mm}\right\rrbracket}

\newcommand{\mean}[1]{\{#1\}}
\newcommand{\meanG}[1]{\{#1\}_\sigma}

\newcommand{\omean}[1]{\langle #1 \rangle}

\newcommand{\smallgmean}[1]{\langle #1 \rangle_\sigma}

\newcommand{\piqh}{\Pi_{Q,\hspace{0.2mm}h}} 
\newcommand{\pivh}{\Pi_{V,\hspace{0.2mm}h}} 
\newcommand{\piq}{\Pi_{Q}} 
\newcommand{\piv}{\Pi_{V}} 

\newcommand{\uppG}[2]{\mathrm{Up}\hspace{-0.5mm}\left[#1\hspace{0.3mm},\hspace{-0.3mm}#2\right]_{\sigma}}
\newcommand{\upp}[2]{\mathrm{Up}\hspace{-0.5mm}\left[#1\hspace{0.3mm},\hspace{-0.3mm}#2\right]}
\newcommand{\up}[2]{F_h^{\,\mathrm{up}}\hspace{-0.5mm}\left[#1\hspace{0.3mm},\hspace{-0.3mm}#2\right]}
\newcommand{\upG}[2]{F_h^{\,\mathrm{up}}\hspace{-0.5mm}\left[#1\hspace{0.3mm},\hspace{-0.3mm}#2\right]_{\sigma}}

\newcommand{\Lnorm}[5]{\left|\left|#1\right|\right|_{\raisebox{-0.5mm}{$\scalebox{0.75}{$L$}^{\scalebox{#4}{\raisebox{#5}{$#2$}}}\hspace{-0.2mm}\scalebox{0.75}{$#3$}$}}}
\newcommand{\lnorm}[5]{||#1||_{\raisebox{-1.3mm}{$\scalebox{0.75}{$L$}^{\scalebox{#4}{\raisebox{#5}{$#2$}}}\hspace{-0.2mm}\scalebox{0.75}{$#3$}$}}}

\newcommand{\wnorm}[5]{||#1||_{\raisebox{-1.3mm}{$\scalebox{0.75}{$W$}^{\scalebox{#4}{\raisebox{#5}{$#2$}}}\hspace{-0.2mm}\scalebox{0.75}{$#3$}$}}}
\newcommand{\expwnorm}[7]{||#1||_{\raisebox{-0.8mm}{$\scalebox{0.75}{$W$}^{\scalebox{#4}{\raisebox{#5}{$#2$}}}\hspace{-0.2mm}\scalebox{0.75}{$#3$}$}}^{\raisebox{#6}{\scalebox{0.75}{$#7$}}}}

\newcommand{\explnorm}[7]{||#1||_{\raisebox{-0.8mm}{$\scalebox{0.75}{$L$}^{\scalebox{#4}{\raisebox{#5}{$#2$}}}\hspace{-0.2mm}\scalebox{0.75}{$#3$}$}}^{\raisebox{#6}{\scalebox{0.75}{$#7$}}}}

\newcommand{\maxnormk}[3]{||#1||_{\raisebox{-1.3mm}{\scalebox{0.75}{$\Ck{#2}#3$}}}}

\newcommand{\Ck}[1]{C^{\hspace{0.3mm}#1}}

\newcommand{\Ckc}[1]{C^{\hspace{0.3mm}#1}_{\raisebox{0.2mm}{\scalebox{0.75}{$c$}}}}

\newcommand{\Cone}{C^{\hspace{0.3mm}1}}

\newcommand{\Cinfty}{C^{\raisebox{0.3mm}{\scalebox{0.6}{$\hspace{0.5mm}\infty$}}}}

\newcommand{\Cinftyc}{C^{\raisebox{0.3mm}{\scalebox{0.6}{$\hspace{0.5mm}\infty$}}}_{\raisebox{0.2mm}{\scalebox{0.75}{$c$}}}}
\newcommand{\Pzero}{P_{\raisebox{-0.5mm}{\scalebox{0.75}{$0$}}}}
\newcommand{\Pone}{P_{\raisebox{-0.5mm}{\scalebox{0.75}{$1$}}}}

\newcommand{\barepsilon}{\overline{\vphantom{\scalebox{1.1}{$\varepsilon$}}\varepsilon}}
\newcommand{\bardelta}{\overline{\vphantom{\scalebox{1.1}{$\delta$}}\delta}}

\makeatletter
\renewenvironment{proof}[1][\proofname]{%
   \par\pushQED{\qed}\normalfont%
   \topsep6\p@\@plus6\p@\relax
   \trivlist\item[\hskip\labelsep\itshape\bfseries#1\@addpunct{.}]%
   \ignorespaces
}{%
   \popQED\endtrivlist\@endpefalse
}
\makeatother

\date{}

\begin{document}


\title{Existence of dissipative solutions to the compressible Navier-Stokes system with potential temperature transport\thanks{This work has been funded by the Deutsche Forschungsgemeinschaft (DFG, German Research Foundation) - Project number 233630050 - TRR 146 as well as by TRR 165 Waves to Weather. M.L. gratefully acknowledges support of the Gutenberg Research College of University Mainz. The authors wish to thank E.~Feireisl (Prague) and A.~Novotn\'y (Toulon) for fruitful discussions.}}

\author{M\' aria Luk\' a\v cov\' a\,-Medvi\softd ov\' a
\and Andreas Sch\"omer
}

\date{\today}

\maketitle

\bigskip

\centerline{Institute of Mathematics, Johannes Gutenberg-University Mainz}
\centerline{Staudingerweg 9, 55128 Mainz, Germany}
\centerline{lukacova@uni-mainz.de, anschoem@uni-mainz.de}

\begin{abstract}
We introduce dissipative solutions to the compressible Navier-Stokes system with potential temperature transport motivated by the concept of Young measures. We prove their global-in-time existence by means of convergence analysis of a mixed finite element-finite volume method. If a strong solution to the compressible Navier-Stokes system with potential temperature transport exists, we prove the strong convergence of numerical solutions. Our results hold for the full range of adiabatic indices including the physically relevant cases in which the existence of global-in-time weak solutions is open.
\end{abstract}



\section{Introduction}\label{sec_intro}

We consider a compressible viscous Newtonian fluid  that is confined to a bounded domain $\Omega\subset\reals^d$, $d\in\{2,3\}$. Its time evolution is governed by the following system:
\begin{align}
    \pt\varrho + \divx(\varrho\fatu) &= 0\phantom{\divx(\S(\gradx\fatu))} \qquad\qquad \text{in $(0,T)\times\Omega$,} \label{cequation} \\[2mm]
    \pt(\varrho\fatu) + \divx(\varrho\fatu\otimes\fatu) + \gradx\hspace{0.2mm} \p(\varrho\theta) &= \divx(\S(\gradx\fatu))\phantom{0} \qquad\qquad \text{in $(0,T)\times\Omega$,} \label{mequation} \\[2mm]
    \pt(\varrho\theta) + \divx(\varrho\theta\fatu) &= 0\phantom{\divx(\S(\gradx\fatu))} \qquad\qquad \text{in $(0,T)\times\Omega$.} \label{zequation}
\end{align}
Here $\varrho\geq 0$, $\fatu$, $p$ and $\theta\geq 0$  stand for the \textit{fluid density}, \textit{velocity},  \textit{pressure}, and \textit{potential temperature}, respectively. The \textit{viscous stress tensor} $\S(\gradx\fatu)$ is given by
\begin{align}
    \S(\gradx\fatu) = \mu\!\left(\gradx\fatu + (\gradx\fatu)^T-\frac{2}{d}\,\divx(\fatu)\Id\right)+\lambda\,\divx(\fatu)\,\Id\,, \label{stress_tensor_0}
\end{align}
where $\mu$ and $\lambda$ are \textit{viscosity constants} satisfying
$    \mu > 0 $ and  $\lambda \geq -\frac{2}{d}\,\mu\, .$
Denoting by $\gamma > 1$ the adiabatic index, the pressure state equation reads
\begin{align}
    \p(\varrho\theta) = a(\varrho\theta)^\gamma\,, \qquad a>0\,. \label{isen_press}
\end{align}
This type of Navier-Stokes equations is often used in meteorological applications; see, e.g., \citep{klein1} and the references therein.
System \eqref{cequation}--\eqref{isen_press}
governs the motion of viscous compressible fluids
with potential temperature, where diabatic processes and the influence of molecular transport on potential
temperature are excluded. Only potential entropy stratification in the initial data is imposed. We refer a reader to
Feireisl et al.~\citep{Klein}, where the singular limit
in the low Mach/Froude number regime of the above Navier-Stokes system with $\gamma > 3/2$ was analyzed. For $\gamma > 9/5$, Bresch et al.~\citep{bresch}  showed that the low Mach number limit for the considered system is the compressible isentropic Navier–Stokes equation.
 In \citep{Lukacova_Wiebe_1} Luká\v{c}ová-Medvi\softd ová et al. use a slightly more complex version of the above system
 as the basis for their cloud model; see also Chertock et al.~\citep{Lukacova_Wiebe_2}, where the uncertainty quantification was investigated.
Due to the link between potential temperature and entropy, system (\ref{cequation})--(\ref{isen_press}) is often reported in the literature as the \textit{Navier-Stokes system with entropy transport}. To avoid any misunderstanding, we call it in the present paper the \textit{Navier-Stokes system with potential temperature transport}.

In literature we can find several existence results for the Navier-Stokes system \eqref{cequation}--\eqref{isen_press}.
The question of stability of weak solutions for $\gamma > 3/2, \, d=3$ was analyzed by Mich\'alek \citep{michalek}; see also \citep{Lions}, where the stability of weak solutions for the compressible Navier-Stokes equations with a scalar transport was studied for $\gamma > 9/5$ by Lions.  Under the assumption $\gamma\geq 9/5$ in the case $d=3$ and $\gamma>1$ in the case $d=2$, system (\ref{cequation})--(\ref{isen_press}) is known to admit global-in-time weak solutions; see Maltese et al.~\citep[Theorem 1 with $\mathcal{T}(s)=s^\gamma$]{Zatorska}. Note that in the aforementioned paper the authors work with the entropy $s$ instead of the potential temperature $\theta$. However, in their framework the specified choice of the function $\mathcal{T}$ yields $s=\theta$. We point out that the physically relevant adiabatic indices $\gamma$ lie in the interval $(1,2]$ if $d=2$ and in the interval $(1,5/3]$ if $d=3$. Consequently, in three space dimensions there are physically relevant values of the adiabatic index for which the global-in-time existence of weak solutions remains an open problem for the Navier-Stokes system \eqref{cequation}-\eqref{isen_press}.

A simpler model for viscous compressible fluid flow is the barotropic Navier-Stokes system with the state equation $\p = a \varrho^\gamma,$ $a= \const.$ The first global-in-time existence result for weak solutions of this system allowing general initial data was established
in 1998  by Lions \citep{Lions} for
$\gamma\geq 3/2$ if $d=2$ and $\gamma\geq 9/5$ if $d=3$.  In 2001, Feireisl, Novotn\'y, and Petzeltov\'a \citep{Feireisl_Novotny_Petzeltova}  extended Lions's result to the situation $\gamma>1$ for $d=2$ and $\gamma>3/2$ for $d=3$; see also Feireisl, Karper, Pokorný \citep{Feireisl_Karper_Pokorny}. To date, the latter is the best available global-in-time existence result for weak solutions for the barotropic Navier-Stokes system. The main obstacle that hampers the derivation of the existence result for $\gamma\leq 3/2$ in three space dimensions is the lack of suitable a priori estimates for the convective term $\varrho\fatu\otimes\fatu $. These difficulties are inherited by the full Navier-Stokes-Fourier system
that includes an energy equation, too. In \citep{feireisl_novotny}, Feireisl and Novotn\'y  obtained
the existence of global-in-time weak solutions for the Navier-Stokes-Fourier system. However, their result holds only
for a very restrictive class of state equations. In particular, the natural example
of the perfect gas law $p = \varrho \theta$ is still open for the existence of weak solutions. In this context, we refer a reader to
\citep{Lukacova_IMA}, where the complete Navier-Stokes-Fourier system for the perfect gas was studied in the context of generalized solutions.

The question of uniqueness of weak solutions remains open in general. However, we have a \textit{weak-strong uniqueness principle} for the barotropic Navier-Stokes equations. It means
that weak and strong solutions to the Navier-Stokes system emanating from the same initial data coincide; see, e.g., Feireisl, Jin, Novotn\'y \citep{Feireisl_Jin_Novotny} or Feireisl \citep{Feireisl_Weak_Strong_2019}.

In \citep{Feireisl_Wiedemann}, Feireisl et al. introduced a new concept of generalized solutions to the barotropic Navier-Stokes system. They work with the so-called dissipative measure-valued (DMV) solutions that are motivated by the concept of Young measures. In this context, a DMV-strong uniqueness principle was established and the existence of global-in-time DMV solutions for a class of pressure state equations including the barotropic case with $\gamma\geq 1$ was achieved. In our recent work \citep{Lukacova_Schoemer}, we have extended the DMV-strong uniqueness result to the Navier-Stokes system with potential temperature transport \eqref{cequation}--\eqref{isen_press}.

In \citep[Chapter 13]{Feireisl_Lukacova_Book}, Feireisl et al.~give a constructive existence proof and demonstrate that DMV solutions to the barotropic Navier-Stokes system can also be obtained by means of a convergent numerical method that was originally developed by Karlsen and Karper \citep{Karlsen_Karper_1}, \citep{Karlsen_Karper_2}, \citep{Karlsen_Karper_3}, \citep{Karper}. However, their result is based on the assumption that $\gamma>6/5$ if $d=3$ and $\gamma>8/7$ if $d=2$; for the three-dimensional case see also Feireisl and Luk\'a\v cov\'a\,-Medvi\softd ov\'a \citep{Feireisl_Lukacova}.

The goal of this paper is to introduce a concept of DMV solutions to the Navier-Stokes system with potential temperature transport and prove the global-in-time existence of such generalized solutions for all $\gamma > 1$ by analyzing the convergence of a suitable numerical scheme.  To this end, we propose a new version of the mixed finite element-finite volume method of
Karlsen and Karper \citep{Karlsen_Karper_1}; see also \citep{Feireisl_Karper_Pokorny}, \citep[Chapter 13]{Feireisl_Lukacova_Book},  \citep{Feireisl_Lukacova}.

The paper is organized as follows: In Section \ref{sec_meas_sol}, we introduce our notion of DMV solutions to the Navier-Stokes system with potential temperature transport and present our main result. Section~\ref{sec_num_scheme} is devoted to the numerical method and the collection of its basic properties. In Section~\ref{sec_stability}, we state a discrete energy equality for our method which serves as a basis for several stability estimates. The consistency of the numerical method is established in Section~\ref{sec_consistency} and in Section~\ref{sec_convergence} we conclude that any Young measure generated by the solutions to our numerical method represents a DMV solution to the Navier-Stokes system with potential temperature transport. In particular, we show that the numerical solutions converge weakly to the expected values with respect to the Young measure.
The convergence of numerical solutions is strong as long as a strong solution of  \eqref{cequation}--\eqref{isen_press} exists.

\section{Dissipative measure\hspace{0.5mm}-valued solutions}\label{sec_meas_sol}
Before defining dissipative measure-valued solutions to the Navier-Stokes system with potential temperature transport, we fix the initial and boundary conditions. The Navier-Stokes system with potential temperature transport (\ref{cequation})--(\ref{isen_press}) is endowed with the initial data
\begin{align}
    \varrho(0,\cdot) = \varrho_0 \,, \qquad \theta(0,\cdot) = \thetazero \,, \qquad \fatu(0,\cdot) = \fatu_0\,, \label{initial}
\end{align}
and the no-slip boundary condition
\begin{align}
    \fatu|_{[0,T]\times\po} = \mathbf{0}\,. \label{no_slip}
\end{align}
We henceforth write $\Omega_t = (0,t)\times\Omega$ whenever $t>0$.
Furthermore, $P:[0,\infty)\to\reals$,
\begin{align}
    P(z) = \frac{a}{\gamma-1}\,z^\gamma\,, \label{press_pot}
\end{align}
is the so-called \textit{pressure potential}. If $\mathcal{V}=\{\mathcal{V}_{(t,\fatx)}\}_{(t,\fatx)\,\in\,\Omega_T}$ is a parametrized probability measure (Young measure) acting on $\reals^{d+2}$, we write
\begin{align*}
    \omean{\mathcal{V}_{(t,\fatx)};g}\equiv\int_{\reals^{d+2}}g\;\dd\mathcal{V}_{(t,\fatx)}\equiv\int_{\reals^{d+2}}g(\bvarrho,\btheta,\bfatu)\;\dd\mathcal{V}_{(t,\fatx)}(\bvarrho,\btheta,\bfatu)
\end{align*}
whenever $g\in C(\reals^{d+2})$. Moreover, we tend to write out the function $g$ in terms of the integration variables $(\bvarrho,\btheta,\bfatu)\in\reals\times\reals\times\reals^d\cong\reals^{d+2}$: if, for example, $g(\bvarrho,\btheta,\bfatu)=\bvarrho\,\bfatu$, then we also write
\begin{align*}
    \omean{\mathcal{V}_{(t,\fatx)};\bvarrho\,\bfatu} \qquad \text{instead of} \qquad \omean{\mathcal{V}_{(t,\fatx)};g}\,.
\end{align*}

We proceed by defining dissipative measure-valued solutions to the Navier-Stokes system with potential temperature transport (\ref{cequation})--(\ref{isen_press}).

\begin{definition}[\textbf{DMV solutions}]\label{def_meas_sol}
A parametrized probability measure $\mathcal{V}=\{\mathcal{V}_{(t,\fatx)}\}_{(t,\fatx)\,\in\,\Omega_T}$ that satisfies
\begin{gather*}
    \mathcal{V}\in\Lp{\infty}{0.2mm}_{\mathrm{weak}^\star}(\Omega_T;\mathcal{P}(\reals^{d+2}))\,\footnotemark, \qquad \reals^{d+2} = \big\{(\bvarrho,\btheta,\bfatu)\,\big|\,\bvarrho,\btheta\in\reals, \bfatu\in\reals^d\big\}\,,
\end{gather*}
and for which there exists a constant $c_\star>0$ such that
\begin{gather*}
    \mathcal{V}_{(t,\fatx)}\big(\{\bvarrho\geq 0\}\cap\{\btheta\geq c_\star\}\big) = 1 \quad \text{for a.a. $(t,\fatx)\in\Omega_T$,}
\end{gather*}
\footnotetext{$\mathcal{P}(\reals^{d+2})$ denotes the space of probability measures on $\reals^{d+2}$.}is called a \textit{dissipative measure-valued (DMV) solution} to the Navier-Stokes system with potential temperature transport (\ref{cequation})--(\ref{isen_press}) with initial and boundary conditions (\ref{initial}) and (\ref{no_slip}) if it satisfies:
\begin{itemize}
    \item{(\textbf{energy inequality})
    \begin{align*}
    \fatu_{\mathcal{V}}\equiv\omean{\mathcal{V};\bfatu}\in\Lp{2}{0mm}(0,T;W^{1,2}_0(\Omega)^d)\,, \qquad \left\langle\mathcal{V};\frac{1}{2}\,\bvarrho\,|\bfatu|^2+P(\bvarrho\,\btheta)\right\rangle\in\Lp{1}{0mm}(\Omega_T)\,,
    \end{align*}
    and the integral inequality
    \begin{align}
        &\into\left\langle\mathcal{V}_{(\tau,\,\cdot\,)};\frac{1}{2}\,\bvarrho\,|\bfatu|^2+P(\bvarrho\,\btheta)\right\rangle\dx + \inttauinto \,\mathbb{S}(\gradx\fatu_{\mathcal{V}}):\gradx\fatu_{\mathcal{V}}\;\dxdt \notag \\[2mm]
        &\qquad \qquad \qquad \qquad + \int_{\,\overline{\Omega}}\dd\mathfrak{E}(\tau) + \int_{\,\overline{\Omega_\tau}}\dd\mathfrak{D} \leq \into\left[\,\frac{1}{2}\,\varrho_0|\fatu_0|^2+P(\varrho_0\thetazero)\right]\dx \label{energy_inequ}
    \end{align}
    holds for a.a. $\tau\in(0,T)$ with the \textit{energy concentration defect}
    \begin{align*}
    \mathfrak{E}\in\Lp{\infty}{0.2mm}_{\mathrm{weak}^\star}(0,T;\mathcal{M}^+(\overline{\Omega}))
    \end{align*}
    and the \textit{dissipation defect}
    \begin{align*}
    \mathfrak{D}\in\mathcal{M}^+(\,\overline{\Omega_T})\,;
    \end{align*}}
    \item{(\textbf{continuity equation})
    \begin{align*}
        \omean{\mathcal{V};\bvarrho\,}\in C_{\mathrm{weak}}([0,T];\Lp{\gamma}{0.5mm}(\Omega))\,, \quad \omean{\mathcal{V}_{(0,\fatx)};\bvarrho\,} = \varrho_0(\fatx) \;\; \text{for a.a. $\fatx\in\Omega$}
    \end{align*}
    and the integral identity
    \begin{align}
        \left[\,\into\,\omean{\mathcal{V}_{(t,\,\cdot\,)};\bvarrho\,}\,\varphi(t,\cdot)\;\dx\right]_{t\,=\,0}^{t\,=\,\tau} = \inttauinto\Big[\omean{\mathcal{V};\bvarrho\,}\,\pt\varphi + \omean{\mathcal{V};\bvarrho\,\bfatu}\cdot\gradx\varphi\Big]\;\dxdt
    \end{align}
    holds for all $\tau\in[0,T]$ and all $\varphi\in W^{1,\infty}(\Omega_T)$\footnotemark;\footnotetext{Here, the (Lipschitz) continuous representative of $\varphi\in W^{1,\infty}(\Omega_T)$ is meant.}}
    \item{(\textbf{momentum equation})
    \begin{align*}
        \omean{\mathcal{V};\bvarrho\,\bfatu}\in C_{\mathrm{weak}}([0,T];\Lp{\frac{2\gamma}{\gamma+1}}{0.5mm}(\Omega)^d)\,, \quad \omean{\mathcal{V}_{(0,\fatx)};\bvarrho\,\bfatu} = \varrho_0(\fatx)\fatu_0(\fatx) \;\; \text{for a.a. $\fatx\in\Omega$}
    \end{align*}
    and the integral identity
    \begin{align}
    \left[\,\into\,\omean{\mathcal{V}_{(t,\,\cdot\,)};\bvarrho\,\bfatu}\cdot\bm{\varphi}(t,\cdot)\;\dx\right]_{t\,=\,0}^{t\,=\,\tau}
    &= \inttauinto\Big[\omean{\mathcal{V};\bvarrho\,\bfatu}\cdot\pt\bm{\varphi} + \omean{\mathcal{V};\bvarrho\,\bfatu\otimes\bfatu+\p(\bvarrho\,\btheta)\Id}:\gradx\bm{\varphi}\Big]\;\dxdt \notag \\[2mm]
    &\phantom{\,=\,}-\inttauinto \,\mathbb{S}(\gradx\fatu_{\mathcal{V}}):\gradx\bm{\varphi}\;\dxdt + \inttauinto\gradx\bm{\varphi}:\dd\bm{\mathfrak{R}}(t)\hspace{0.3mm}\dt
    \end{align}
    holds for all $\tau\in[0,T]$ and all $\bm{\varphi}\in\Cone(\,\overline{\Omega_T})^d$ satisfying $\bm{\varphi}|_{[0,T]\times\po}=\bm{0}$, where the \textit{Reynolds concentration defect} fulfills
    \begin{gather*}
        \bm{\mathfrak{R}}\in\Lp{\infty}{0.2mm}_{\mathrm{weak}^\star}(0,T;\mathcal{M}(\overline{\Omega})^{d\times d}_{\mathrm{sym},+})\,\footnotemark \\[2mm]
        \text{and} \qquad \underline{d}\hspace{0.3mm}\mathfrak{E}\leq\mathrm{tr}(\bm{\mathfrak{R}})\leq\overline{d}\hspace{0.3mm}\mathfrak{E} \quad \text{for some constants $\overline{d}\geq\underline{d}> 0$;}
    \end{gather*}
    \footnotetext{$\mathcal{M}(\overline{\Omega})^{d\times d}_{\mathrm{sym},+}$ denotes the set of bounded Radon measures defined on $\overline{\Omega}$ and ranging in the set of symmetric positive semi-definite matrices, i.e., $\mathcal{M}(\overline{\Omega})^{d\times d}_{\mathrm{sym},+} = \left\{\mu\in\mathcal{M}(\overline{\Omega})^{d\times d}_{\mathrm{sym}}\,\left|\,\int_{\,\overline{\Omega}}\,\phi(\xi\otimes\xi):\dd\mu \geq 0 \;\text{for all} \; \xi\in\reals^d, \; \phi\in C(\overline{\Omega}), \; \phi\geq 0\right.\right\}$.}}
    \item{(\textbf{potential temperature equation})
    \begin{align*}
        \omean{\mathcal{V};\bvarrho\,\btheta\,}\in C_{\mathrm{weak}}([0,T];\Lp{\gamma}{0.5mm}(\Omega))\,, \quad \omean{\mathcal{V}_{(0,\fatx)};\bvarrho\,\btheta\,} = \varrho_0(\fatx)\thetazero(\fatx) \;\; \text{for a.a. $\fatx\in\Omega$}
    \end{align*}
    and the integral identity
    \begin{align}
        \left[\,\into\,\omean{\mathcal{V}_{(t,\,\cdot\,)};\bvarrho\,\btheta\,}\,\varphi(t,\cdot)\;\dx\right]_{t\,=\,0}^{t\,=\,\tau} = \inttauinto\Big[\omean{\mathcal{V};\bvarrho\,\btheta\,}\,\pt\varphi + \omean{\mathcal{V};\bvarrho\,\btheta\,\bfatu}\cdot\gradx\varphi\Big]\;\dxdt
    \end{align}
    holds for all $\tau\in[0,T]$ and all $\varphi\in W^{1,\infty}(\Omega_T)$;}
    \item{(\textbf{entropy inequality})
    \begin{align*}
        \omean{\mathcal{V}_{(0,\fatx)};\bvarrho\hspace{0.3mm}\ln(\btheta)} = \varrho_0(\fatx)\ln(\thetazero(\fatx)) \;\;\text{for a.a. $\fatx\in\Omega$}
    \end{align*}
    and for any $\psi\in W^{1,\infty}(\Omega_T)$, $\psi\geq 0$, the integral inequality
    \begin{align}
        \left[\,\into\omean{\mathcal{V}_{(t,\,\cdot\,)};\bvarrho\hspace{0.3mm}\ln(\btheta)}\,\psi(t,\cdot)\;\dx\right]^{t\,=\,\tau}_{t\,=\,0} \geq \inttauinto \big[\omean{\mathcal{V};\bvarrho\hspace{0.3mm}\ln(\btheta)}\,\pt\psi + \omean{\mathcal{V};\bvarrho\hspace{0.3mm}\ln(\btheta)\bfatu}\cdot\gradx\psi\big]\,\dxdt \label{entropy_inequ}
    \end{align}
    is satisfied for a.a. $\tau\in(0,T)$;}
    \item{(\textbf{Poincaré's inequality}) \\
    there exists a constant $C_P>0$ such that
    \begin{align}
    \inttauinto \,\omean{\mathcal{V};|\bfatu-\bm{U}|^2}\;\dxdt \leq C_P\!\left(\,\inttauinto|\gradx(\fatu_{\mathcal{V}}-\bm{U})|^2\;\dxdt+ \inttau\!\!\int_{\,\overline{\Omega}}\,\dd\mathfrak{E}(t)\hspace{0.3mm}\dt + \int_{\,\overline{\Omega_\tau}}\dd\mathfrak{D}\right) \label{poincare_inequ}
    \end{align}
    for a.a. $\tau\in(0,T)$ and all $\bm{U}\in\Lp{2}{0mm}(0,T;W^{1,2}_0(\Omega)^d)$.}
\end{itemize}
\end{definition}


\begin{remark}
Note that the physical entropy $S$ is proportional to $\varrho \ln( \theta )$.  We require that our dissipative solutions satisfy the Second Law of Thermodynamics that is expressed by \eqref{entropy_inequ} for adiabatic processes.
The entropy inequality (\ref{entropy_inequ}) and Poincaré's inequality (\ref{poincare_inequ}) included in the definition of DMV solutions to the Navier-Stokes system with potential temperature transport are fundamental to guarantee DMV-strong uniqueness; see \citep{Lukacova_Schoemer}.
\end{remark}

We are ready to formulate the main result of this paper: the existence of DMV solutions to the Navier-Stokes system with potential temperature transport.

\begin{theorem}[\textbf{Existence of DMV solutions}]\label{thm_main}
Let $\gamma>1$, $T>0$, $d\in\{2,3\}$, and $\Omega\subset\reals^d$ a bounded Lipschitz domain. Further, let $\varrho_0,\thetazero\in\Lp{\infty}{0.5mm}(\Omega)$ and $\fatu_0\in W^{1,2}_0(\Omega)^d$, where
\begin{align}
    \varrho_0 > 0 \quad \text{a.e. in $\Omega$} \qquad \text{and} \qquad c_\star < \thetazero < c^\star \quad \text{a.e. in $\Omega$} \label{initial_2}
\end{align}
for some constants $0<c_\star < c^\star$.
Then there is a DMV solution $\mathcal{V}$ to system $\mathrm{(\ref{cequation})}$--$\mathrm{(\ref{isen_press})}$ subject to the initial and boundary conditions $\mathrm{(\ref{initial})}$ and $\mathrm{(\ref{no_slip})}$ that additionally satisfies
\begin{align}
    \mathcal{V}_{(t,\fatx)}\big(\{\btheta\leq c^\star\}\big) = 1 \quad \text{for a.a. $(t,\fatx)\in\Omega_T$.}\label{abcde}
\end{align}
\phantom{.}\hfill
\end{theorem}



\section{Numerical scheme}\label{sec_num_scheme}
In this section, we present our numerical method, the mixed finite element-finite volume method.

\subsection{Spatial discretization}
We choose $H\in(0,1)$ and approximate the spatial domain $\Omega\subset\reals^d$ by a family $\{\Omega_h\}_{h\,\in\,(0,H]}$ that is related to a family of (finite) meshes $(\mathcal{T}_h)_{h\,\in\,(0,H]}$ by the constraint
\begin{align*}
    \overline{\Omega_h} = \bigcup_{K\,\in\,\mathcal{T}_h} K \qquad \text{for all $h\in(0,H]$.}
\end{align*}
We assume that the subsequent conditions hold:
\begin{itemize}
\item{$\Omega_h\subset\Omega$ for all $h\in(0,H]$;}
\item{each element $K$ of a mesh $\mathcal{T}_h$ is a $d$-simplex that can be written as
    \begin{align*}
        K = h\mathbb{A}_K (K_{\mathrm{ref}}) + \mathbf{a}_K\,, \qquad \mathbb{A}_K\in\reals^{d\times d}\,, \qquad \mathbf{a}_K\in\reals^d\,,
    \end{align*}
    where the \textit{reference element} $K_\mathrm{ref}$ is the convex hull of the zero vector $\bm{0}\in\reals^d$ and the standard unit vectors $\ei{1},\dots,\ei{d}\in\reals^d$, i.e.,
  $      K_{\mathrm{ref}} = \conv\{\mathbf{0},\ei{1},\dots,\ei{d}\}\,$;}
    \item{there exist constants $C>c>0$ such that
    \begin{align*}
        \mathrm{spectrum}(\mathbb{A}_K^T\mathbb{A}_K) \subset [c,C] \qquad \text{for all $K\in{\displaystyle\bigcup_{h\,\in\,(0,H]}}\mathcal{T}_h$}\,;
    \end{align*}}
    \item{the intersection of two distinct elements $K_1,K_2$ of a mesh $\mathcal{T}_h$ is either empty, a common vertex, a common edge, or (in the case $d=3$) a common face;}
    \item{for all compact sets $\mathcal{K}\subset\Omega$ there exists a constant $h_0\in(0,H]$ such that
    \begin{align}
        \mathcal{K}\subset \Omega_h \qquad \text{for all $h\in(0,h_0)$. \label{comp_set_mesh}}
    \end{align}}
\end{itemize}
The symbol $\gh$ denotes the set of all faces ($d=3$) or all edges ($d=2$) in the mesh $\mathcal{T}_h$. Further, we define the sets
\begin{align*}
    \ghext = \big\{\sigma\in\gh\,\big|\,\sigma\subset\poh\big\} \qquad \text{and} \qquad \ghint = \gh\bs\ghext
\end{align*}
and, for $K\in\mathcal{T}_h$, the sets
\begin{align*}
    \mathcal{E}_h(K) = \big\{\sigma\in\mathcal{E}_h\,\big|\,\sigma\subset K\big\} \qquad \text{and} \qquad \mathcal{E}_{h,z}(K) = \big\{\sigma\in\mathcal{E}_{h,z}\,\big|\,\sigma\subset K\big\}\,,
\end{align*}
where $z\in\{\mathrm{int},\mathrm{ext}\}$. The elements of $\mathcal{E}_{h,\mathrm{int}}$, $\mathcal{E}_{h,\mathrm{int}}(K)$ and $\mathcal{E}_{h,\mathrm{ext}}$, $\mathcal{E}_{h,\mathrm{ext}}(K)$ are referred to as \textit{exterior} and \textit{interior faces} (\textit{edges}), respectively. In connection with these sets, we introduce the notation
\begin{align*}
    \int_{\mathcal{E}_{h,\mathrm{int}}} \equiv \sginth\intg \qquad \text{and} \qquad \int_{\mathcal{E}_{h}(K)} \equiv \sk\sgkh\intg\,.
\end{align*}
Moreover, we equip each $\sigma\in\gh$ with a unit vector $\ngamma$ by following the subsequent procedure:

We fix an arbitrary element $K_\sigma\in\mathcal{T}_h$ such that $\sigma\in \mathcal{E}_h(K_\sigma)$ and set $\ngamma = \fatn_{K_\sigma}(\fatx_\sigma)$. Here, $\fatx_\sigma$ denotes the center of mass of $\sigma$ and $\fatn_{K_\sigma}(\fatx_\sigma)$ is the outward-pointing unit normal vector to the element $K_\sigma$ at $\fatx_\sigma$.
Finally, it will be convenient to write $A \lesssim B$ whenever there is an $h$-independent constant $c>0$ such that $A \leq cB$ and $A \approx B$ whenever $A\lesssim B$ and $B\lesssim A$.

\subsection{Function spaces and projection operators}\label{sec_projection}
We proceed by defining the relevant discrete function spaces. The space of piecewise constant functions is denoted by
\begin{align*}
    Q_h &= \big\{v\in L^2(\Omega)\,\big|\,\text{$v|_{\Omega\backslash\overline{\Omega_h}}=0$ \; and \; $v|_K\in \Pzero(K)$ for all $K\in\mathcal{T}_h$}\big\}\,\footnotemark.
\end{align*}
\footnotetext{$P_n(K)$ denotes the set of all restrictions of polynomial functions $\reals^d\to\reals$ of degree at most $n$ to the set $K$.}For $v\in Q_h$ and $K\in\mathcal{T}_h$ we set $v_K = v(\fatx_K)$, where $\fatx_K$ denotes the center of mass of $K$. The projection $\piqh\equiv\overline{\vphantom{t}\;\cdot\;}:L^2(\Omega)\to Q_h$ associated with $Q_h$ is characterized by
\begin{align*}
    (\piqh v)\big|_K \equiv \overline{v}|_K \equiv \frac{1}{|K|}\intk v\;\dy \quad \text{for all $K\in\mathcal{T}_h$}\,.
\end{align*}
The Crouzeix-Raviart finite element spaces are denoted by
\begin{align*}
    V_h &= \left\{v\in L^2(\Omega)\left|\,\begin{array}{c}
        \text{$v|_{\Omega\backslash{\overline{\Omega_h}}}=0$, \; $v|_K\in \Pone(K)$ for all $K\in\mathcal{T}_h$, \; and} \\[2mm]
        {\displaystyle\intg\,\lim_{\delta\,\to\,0^+}}\big( v(\fatx-\delta\ngamma)-v(\fatx+\delta\ngamma)\big)\,\dsx = 0 \quad \text{for all $\sigma\in\ghint$}
    \end{array}\!\!\right.\right\}, \\[4mm]
    V_{0,h} &= \left\{v\in V_h\left|\;\intg\, \lim_{\delta\,\to\,0^+}v(\fatx-\delta\ngamma)\;\dsx = 0 \quad \text{for all $\sigma\in\ghext$}\right.\right\}\,.
\end{align*}
With these spaces we associate the projection $\pivh:W^{1,2}(\Omega)\to V_h$ that is determined by
\begin{align*}
    \intg \pivh v\;\dsx = \intg v\;\dsx \quad \text{for all $\sigma\in\gh$.}
\end{align*}
Additionally, we agree on the notation
\begin{gather*}
    Q_h^+ = \big\{v\in Q_h\,\big|\,v|_K > 0 \;\; \text{for all $K\in\mathcal{T}_h$}\big\}\,, \qquad Q_h^{0,+} = \big\{v\in Q_h\,\big|\,v|_K \geq 0 \;\; \text{for all $K\in\mathcal{T}_h$}\big\}\,, \notag \\[2mm]
    \bm{Q}_h = (Q_h)^d\,, \qquad \bm{V}_h = (V_h)^d\,, \qquad \text{and} \qquad \bm{V}_{0,h} = (V_{0,h})^d\,.
\end{gather*}


\subsection{Mesh-related operators}\label{sec_further_op}
Next, we define some mesh-related operators. We start by introducing the discrete counterparts of the differential operators $\gradx$ and $\divx$. They are determined by the stipulations
\begin{align*}
    {\arraycolsep=0pt
    \begin{array}{crll}
        & \qquad (\gradh \fatv)|_K &\,= \gradx (\fatv|_K) & \quad \text{for all $\fatv\in V_h\cup\bm{V}_h$ and all $K\in\mathcal{T}_h$} \\[4mm]
        \text{and} & \qquad \divh(\fatv)|_{K} &\,= \divx(\fatv|_K) & \quad \text{for all $\fatv\in\bm{V}_h$ and all $K\in\mathcal{T}_h$,}
    \end{array}}
\end{align*}
respectively. We continue by defining several trace operators. For arbitrary $\sigma\in\gh$, $\fatx\in\sigma$, and
\[\fatv\in (Q_h\cup\bm{Q}_h)\cup(V_h\cup\bm{V}_h)\cup(C(\overline{\Omega})\cup C(\overline{\Omega})^d)\]
we put
\begin{align*}
    \fatv^{\,\mathrm{in},\,\sigma}(\fatx) = \lim_{\delta\,\to\,0^+}\fatv(\fatx-\delta\ngamma)\,, \qquad \fatv^{\,\mathrm{out},\,\sigma}(\fatx) =\left\{\begin{array}{cl}
        {\displaystyle\lim_{\delta\,\to\,0^+}\fatv(\fatx+\delta\ngamma)} & \text{if $\sigma\in\ghint$,}  \\[2mm]
        \bm{0} & \text{else}
    \end{array}\right.\,.
\end{align*}
Further, we define
\begin{align*}
    \jumpG{\fatv} = \fatv^{\,\mathrm{out},\,\sigma}-\fatv^{\,\mathrm{in},\,\sigma}\,,  \quad \meanG{\fatv} = \dfrac{\fatv^{\,\mathrm{out},\,\sigma}+\fatv^{\,\mathrm{in},\,\sigma}}{2}\,, \quad \text{and} \quad \smallgmean{\fatv} = \frac{1}{|\sigma|}\intg \fatv^{\,\mathrm{in},\,\sigma}\;\dsx\,.
\end{align*}

The convective terms will be approximated by means of a dissipative upwind operator. For $\sigma\in\gh$, $\fatv\in\bm{V}_{0,h}$, and $\fatr\in Q_h\cup \bm{Q}_h$ we put
\begin{align*}
    \uppG{\fatr\!}{\fatv} &= \fatr^{\,\mathrm{out},\,\sigma}\left[\smallgmean{\fatv\cdot\ngamma}\right]^- + \fatr^{\,\mathrm{in},\,\sigma}\left[\smallgmean{\fatv\cdot\ngamma}\right]^+\,, \\[2mm]
    \upG{\fatr\!}{\fatv} &= \uppG{\fatr\!}{\fatv} -\hee\,\jumpG{\fatr}
    = \meanG{\bm{r}}\smallgmean{\fatv\cdot\ngamma}-\frac{1}{2}\,\jumpG{\bm{r}}\big(\he+|\smallgmean{\fatv\cdot\ngamma}|\big)\,, 
\end{align*}
where $\varepsilon>0$ is a given constant,
\begin{align*}
[x]^+ = \max\{x,0\}\,, \qquad \text{and} \qquad [x]^- = \min\{x,0\}\,.
\end{align*}

\begin{remark}
In the sequel, we tend to omit the letter $\sigma$ in the subscripts and superscripts of the operators defined in Sections \ref{sec_projection} and \ref{sec_further_op}. In some places, we also suppress the letter $h$ and the superscript \textit{in} in the notation if no confusion arises.
\end{remark}

\subsection{Time discretization}
In order to approximate the time derivatives, we apply the backward Euler method, i.e., the time derivative is represented by
\begin{align*}
    \Dt \fats^k_h = \frac{\fats^k_h-\fats^{k-1}_h}{\deltat}\,,
\end{align*}
where $\deltat>0$ is a given time step and $\fats^{k-1}_h$ and $\fats^k_h$ are the numerical solutions at the time levels $t_{k-1}=(k-1)\deltat$ and $t_k=k\deltat$, respectively. For the sake of simplicity, we assume that $\deltat$ is constant and that there is a number $N_T\in\naturals$ such that $N_T\deltat = T$.

\subsection{Numerical scheme}\label{sub_sec_scheme}
We are now ready to formulate our mixed finite element-finite volume (FE-FV) method.
\begin{definition}[\textbf{FE-FV method}]
A sequence $(\varrho_h^k,\thetakh,\fatu_h^k)_{k\,\in\,\naturals}\subset Q_h^+\times Q_h^+\times \bm{V}_{0,h}$ is a \textit{solution to our FE-FV method starting from the initial data} $(\varrho_h^0,\thetazeroh,\fatu_h^0)\in Q_h^+\times Q_h^+\times \bm{V}_h$ if the following equations hold for all $k\in\naturals$, $\phi_h\in Q_h$, and $\bm{\phi}_h\in\bm{V}_{0,h}$:
\begin{align}
    \intoh (\Dt\varrho^k_h)\,\phi_h\;\dx - \int_{\mathcal{E}_{\mathrm{int}}}\up{\varrho^k_h}{\fatu_h^k}\jump{\phi_h}\;\dsx &= 0\,, \label{r_method} \\[2mm]
    \intoh \Dt(\varrho^k_h\thetakh)\,\phi_h\;\dx - \int_{\mathcal{E}_{\mathrm{int}}}\up{\varrho^k_h\thetakh}{\fatu_h^k}\jump{\phi_h}\;\dsx &= 0\,, \label{z_method} \\[2mm]
    \intoh \Dt\!\left(\varrho_h^k\overline{\fatu_h^k}\,\right)\cdot\bm{\phi}_h\;\dx - \int_{\mathcal{E}_{\mathrm{int}}}\up{\varrho^k_h\overline{\fatu_h^k}}{\fatu_h^k}\cdot\Jump{\overline{\bm{\phi}_h}\vphantom{\overline{\overline{\phi}}}}\dsx + \mu\intoh \gradh\fatu_h^k:\gradh\bm{\phi}_h\;\dx & \phantom{.} \notag \\[2mm]
    {}+\nu\intoh \divh(\fatu_h^k)\,\divh(\bm{\phi}_h)\;\dx -\intoh \left(\p(\varrho^k_h\thetakh)+\hd\!\left[(\varrho^k_h)^2+(\varrho^k_h\thetakh)^2\right]\right)\divh(\bm{\phi}_h)\;\dx &= 0\,, \label{m_method}
\end{align}
where
\begin{align*}
    \delta>0 \qquad \text{and} \qquad \nu=\frac{d-2}{d}\mu+\lambda\geq 0\,.
\end{align*}
\end{definition}

\begin{remark}
We note that our FE-FV method is a generalization of the scheme presented in \citep[Chapter 13]{Feireisl_Lukacova_Book}. New ingredients are a modified upwind operator and the artificial pressure terms $\hd(\varrho_h^k)^2$, $\hd(\varrho^k_h\thetakh)^2$. The latter are added to ensure the consistency of our method for values of $\gamma$ close to $1$, see Sections \ref{sec_stability}, \ref{sec_consistency}.

\end{remark}

\subsubsection{Initial data}\label{subsubsec_data}
The initial data for the FE-FV method (\ref{r_method})--(\ref{m_method}) are given as
\begin{align}
    \varrho_h^0 = \piq\varrho_0\,, \qquad \thetazeroh = \piq \thetazero\,, \qquad \text{and} \qquad \fatu_h^0 = \piv\fatu_0\,. \label{disc_int_1}
\end{align}
As a consequence of this stipulation, we observe that $(\varrho_h^0,\thetazeroh,\fatu_h^0)\in Q_h^+\times Q_h^+\times\bm{V}_h$ and
\begin{align}
    c_\star < \thetazeroh < c^\star \quad \text{in $\Omega_h$.} \label{disc_int_2}
\end{align}

\subsubsection{Properties of the numerical method}
We proceed by summarizing several properties of the FE-FV method (\ref{r_method})--(\ref{m_method}).

\begin{lemma}\label{lem_sol}
Let $k\in\naturals$ and $(\varrho^{k-1}_h,\thetakmh,\fatu_h^{k-1})\in Q_h^+\times Q_h^+\times \bm{V}_h$ be given. \hypertarget{lem_sol_i}{}
\begin{enumerate}
    \item[$(\mathrm{i})$]{There exists a triplet $(\varrho^{k}_h,\thetakh,\fatu_h^{k})\in Q_h^+\times Q_h^+\times \bm{V}_{0,h}$ such that $\mathrm{(\ref{r_method})}$--$\mathrm{(\ref{m_method})}$ are satisfied for all $(\phi_h,\bm{\phi}_h)\in Q_h\times \bm{V}_{0,h}$. \hypertarget{lem_sol_ii}{}}
    \item[$(\mathrm{ii})$]{If $(\varrho^{k}_h,\thetakh,\fatu_h^{k})\in Q_h\times Q_h\times \bm{V}_{0,h}$ is a triplet of functions such that $\mathrm{(\ref{r_method})}$--$\mathrm{(\ref{m_method})}$ are satisfied for all $(\phi_h,\bm{\phi}_h)\in Q_h\times \bm{V}_{0,h}$, then $\varrho^{k}_h,\thetakh\in Q_h^+$ and
    \[\lnorm{\varrho_h^k}{1}{(\Omega_h)}{0.55}{-0.5mm}=\lnorm{\varrho_h^{k-1}}{1}{(\Omega_h)}{0.55}{-0.5mm} \quad \text{and} \quad \lnorm{\varrho^k_h\thetakh}{1}{(\Omega_h)}{0.55}{-0.5mm}=\lnorm{\varrho^{k-1}_h\thetakmh}{1}{(\Omega_h)}{0.55}{-0.5mm}\,.\]
    If, in addition, there are constants $0<\underline{c}<\overline{c}$ such that $\underline{c}<\thetakmh<\overline{c}$ in $\Omega_h$, then
    \[\underline{c}< \thetakh < \overline{c} \quad \text{in $\Omega_h$}\,.\]}
\end{enumerate}
\end{lemma}

\begin{proof}
For the proof we refer the reader to Appendix \ref{app_solv}.
\end{proof}

From Lemma~\ref{lem_sol} we easily deduce the following corollary.

\begin{corollary}\label{cor_sol}
Any solution $(\varrho_h^k,\thetakh,\fatu_h^k)_{k\,\in\,\naturals}$ to the FE-FV method $\mathrm{(\ref{r_method})}$--$\mathrm{(\ref{m_method})}$ starting from the discrete initial data $\mathrm{(\ref{disc_int_1})}$ has the following properties:
\hypertarget{cor_ii}{}
\begin{enumerate}
    \item[$(\mathrm{i})$]{For every $k\in\naturals$ it holds $c_\star < \thetakh < c^\star$ in $\Omega_h$. \hypertarget{cor_iii}{}}
    \item[$(\mathrm{ii})$]{It fulfills $\lnorm{\varrho_h^k}{1}{(\Omega_h)}{0.55}{-0.5mm}=\lnorm{\varrho_h^0}{1}{(\Omega_h)}{0.55}{-0.5mm}$ and $\lnorm{\varrho_h^k\thetakh}{1}{(\Omega_h)}{0.55}{-0.5mm}=\lnorm{\varrho_h^0\thetazeroh}{1}{(\Omega_h)}{0.55}{-0.5mm}$ for all $k\in\naturals$.}
\end{enumerate}
\end{corollary}


\section{Stability}\label{sec_stability}
We continue by discussing the stability of the FE-FV method (\ref{r_method})--(\ref{m_method}) that follows from a discrete energy balance. For its derivation, we rely on the concept of (\textit{discrete}) \textit{renormalization}. The same technique will be used to establish a discrete entropy inequality.

\subsection{Discrete renormalization}
In the sequel, we shall state renormalized versions of (\ref{r_method}) and (\ref{z_method}) that describe the evolution of $b(\varrho_h^k)$ and $b(\varrho_h^k\thetakh)$, $\varrho_h^k b(\thetakh)$, respectively, where $b\in \Ck{2}(0,\infty)$. Together with suitable choices for the function $b$, the first two renormalized equations will help us to handle the pressure terms when deriving the discrete energy balance. The last equation will be used to establish the discrete entropy inequality.

\begin{lemma}\label{lem_disc_renom}
Let $(\varrho_h^k,\thetakh,\fatu_h^k)_{k\,\in\,\naturals}$ be a solution to the FE-FV method $\mathrm{(\ref{r_method})}$--$\mathrm{(\ref{m_method})}$ starting from the discrete initial data $\mathrm{(\ref{disc_int_1})}$. Further, let $(r_h^k)_{k\,\in\,\naturals_0}\in\{(\varrho_h^k)_{k\,\in\,\naturals_0},(\varrho_h^k\thetakh)_{k\,\in\,\naturals_0}\}$.
Then for every $(b,k)\in\Ck{2}(0,\infty)\times\naturals$ \hypertarget{lem_disc_renom_i}{}
\begin{enumerate}
    \item[$\mathrm{(i)}$]{there exist values $(\xi^{(1)}_{r,b,k,\sigma})_{\sigma\,\in\,\gint},(\xi^{(2)}_{r,b,k,\sigma})_{\sigma\,\in\,\gint}\subset\reals$ satisfying 
    \begin{align*}
        \min\{(r^k_h)^{\mathrm{in},\hspace{0.3mm}\sigma}(\fatx_\sigma),(r^k_h)^{\mathrm{out},\hspace{0.3mm}\sigma}(\fatx_\sigma)\}\leq\xi^{(1)}_{r,b,k,\sigma},\xi^{(2)}_{r,b,k,\sigma}\leq\max\{(r^k_h)^{\mathrm{in},\hspace{0.3mm}\sigma}(\fatx_\sigma),(r^k_h)^{\mathrm{out},\hspace{0.3mm}\sigma}(\fatx_\sigma)\}
    \end{align*}
    and a function $\xi_{r,b,k}\in Q_h$ satisfying
    \begin{align*}
        \min\{r^{k-1}_h,r^k_h\}\leq\xi_{r,b,k}\leq\max\{r^{k-1}_h,r^k_h\}
    \end{align*}
    such that 
    \begin{align}
        0 &= \intoh \Dt \hspace{0.3mm} b(r^k_h)\;\dx + \intoh\big(\bp(r^k_h)r^k_h-b(r^k_h)\big)\,\divh(\fatu^k_h)\;\dx \notag \\[2mm]  
        &\phantom{\,=\,}+\frac{1}{2}\intoh \bpp(\xi_{r,b,k})\,\frac{(r^k_h-r^{k-1}_h)^2}{\deltat}\;\dx + \hee\int_{\mathcal{E}_{\mathrm{int}}} \jump{r^k_h}\,\jump{\bp(r^k_h)}\;\dsx \notag \\[2mm]
        &\phantom{\,=\,}+\frac{1}{2}\int_{\mathcal{E}_{\mathrm{int}}} \big(\bpp(\xi^{(1)}_{r,b,k,\sigma})\big[\omean{\fatu^k_h\cdot\ngamma}\big]^+ - \bpp(\xi^{(2)}_{r,b,k,\sigma})\big[\omean{\fatu_h^k\cdot\ngamma}\big]^-\big)\,\jump{r^k_h}^2\;\dsx\,; \hypertarget{lem_disc_renom_ii}{}
    \end{align}} 
    \item[$\mathrm{(ii)}$]{there exist values $(\zeta^{(1)}_{b,k,\sigma})_{\sigma\,\in\,\gint},(\zeta^{(2)}_{b,k,\sigma})_{\sigma\,\in\,\gint}\subset\reals$ satisfying
    \begin{align*}
        \min\{(\thetakh)^{\mathrm{in},\hspace{0.3mm}\sigma}(\fatx_\sigma),(\thetakh)^{\mathrm{out},\hspace{0.3mm}\sigma}(\fatx_\sigma)\}\leq\zeta^{(1)}_{b,k,\sigma},\zeta^{(2)}_{b,k,\sigma}\leq\max\{(\thetakh)^{\mathrm{in},\hspace{0.3mm}\sigma}(\fatx_\sigma),(\thetakh)^{\mathrm{out},\hspace{0.3mm}\sigma}(\fatx_\sigma)\}
    \end{align*}
    and a function $\zeta_{b,k}\in Q_h$ satisfying
    \begin{align*}
        \min\{\thetakmh,\thetakh\}\leq\zeta_{b,k}\leq\max\{\thetakmh,\thetakh\}
    \end{align*}
    such that 
    \begin{align}
        0 &= \intoh \Dt(\varrho_h^k b(\thetakh))\,\psi_h\;\dx - \inteint \upp{\varrho_h^k b(\thetakh)}{\fatu_h^k}\jump{\psi_h}\;\dsx \notag \\[2mm]
        & \phantom{\,=\,} + \frac{\deltat}{2}\intoh \bpp(\zeta_{b,k})\,\varrho_h^{k-1}\left(\frac{\thetakh-\thetakmh}{\deltat}\right)^2\psi_h\;\dx + \hee \inteint \jump{\varrho_h^k \thetakh}\,\jump{\bp(\thetakh)\,\psi_h}\;\dsx \notag \\[2mm]
        & \phantom{\,=\,} +  \frac{1}{2}\int_{\mathcal{E}_{\mathrm{int}}} \big(\bpp(\zeta^{(1)}_{b,k,\sigma})\varrho_h^k\psi_h^{\,\mathrm{out}}\big[\omean{\fatu^k_h\cdot\ngamma}\big]^+ - \bpp(\zeta^{(2)}_{b,k,\sigma})(\varrho_h^k)^{\mathrm{out}}\psi_h\big[\omean{\fatu_h^k\cdot\ngamma}\big]^-\big)\,\jump{\thetakh}^2\;\dsx \notag \\[2mm]
        & \phantom{\,=\,} + \hee \inteint \jump{\varrho_h^k}\Jump{\big(b(\thetakh)-\bp(\thetakh)\thetakh\big)\psi_h}\dsx 
    \end{align}
    for all $\psi_h\in Q_h$.}
\end{enumerate}

\end{lemma}

\begin{proof}
The proof of assertion (\hyperlink{lem_disc_renom_i}{i}) can be found in \citep[Lemma 5.1]{Karlsen_Karper_2}.
The main idea is to take $\phi_h=\bp(\varrho^k_h)\mathds{1}_{\overline{\Omega_h}}$ in (\ref{r_method}) and $\phi_h=\bp(\varrho^k_h\thetakh)\mathds{1}_{\overline{\Omega_h}}$ in (\ref{z_method}) and to rewrite the results by means of basic algebraic manipulations, Gauss's theorem, and Taylor expansions. 

Assertion (\hyperlink{lem_disc_renom_ii}{ii}) can be proven similarly; see, e.g., \citep[Lemma A.1 with $h^{1-\varepsilon}$ replaced by $\he/2$]{Navier_Stokes_Fourier}. Here, one chooses $\phi_h=\bp(\thetakh)\psi_h$ in (\ref{z_method}).
\end{proof}

\subsection{Discrete energy balance}
We now have all necessary tools at hand to establish the energy balance for our numerical method. 

\begin{lemma}
Let $(\varrho_h^k,\thetakh,\fatu_h^k)_{k\,\in\,\naturals}$ be a solution to the FE-FV method $\mathrm{(\ref{r_method})}$--$\mathrm{(\ref{m_method})}$ starting from the discrete initial data $\mathrm{(\ref{disc_int_1})}$ and $P$ the pressure potential introduced in $\mathrm{(\ref{press_pot})}$. Denoting the discrete energy at the time level $k\in\naturals_0$ by
\begin{align}
E_h^k\equiv E_h^k(\varrho_h^k,\thetakh,\fatu_h^k) = \intoh \left[\,\frac{1}{2}\,\varrho_h^k|\overline{\fatu_h^k}|^2+P(\varrho_h^k\thetakh)+\hd\big((\varrho_h^k)^2+(\varrho_h^k\thetakh)^2\big)\right]\dx\,,
\end{align}
we deduce that
\begin{align}
    &\Dt E_h^k + \intoh\left[\,\mu\hspace{0.3mm}|\gradh\fatu_h^k|^2+\nu\hspace{0.3mm}|\divh(\fatu_h^k)|^2\,\right]\dx \notag \\[2mm]
    & \quad = -\,\frac{1}{2}\intoh \Ppp(\xi_{\varrho\theta,P,k})\,\frac{(\varrho_h^k\thetakh-\varrho_h^{k-1}\thetakmh)^2}{\deltat}\;\dx -\frac{\he}{2}\inteint\,\jump{\varrho^k_h\thetakh}\,\jump{\Pp(\varrho^k_h\thetakh)}\;\dsx \notag \\[2mm]
    & \quad \phantom{\,=\,} -\frac{1}{2}\inteint \left(\Ppp(\xi^{(1)}_{\varrho\theta,P,k,\sigma})\big[\smallgmean{\fatu_h^k\cdot\ngamma}\big]^+ - \Ppp(\xi^{(2)}_{\varrho\theta,P,k,\sigma})\big[\smallgmean{\fatu_h^k\cdot\ngamma}\big]^-\right)\jump{\varrho^k_h\thetakh}^2\;\dsx \notag \\[2mm]
    & \quad \phantom{\,=\,} - \intoh\frac{\deltat}{2}\,\varrho^{k-1}_h\left|\frac{\overline{\fatu_h^k}-\overline{\fatu_h^{k-1}}}{\deltat}\right|^2\dx - \frac{\he}{2}\inteint\,\mean{\varrho^k_h}\Jump{\overline{\bm{\fatu}_h^k}\vphantom{\overline{\overline{\phi}}}}^2\dsx \notag \\[2mm]
    & \quad \phantom{\,=\,} -\frac{1}{2}\inteint \left((\varrho^k_h)^{\mathrm{in}}\big[\smallgmean{\fatu_h^k\cdot\ngamma}\big]^+ - (\varrho^k_h)^{\mathrm{out}}\big[\smallgmean{\fatu_h^k\cdot\ngamma}\big]^-\right)\Jump{\overline{\bm{\fatu}_h^k}\vphantom{\overline{\overline{\phi}}}}^2\dsx \notag \\[2mm]
    & \quad \phantom{\,=\,} -\hd\intoh \frac{(\varrho_h^k-\varrho_h^{k-1})^2}{\deltat}\;\dx -\hd\inteint\big(\he+|\smallgmean{\fatu_h^k\cdot\ngamma}|\big)\,\jump{\varrho^k_h}^2\;\dsx \notag \\[2mm]
    & \quad \phantom{\,=\,} -\hd\intoh \frac{(\varrho_h^k\thetakh-\varrho_h^{k-1}\thetakmh)^2}{\deltat}\;\dx -\hd\inteint\big(\he+|\smallgmean{\fatu_h^k\cdot\ngamma}|\big)\,\jump{\varrho^k_h\thetakh}^2\;\dsx\,, \label{disc_energy}
\end{align}
for all $k\in\naturals$, where $\xi_{\varrho\theta,P,k}\in Q_h$ and $(\xi^{(1)}_{\varrho\theta,P,k,\sigma})_{\sigma\,\in\,\gint},(\xi^{(2)}_{\varrho\theta,P,k,\sigma})_{\sigma\,\in\,\gint}\subset\reals$ are chosen as in Lemma~$\mathrm{\ref{lem_disc_renom}(\hyperlink{lem_sic_renom_i}{i})}$.
\end{lemma}

\begin{proof}
The proof can be done following the arguments in \citep[Chapter 7.5]{Feireisl_Karper_Pokorny}. Therefore, we depict only the most important steps. First, taking $\bm{\phi}_h=\fatu_h^k$ in (\ref{m_method}) yields
\begin{align}
    &\intoh \Dt\!\left(\varrho_h^k\overline{\fatu_h^k}\,\right)\cdot\fatu_h^k\;\dx - \int_{\mathcal{E}_{\mathrm{int}}}\up{\varrho^k_h\overline{\fatu_h^k}}{\fatu_h^k}\cdot\Jump{\overline{\bm{\fatu}_h^k}\vphantom{\overline{\overline{\phi}}}}\dsx \notag \\[2mm]
    &+\intoh\left[\,\mu\hspace{0.3mm}|\gradh\fatu_h^k|^2+\nu\hspace{0.3mm}|\divh(\fatu_h^k)|^2\,\right]\dx -\intoh \left(\p(\varrho^k_h\thetakh)+\hd\!\left[(\varrho^k_h)^2+(\varrho_h^k\thetakh)^2\right]\right)\divh(\fatu_h^k)\;\dx = 0\,. \label{disc_energy_1}
\end{align}
Next, we observe that
\begin{align}
    &\intoh \Dt\!\left(\varrho_h^k\overline{\fatu_h^k}\,\right)\cdot\fatu_h^k\;\dx = \frac{1}{2}\intoh\left( \Dt\!\left(\varrho_h^k|\overline{\fatu_h^k}|^2\right) + (\Dt\varrho_h^k)|\overline{\fatu_h^k}|^2 + \deltat\,\varrho^{k-1}_h\left|\frac{\overline{\fatu_h^k}-\overline{\fatu_h^{k-1}}}{\deltat}\right|^2\right)\dx\,. \label{disc_energy_2}
\end{align}
Then, we use $\phi_h=\frac{1}{2}|\overline{\fatu_h^k}|^2$ as a test function in (\ref{r_method}) to deduce that
\begin{align}
    \frac{1}{2}\intoh (\Dt\varrho_h^k)|\overline{\fatu_h^k}|^2\;\dx = \frac{1}{2}\int_{\mathcal{E}_{\mathrm{int}}}\up{\varrho^k_h}{\fatu_h^k}\Jump{|\overline{\fatu_h^k}|^2}\,\dsx\,. \label{disc_energy_3}
\end{align}
Moreover, by applying Lemma \ref{lem_disc_renom}(\hyperlink{lem_disc_renom_i}{i}) with $b=P$ and $b=\hd(\cdot)^2$, we obtain
\begin{align}
    &\intoh \p(\varrho_h^k\thetakh)\,\divh(\fatu_h^k)\;\dx \notag \\[2mm]
    &= -\intoh\Dt P(\varrho_h^k\thetakh)\;\dx -\frac{1}{2}\intoh \Ppp(\xi_{\varrho\theta,P,k})\,\frac{(\varrho_h^k\thetakh-\varrho_h^{k-1}\thetakmh)^2}{\deltat}\;\dx -\frac{\he}{2}\inteint\,\jump{\varrho^k_h\thetakh}\,\jump{\Pp(\varrho^k_h\thetakh)}\;\dsx \notag \\[2mm]
    &\phantom{\,=\,} -\frac{1}{2}\inteint \left(\Ppp(\xi^{(1)}_{\varrho\theta,P,k,\sigma})\big[\smallgmean{\fatu_h^k\cdot\ngamma}\big]^+ - \Ppp(\xi^{(2)}_{\varrho\theta,P,k,\sigma})\big[\smallgmean{\fatu_h^k\cdot\ngamma}\big]^-\right)\jump{\varrho^k_h\thetakh}^2\;\dsx\,, \label{disc_energy_4} \\[4mm]
    &\hd\intoh (\varrho^k_h)^2\,\divh(\fatu_h^k)\;\dx \notag \\[2mm]
    &= -\hd\left[\intoh\Dt(\varrho^k_h)^2 \;\dx + \intoh \frac{(\varrho_h^k-\varrho_h^{k-1})^2}{\deltat}\;\dx +\inteint\big(\he+|\smallgmean{\fatu_h^k\cdot\ngamma}|\big)\,\jump{\varrho^k_h}^2\;\dsx\right], \label{disc_energy_5} \\[4mm]
    &\hd\intoh (\varrho^k_h\thetakh)^2\,\divh(\fatu_h^k)\;\dx \notag \\[2mm]
    &= -\hd\left[\intoh\Dt(\varrho^k_h\thetakh)^2 \;\dx +\intoh \frac{(\varrho_h^k\thetakh-\varrho_h^{k-1}\thetakmh)^2}{\deltat}\;\dx +\inteint\big(\he+|\smallgmean{\fatu_h^k\cdot\ngamma}|\big)\,\jump{\varrho^k_h\thetakh}^2\;\dsx\right]. \label{disc_energy_6}
\end{align}
Plugging (\ref{disc_energy_2})--(\ref{disc_energy_6}) into (\ref{disc_energy_1}), we see that we have almost arrived at (\ref{disc_energy}). Indeed, it only remains to show that
\begin{align*}
    &\frac{1}{2}\int_{\mathcal{E}_{\mathrm{int}}}\up{\varrho^k_h}{\fatu_h^k}\Jump{|\overline{\fatu_h^k}|^2}\dsx-\int_{\mathcal{E}_{\mathrm{int}}}\up{\varrho^k_h\overline{\fatu_h^k}}{\fatu_h^k}\cdot\Jump{\overline{\bm{\fatu}_h^k}\vphantom{\overline{\overline{\phi}}}}\dsx \\[2mm]
    & \quad = \frac{\he}{2}\inteint\,\mean{\varrho^k_h}\Jump{\overline{\bm{\fatu}_h^k}\vphantom{\overline{\overline{\phi}}}}^2\dsx + \frac{1}{2}\inteint \left((\varrho^k_h)^{\mathrm{in}}\big[\smallgmean{\fatu_h^k\cdot\ngamma}\big]^+ - (\varrho^k_h)^{\mathrm{out}}\big[\smallgmean{\fatu_h^k\cdot\ngamma}\big]^-\right)\Jump{\overline{\bm{\fatu}_h^k}\vphantom{\overline{\overline{\phi}}}}^2\dsx\,,
\end{align*}
which follows by direct calculations. This completes the proof. 
\end{proof}

\subsection{Time\hspace{0.3mm}-dependent numerical solutions and energy estimates}
Next, we formulate appropriate stability estimates for the time-dependent numerical solutions introduced below.

\begin{definition}
Let $(\varrho_h^k,\thetakh,\fatu_h^k)_{k\,\in\,\naturals}$ be a solution to the FE-FV method (\ref{r_method})--(\ref{m_method}) starting from the initial data $(\varrho_h^0,\thetazeroh,\fatu_h^0)$. We define the functions 
\begin{align*}
    \varrho_h,\thetah:\reals\times\Omega\to\reals\,, \quad \fatu_h:\reals\times\Omega\to\reals^d\,,
\end{align*}
that are piecewise constant in time by setting
\begin{align*}
    (\varrho_h,\thetah,\fatu_h)(t,\cdot) &= \left\{{\arraycolsep=0pt 
    \begin{array}{ll}
        \;(\varrho_h^k,\thetakh,\fatu_h^k) & \quad \text{if $t\in((k-1)\deltat,k\deltat]$ for some $k\in\naturals$ and} \\[2mm]
        \;(\varrho_h^0,\thetazeroh,\fatu_h^0) & \quad \text{if $t\leq 0$,}
    \end{array}}
    \right.
\end{align*}
\end{definition}


The most important stability estimates that can be obtained from the discrete energy balance (\ref{disc_energy}) read as follows.

\begin{corollary}[\textbf{Stability estimates}]\label{cor_stab}
Any solution $(\varrho_h,\thetah,\fatu_h)$ to the FE-FV method $\mathrm{(\ref{r_method})}$--$\mathrm{(\ref{m_method})}$ starting from the initial data $\mathrm{(\ref{disc_int_1})}$ has the following properties:
\begin{gather}
    \lnorm{\varrho_h|\overline{\fatu_h}|^2}{\infty}{(0,T;\Lp{1}{0.0mm}(\Omega_h))}{0.55}{0.0mm} \lesssim 1\,, \quad \lnorm{\varrho_h}{\infty}{(0,T;\Lp{\gamma}{0.5mm}(\Omega_h))}{0.55}{0.0mm} \lesssim 1\,, \quad \lnorm{\varrho_h\overline{\fatu_h}}{\infty}{(0,T;\Lp{2\gamma/(\gamma+1)}{0.0mm}(\Omega_h)^d)}{0.55}{0.0mm} \lesssim 1\,, \label{energy_est_1} \\[2mm]
    \lnorm{\gradh\fatu_h}{2}{(0,T;\Lp{2}{0.0mm}(\Omega_h)^{d\times d})}{0.55}{-0.5mm} \lesssim 1\,, \quad \lnorm{\divh(\fatu_h)}{2}{(0,T;\Lp{2}{0.0mm}(\Omega_h))}{0.55}{-0.5mm} \lesssim 1\,, \quad \lnorm{\fatu_h}{2}{(0,T;\Lp{q}{0.5mm}(\Omega_h)^d)}{0.55}{-0.5mm} \lesssim 1\,, \label{energy_est_2} \\[2mm]
    \lnorm{\varrho_h\thetah}{\infty}{(0,T;\Lp{\gamma}{0.5mm}(\Omega_h))}{0.55}{0.0mm} \lesssim 1\,, \quad  \lnorm{h^{\delta/2}\varrho_h}{\infty}{(0,T;\Lp{2}{0mm}(\Omega_h))}{0.55}{0.0mm} \lesssim 1\,, \quad \lnorm{h^{\delta/2}\varrho_h\thetah}{\infty}{(0,T;\Lp{2}{0mm}(\Omega_h))}{0.55}{0.0mm} \lesssim 1\,, \label{energy_est_3} \\[2mm]
    \lnorm{\varrho_h\thetah\overline{\fatu_h}}{\infty}{(0,T;\Lp{2\gamma/(\gamma+1)}{0.0mm}(\Omega_h)^d)}{0.55}{0.0mm} \lesssim 1\,, \quad \lnorm{\varrho_h\overline{\fatu_h}}{2}{(0,T;\Lp{2}{0.0mm}(\Omega_h)^d)}{0.55}{-0.5mm} \lesssim h^{-\frac{d+3\delta}{6}}\,, \label{energy_est_4} \\[2mm]
    \hd\intt\!\!\inteint\max\left\{\he,|\smallgmean{\fatu_h\cdot\ngamma}|\right\}\jump{\varrho_h}^2\;\dsxdt \lesssim 1\,, \label{energy_est_5} \\[2mm]
    \hd\intt\!\!\inteint\max\left\{\he,|\smallgmean{\fatu_h\cdot\ngamma}|\right\}\jump{\varrho_h\thetah}^2\;\dsxdt \lesssim 1\,, \label{energy_est_6} \\[2mm]
    \frac{\he}{2}\intt\!\!\inteint\,\mean{\varrho_h}\jump{\overline{\bm{\fatu}_h}}^2\;\dsxdt \lesssim 1\,, \label{energy_est_7} \\[2mm]
    \frac{1}{2}\intt\!\!\inteint \left(\varrho_h^{\,\mathrm{in}}\big[\smallgmean{\fatu_h\cdot\ngamma}\big]^+ - \varrho_h^{\,\mathrm{out}}\big[\smallgmean{\fatu_h\cdot\ngamma}\big]^-\right)\jump{\overline{\bm{\fatu}_h}}^2\;\dsxdt \lesssim 1\,, \label{energy_est_8} \\[2mm]
    \intt\!\!\inteint \big|\jump{\varrho_h}\smallgmean{\fatu_h\cdot\ngamma}\big|\;\dsxdt \lesssim h^{-\delta/2}(1+h^{-1/2})\,, \label{energy_est_9} \\[2mm]
    \intt\!\!\inteint \big|\jump{\varrho_h\thetah}\smallgmean{\fatu_h\cdot\ngamma}\big|\;\dsxdt \lesssim h^{-\delta/2}(1+h^{-1/2})\,, \label{energy_est_10}
\end{gather}
where $q\in[1,\infty)$ if $d=2$ and $q\in[1,6]$ if $d=3$.
\end{corollary}

\begin{proof}
The proof is provided in Appendix \ref{app_stab}.
\end{proof}

\subsection{Discrete entropy inequality}
We conclude this section by stating a discrete entropy inequality. It is obtained by taking $b=\chi$ in Lemma \ref{lem_disc_renom}(\hyperlink{lem_disc_renom_ii}{ii}).

\begin{lemma}\label{lem_disc_entropy_inequ}
Let $(\varrho_h^k,\thetakh,\fatu_h^k)_{k\,\in\,\naturals}$ be a solution to the FE-FV method $\mathrm{(\ref{r_method})}$--$\mathrm{(\ref{m_method})}$ starting from the discrete initial data $\mathrm{(\ref{disc_int_1})}$ and $\chi\in\Ck{2}(0,\infty)$ a concave function. Then
\begin{align}
    0 &\leq \intoh \Dt(\varrho_h^k \chi(\thetakh))\,\psi_h\;\dx - \inteint \upp{\varrho_h^k \chi(\thetakh)}{\fatu_h^k}\jump{\psi_h}\;\dsx \notag \\[2mm]
    & \phantom{\,=\,} + \hee \inteint \jump{\varrho_h^k \thetakh}\,\jump{\chip(\thetakh)\psi_h}\;\dsx + \, \hee \inteint \jump{\varrho_h^k}\Jump{\big(\chi(\thetakh)-\chip(\thetakh)\thetakh\big)\psi_h}\dsx \label{disc_entropy_inequ}
\end{align}
for all $\psi_h\in Q_h^{0,+}$.
\end{lemma}

\section{Consistency}\label{sec_consistency}
The goal of this section is to establish the consistency of the FE-FV method (\ref{r_method})--(\ref{m_method}).

\begin{theorem}[\textbf{Consistency of the FE-FV method}]\label{thm_cons}
Let $d\in\{2,3\}$. Further, suppose $(\varrho_h,\thetah,\fatu_h)_{h\,\in\,(0,H]}$ is a family of solutions to the FE-FV method $\mathrm{(\ref{r_method})}$--$\mathrm{(\ref{m_method})}$ with
\begin{align}
    \gamma > 1\,, \qquad \deltat \approx h\,, \qquad \varepsilon > 1, \qquad \text{and} \qquad 0<\delta<\tfrac{1}{2} \label{cons_cond}
\end{align}
starting from the initial data $(\varrho_h^0,\thetazeroh,\fatu_h^0)_{h\,\in\,(0,H]}$ defined in $\mathrm{(\ref{disc_int_1})}$. Then, for $\beta=\min\left\{\varepsilon-1,\tfrac{1-2\delta}{4}\right\}$,
\begin{align}
    -\into \varrho_h^0\,\varphi(0,\cdot)\;\dx = \inttinto \big[\varrho_h\pt\varphi+\varrho_h\overline{\fatu_h}\cdot\gradx\varphi\big]\,\dxdt + \mathcal{O}(\hb) \label{cons_r_method}
\end{align}
for all $\varphi\in \Cinftyc([0,T)\times\overline{\Omega})$ as $h\downarrow 0$,
\begin{align}
    -\into \varrho_h^0\thetazeroh\,\varphi(0,\cdot)\;\dx = \inttinto \big[\varrho_h\thetah\pt\varphi+\varrho_h\thetah\overline{\fatu_h}\cdot\gradx\varphi\big]\,\dxdt + \mathcal{O}(\hb) \label{cons_z_method}
\end{align}
for all $\varphi\in \Cinftyc([0,T)\times\overline{\Omega})$ as $h\downarrow 0$,
\begin{align}
    &-\into \varrho_h^0\overline{\fatu_h^0}\cdot\bm{\varphi}(0,\cdot)\;\dx + \inttinto\big[\mu\gradh\fatu_h:\gradx\bm{\varphi}+\nu\,\divh(\fatu_h)\,\divx(\bm{\varphi})\big]\,\dxdt + \mathcal{O}(\hb) \notag \\[2mm]
    & \qquad = \inttinto \big[\varrho_h\overline{\fatu_h}\cdot\pt\bm{\varphi}+\varrho_h\overline{\fatu_h}\otimes\overline{\fatu_h}:\gradx\bm{\varphi}+\left(\p(\varrho_h\thetah)+\hd\big[\varrho_h^2+(\varrho_h\thetah)^2\big]\right)\divx(\bm{\varphi})\big]\,\dxdt \label{cons_m_method}
\end{align}
for all $\bm{\varphi}\in \Cinftyc([0,T)\times\Omega)^d$ as $h\downarrow 0$, and
\begin{align}
    -\into \varrho_h^0\ln(\thetazeroh)\psi(0,\cdot)\;\dx \geq \inttinto\big[\varrho_h\ln(\thetah)\pt\psi+\varrho_h\ln(\thetah)\,\overline{\fatu_h}\cdot\gradx\psi\big]\,\dxdt + \mathcal{O}(h^\beta) \label{cons_entropy_inequ}
\end{align}
for all $\psi\in \Cinftyc([0,T)\times\overline{\Omega})$, $\psi\geq 0$, as $h\downarrow 0$.
\end{theorem}

The structure of the proof of Theorem \ref{thm_cons} is essentially the same as that of \citep[Theorem 13.2]{Feireisl_Lukacova_Book}. In particular, we will use similar tools. Apart from the estimates listed in Appendix \ref{subsec_mesh_est}, we will need the following results.

\begin{lemma}\label{lem_time_der}
Let $\phi\in\Ckc{2}([0,T)\times\overline{\Omega})$, $(r_h^k)_{k\,\in\,\naturals_0}\subset Q_h$, and define $r_h:\reals\times\Omega\to\reals$ via
\[r_h(t,\cdot) = \left\{{\arraycolsep=0pt
    \begin{array}{ll}
        \;r_h^k & \quad \text{if $t\in((k-1)\deltat,k\deltat]$ for some $k\in\naturals$ and} \\[2mm]
        \;r_h^0 & \quad \text{if $t\leq 0$.}
    \end{array}}
    \right.\]
Then the subsequent relations hold:
\begin{align}
    &\left|\,\inttintoh \big[(\Dt r_h)\,\piq\phi + r_h\,\pt\phi\big]\,\dxdt + \into r_h^0\,\phi(0,\cdot)\;\dx\,\right| \notag \\[2mm]
    & \qquad \lesssim \deltat\,\maxnormk{\phi}{2}{(\overline{\Omega_T})}\left(\lnorm{r_h}{1}{(0,T;\Lp{1}{0mm}(\Omega_h))}{0.55}{-0.5mm} + \lnorm{r_h^0}{1}{(\Omega_h)}{0.55}{-0.5mm}\right),\\[4mm]
    &\left|\,\inttintoh \big[(\Dt r_h)\,\piv\phi + r_h\,\pt\phi\big]\,\dxdt + \into r_h^0\,\phi(0,\cdot)\;\dx\,\right| \notag \\[2mm]
    & \qquad \lesssim (\deltat+h)\,\left(\maxnormk{\phi}{2}{(\overline{\Omega_T})}\lnorm{r_h}{1}{(0,T;\Lp{1}{0mm}(\Omega_h))}{0.55}{-0.5mm} + \maxnormk{\phi}{1}{(\overline{\Omega_T})} \lnorm{r_h^0}{1}{(\Omega_h)}{0.55}{-0.5mm}\right).
\end{align}
\end{lemma}

\begin{lemma}\label{lem_upwind}
Let $r,f\in Q_h$, $\fatv\in\bm{V}_{0,h}$, and $\phi\in\Ck{1}(\overline{\Omega_h})$. Then
\begin{align}
    &\intoh r\fatv\cdot\gradx\phi\;\dx - \inteint \up{r}{\fatv}\jump{f}\;\dsx \notag \\[2mm]
    & \qquad = \intoh r\,(f-\phi)\,\divh(\fatv)\;\dx + \intek (f-\phi)\,\jump{r}\,\big[\smallgmean{\fatv\cdot\nk}\big]^-\;\dsx \notag \\[2mm]
    & \qquad \phantom{\,=\,} + \intek \big(\phi-\smallgmean{\phi}\big)\,r\,\big(\fatv\cdot\nk-\smallgmean{\fatv\cdot\nk}\big)\,\dsx + \frac{\he}{2}\inteint \,\jump{r}\,\jump{f}\;\dsx\,.\,\footnotemark
\end{align}
\end{lemma}
\footnotetext{In integrals of the form $\intek$ we consider the the vector $\ngamma$ in the definition of the trace operators $(\cdot)^{\mathrm{in},\sigma}$ and $(\cdot)^{\mathrm{out},\sigma}$ to be replaced by $\fatn_K$.}

\begin{corollary}\label{cor_upwind}
Let $\fats,\fatg\in \bm{Q}_h$, $\fatw\in\bm{V}_{0,h}$, and $\bm{\psi}\in\Ck{1}(\overline{\Omega_h})^d$. Then
\begin{align}
    &\intoh \fats\otimes\fatw:\gradx\bm{\psi}\;\dx - \inteint \up{\fats}{\fatw}\cdot\jump{\fatg}\;\dsx \notag \\[2mm]
    & \qquad = \intoh \fats\cdot(\fatg-\bm{\psi})\,\divh(\fatw)\;\dx + \intek (\fatg-\bm{\psi})\cdot\jump{\fats}\,\big[\smallgmean{\fatw\cdot\nk}\big]^-\;\dsx \notag \\[2mm]
    & \qquad \phantom{\,=\,} + \intek \big(\bm{\psi}-\smallgmean{\bm{\psi}}\big)\cdot\fats\,\big(\fatw\cdot\nk-\smallgmean{\fatw\cdot\nk}\big)\,\dsx + \frac{\he}{2}\inteint \,\jump{\fats}\cdot\jump{\fatg}\;\dsx\,.
\end{align}
\end{corollary}

\begin{remark}
The formula in Lemma \ref{lem_upwind} also holds true when the dissipative upwind term is replaced by the usual upwind term and the last term on the right-hand side of the identity is canceled. The same applies to Corollary \ref{cor_upwind}.
\end{remark}

\begin{lemma}\label{lem_press}
Let $r\in Q_h$, $v\in V_{0,h}$, $\phi\in \Ck{1}_0(\Omega_h)$, and $\bm{\phi}\in \Ck{1}_0(\Omega_h)^d$. Then
\begin{align}
    \intoh \gradh v \cdot \gradh\piv\phi\;\dx &= \intoh \gradh v\cdot\gradx\phi\;\dx \\[2mm]
    \text{and} \qquad \intoh r\,\divh(\piv\bm{\phi})\;\dx &= \intoh r\,\divx(\bm{\phi})\;\dx\,.
\end{align}
\end{lemma}

For the proof of the Lemmata \ref{lem_time_der}, \ref{lem_upwind}, and \ref{lem_press}, we refer to \citep[Preliminaries, Lemma 8]{Feireisl_Lukacova_Book}, \citep[Chapter 9.2, Lemma 7 with $\chi=1$]{Feireisl_Karper_Pokorny}, and \citep[Chapter 9.3, Lemma 8]{Feireisl_Karper_Pokorny}, respectively. For the proof of Lemma \ref{lem_upwind}, we additionally need to observe that
\begin{align*}
    \intg \smallgmean{\phi}\,r\,\big(\fatv\cdot\nk-\smallgmean{\fatv\cdot\nk}\big)\,\dsx = 0\,,
\end{align*}
which follows from the fact that $r\in Q_h$. Corollary \ref{cor_upwind} can be proven by applying Lemma \ref{lem_upwind} with $(r,f,\fatv,\phi)=(s_i,g_i,\fatw,\psi_i)$, $i\in\{1,\dots,d\}$.

Having all necessary tools at our disposal, we can approach the proof of Theorem \ref{thm_cons}.

\begin{proof}[Proof of Theorem $\mathbf{\ref{thm_cons}}$]
Let $\varphi,\psi\in \Cinftyc([0,T)\times\overline{\Omega})$, $\psi\geq 0$, and $\bm{\varphi}\in \Cinftyc([0,T)\times\Omega)^d$ be arbitrary test functions. We set $\varphi_h=\piq\varphi$, $\psi_h=\piq\psi$, $\bm{\varphi}_h = \piv\bm{\varphi}$ and make the following introductory observations:

\begin{itemize}
    \item{Due to the construction of the family $(\Omega_h)_{h\,\in\,(0,H]}$, we have $\bm{\varphi}\in \Cinftyc([0,T)\times\Omega_h)^d$, provided $h\in(0,H]$ is sufficiently small (cf. (\ref{comp_set_mesh})), which we henceforth assume.}
    \item{Recall that the elements of $Q_h$ and $V_h$ vanish outside $\Omega_h$. This allows us to replace $\Omega_h$ by $\Omega$ when appropriate.}
\end{itemize}

\noindent \textbf{The continuity equation.} \\[2mm]
From (\ref{r_method}) we deduce that
\begin{align}
    \inttinto (\Dt\varrho_h)\,\varphi_h\;\dxdt - \intt\!\!\inteint \up{\varrho_h}{\fatu_h}\jump{\varphi_h}\;\dsxdt = 0\,. \label{int_r_method}
\end{align}
Applying the first estimate in Lemma \ref{lem_time_der} with $(r_h,\phi)=(\varrho_h,\varphi)$ as well as (\ref{bound_piq}), the second estimate in (\ref{energy_est_1}), and the fact that $\deltat\approx h$, we obtain
\begin{align*}
    \inttinto (\Dt\varrho_h)\,\varphi_h\;\dxdt = -\into \varrho_h^0\,\varphi(0,\cdot)\;\dx - \inttinto \varrho_h\,\pt\varphi\;\dxdt - I_{1,h}\,,
\end{align*}
where
\begin{align*}
    |I_{1,h}| &\lesssim \deltat\left(\maxnormk{\varphi}{2}{(\overline{\Omega_T})}\,\lnorm{\varrho_h}{1}{(0,T;\Lp{1}{0mm}(\Omega))}{0.55}{-0.5mm} + \maxnormk{\varphi}{1}{(\overline{\Omega_T})}\,\lnorm{\varrho_h^0}{1}{(\Omega)}{0.55}{-0.5mm}\right) \notag \\[2mm]
    &\lesssim h\left(\lnorm{\varrho_h}{\infty}{(0,T;\Lp{\gamma}{0.5mm}(\Omega))}{0.55}{0mm} + \lnorm{\varrho_0}{1}{(\Omega)}{0.55}{-0.5mm}\right) \lesssim h\,.
\end{align*}
Next, let us consider the second term on the left-hand side of (\ref{int_r_method}). Using Lemma \ref{lem_upwind} with $(r,\fatv,f,\phi)=(\varrho_h,\fatu_h,\varphi_h,\varphi)(t,\cdot)$, $t\in[0,T]$, as well as the estimates (\ref{est_piqh_1})--(\ref{est_piqh_3}) and (\ref{est_piqh_4}), we deduce that
\begin{align*}
    \intt\!\!\inteint \up{\varrho_h}{\fatu_h}\,\jump{\varphi_h}\;\dsxdt = \inttinto \varrho_h\fatu_h\cdot\gradx\varphi\;\dxdt + \sum_{j\,=\,2}^5 I_{j,h}\,,
\end{align*}
where \hypertarget{first_term}{}
\begin{gather*}
    |I_{2,h}| \lesssim h \intt\!\!\intek \big|\jump{\varrho_h}\big[\smallgmean{\fatu_h\cdot\nk}\big]^-\big|\,\dsxdt\,, \qquad 
    |I_{3,h}| \lesssim h \intt\!\!\intek \left|\,\varrho_h\big(\fatu_h-\smallgmean{\fatu_h}\big)\right|\dsxdt\,, \\[2mm]
    |I_{4,h}| \lesssim h \inttinto \big|\varrho_h\,\divh(\fatu_h)\big|\,\dxdt\,, \qquad |I_{5,h}| \lesssim h^{1+\varepsilon} \intt\!\!\inteint \big|\jump{\varrho_h}\big|\,\dsxdt\,.
\end{gather*}
These terms can be further estimated as follows.
\begin{itemize}
    \item{\textbf{Term} $|I_{2,h}|$. Due to (\ref{energy_est_9}), we obtain
    \begin{align*}
        |I_{2,h}|\lesssim h \intt\!\!\inteint \big|\jump{\varrho_h}\,\smallgmean{\fatu_h\cdot\ngamma}\big|\,\dsxdt \lesssim h^{1-\delta/2}(1+h^{-1/2})\,.
    \end{align*}}
    \item{\textbf{Term} $|I_{3,h}|$. By means of Hölder's inequality, the second estimate in (\ref{est_vh_1}), the first estimate in (\ref{trace_neg_lp}), the second estimate in (\ref{energy_est_3}), and the first estimate in (\ref{energy_est_2}), we derive
    \begin{align*}
        |I_{3,h}|&\lesssim h\,\lnorm{\varrho_h}{2}{(0,T;\Lp{2}{0mm}(\Omega))}{0.55}{-0.5mm}\,\lnorm{\gradh\fatu_h}{2}{(0,T;\Lp{2}{0mm}(\Omega)^{d\times d})}{0.55}{-0.5mm} \notag \\[2mm]
        &\lesssim h^{1-\delta/2}\,\lnorm{h^{\delta/2}\varrho_h}{\infty}{(0,T;\Lp{2}{0mm}(\Omega))}{0.55}{-0mm}\,\lnorm{\gradh\fatu_h}{2}{(0,T;\Lp{2}{0mm}(\Omega)^{d\times d})}{0.55}{-0.5mm} \lesssim h^{1-\delta/2}.
    \end{align*}}
    \item{\textbf{Term} $|I_{4,h}|$. Employing Hölder's inequality, the second estimate in (\ref{energy_est_2}), and the second estimate in (\ref{energy_est_3}), we conclude that
    \begin{align*}
        |I_{4,h}| &\lesssim h\,\lnorm{\varrho_h}{2}{(0,T;\Lp{2}{0mm}(\Omega))}{0.55}{-0.5mm}\,\lnorm{\divh(\fatu_h)}{2}{(0,T;\Lp{2}{0mm}(\Omega))}{0.55}{-0.5mm} \notag \\[2mm]
        &\lesssim h^{1-\delta/2}\,\lnorm{h^{\delta/2}\varrho_h}{\infty}{(0,T;\Lp{2}{0mm}(\Omega))}{0.55}{-0mm}\,\lnorm{\divh(\fatu_h)}{2}{(0,T;\Lp{2}{0mm}(\Omega))}{0.55}{-0.5mm} \lesssim h^{1-\delta/2}\,.
    \end{align*}}
    \item{\textbf{Term} $|I_{5,h}|$.
    Applying the first estimate in (\ref{trace_neg_lp}) and the second estimate in (\ref{energy_est_1}), we get
    \begin{align*}
        |I_{5,h}|\lesssim
        \he\, \lnorm{\varrho_h}{1}{(0,T;\Lp{1}{0mm}(\Omega))}{0.55}{-0.5mm}\lesssim \he\, \lnorm{\varrho_h}{\infty}{(0,T;\Lp{\gamma}{0.5mm}(\Omega))}{0.55}{0mm} \lesssim \he\,.
    \end{align*}}
\end{itemize}
Consequently,
\begin{align}
    -\into \varrho_h^0\,\varphi(0,\cdot)\;\dx = \inttinto \big[\varrho_h\pt\varphi+\varrho_h\fatu_h\cdot\gradx\varphi\big]\,\dxdt + \mathcal{O}(h^{\alpha_1}) \label{proof_cons_r_method_1}
\end{align}
with $\alpha_1 = \min\left\{\varepsilon,\frac{1-\delta}{2}\right\}>0$ as $h\downarrow 0$. Next, using Hölder's inequality, the first estimate in (\ref{est_vh_1}), the second estimate in (\ref{energy_est_3}), and the first estimate in (\ref{energy_est_2}), we see that
\begin{align}
    &\left|\,\inttinto \varrho_h\big(\fatu_h-\overline{\fatu_h}\big)\cdot\gradx\varphi\;\dxdt\,\right|
    \lesssim \lnorm{\varrho_h}{2}{(0,T;\Lp{2}{0mm}(\Omega))}{0.55}{-0.5mm}\,\lnorm{\fatu_h-\overline{\fatu_h}}{2}{(0,T;\Lp{2}{0mm}(\Omega)^d)}{0.55}{-0.5mm} \notag \\[2mm]
    & \qquad \lesssim h^{1-\delta/2}\,\lnorm{h^{\delta/2}\varrho_h}{\infty}{(0,T;\Lp{2}{0mm}(\Omega))}{0.55}{-0mm}\,\lnorm{\gradh\fatu_h}{2}{(0,T;\Lp{2}{0mm}(\Omega)^{d\times d})}{0.55}{-0.5mm} \lesssim h^{1-\delta/2}\,. \label{last_step_cons_r_method}
\end{align}
Therefore, we may rewrite (\ref{proof_cons_r_method_1}) as
\begin{align*}
    -\into \varrho_h^0\,\varphi(0,\cdot)\;\dx = \inttinto \big[\varrho_h\pt\varphi+\varrho_h\overline{\fatu_h}\cdot\gradx\varphi\big]\,\dxdt + \mathcal{O}(h^{\alpha_1}) \qquad \text{as $h\downarrow 0$.}
\end{align*}

\noindent \textbf{The potential temperature equation.} \\[2mm]
The proof of (\ref{cons_z_method}) can be done by repeating the proof of (\ref{cons_r_method}) with $\varrho_h$ and $\varrho_h^0$ replaced by $\varrho_h\thetah$ and $\varrho_h^0\thetazeroh$, respectively. \\[4mm]

\noindent \textbf{The momentum equation.} \\[2mm]
From (\ref{m_method}) we deduce that
\begin{align}
    &\inttinto \Dt(\varrho_h\overline{\fatu_h})\cdot\bm{\varphi}_h\;\dxdt - \intt\!\!\int_{\mathcal{E}_{\mathrm{int}}}\up{\varrho_h\overline{\fatu_h}}{\fatu_h}\cdot\jump{\overline{\bm{\varphi}_h}}\;\dsxdt + \mu\inttinto \gradh\fatu_h:\gradh\bm{\varphi}_h\;\dxdt \notag \\[2mm]
    & \qquad +\nu\inttinto \divh(\fatu_h)\,\divh(\bm{\varphi}_h)\;\dxdt -\inttinto \left(\p(\varrho_h\thetah)+\hd\!\left[\varrho_h^2+(\varrho_h\thetah)^2\right]\right)\divh(\bm{\varphi}_h)\;\dxdt = 0\,. \label{int_m_method}
\end{align}
Let us consider the first term on the left-hand side of (\ref{int_m_method}). Due to the second estimate in Lemma~\ref{lem_time_der} with $(r_h,\phi)=(\varrho_h \overline{u_{h,i}},\varphi_i)$, $i\in\{1,\dots,d\}$, as well as Remark \ref{rem_proj_cont}, Hölder's inequality, the third estimate in (\ref{energy_est_1}), and the fact that $\deltat\approx h$, we have
\begin{align*}
    \inttinto \Dt(\varrho_h\overline{\fatu_h})\cdot\bm{\varphi}_h\;\dxdt = -\into \varrho_h^0\overline{\fatu_h^0}\cdot\bm{\varphi}(0,\cdot)\;\dx - \inttinto \varrho_h\overline{\fatu_h}\cdot\pt\bm{\varphi}\;\dxdt - J_{1,h}\,,
\end{align*}
where
\begin{align*}
    |J_{1,h}| &\lesssim (\deltat+h)\left(\maxnormk{\bm{\varphi}}{2}{(\overline{\Omega_T})^d}\,\lnorm{\varrho_h\overline{\fatu_h}}{1}{(0,T;\Lp{1}{0mm}(\Omega)^d)}{0.55}{-0.5mm} + \maxnormk{\bm{\varphi}}{1}{(\overline{\Omega_T})^d}\,\lnorm{\varrho_h^0\overline{\fatu_h^0}}{1}{(\Omega)^d}{0.55}{-0.5mm}\right) \notag \\[2mm]
    &\lesssim h\left(\lnorm{\varrho_h\overline{\fatu_h}}{\infty}{(0,T;\Lp{2\gamma/(\gamma+1)}{0mm}(\Omega)^d)}{0.55}{0mm} + \lnorm{\varrho_0}{2}{(\Omega)}{0.55}{-0.5mm}\,\wnorm{\fatu_0}{1,2}{(\Omega)^d}{0.55}{-0.5mm}\right) \lesssim h\,.
\end{align*}
Next, we turn to the last three terms on the left-hand side of (\ref{int_m_method}). It follows from Lemma \ref{lem_press} that
\begin{align*}
    &\mu\inttinto \gradh\fatu_h:\gradh\bm{\varphi}_h\;\dxdt
    +\inttinto \left(\nu\,\divh(\fatu_h)-\p(\varrho_h\thetah)-\hd\!\left[\varrho_h^2+(\varrho_h\thetah)^2\right]\right)\divh(\bm{\varphi}_h)\;\dxdt \notag \\[2mm]
    & = \mu\inttinto \gradh\fatu_h:\gradx\bm{\varphi}\;\dxdt
    +\inttinto \left(\nu\,\divh(\fatu_h)-\p(\varrho_h\thetah)-\hd\!\left[\varrho_h^2+(\varrho_h\thetah)^2\right]\right)\divx(\bm{\varphi})\;\dxdt\,.
\end{align*}
Finally, let us examine the second term on the left-hand side of (\ref{int_m_method}). Applying Corollary \ref{cor_upwind} with $(\fats,\fatw,\fatg,\bm{\psi})=(\varrho_h\overline{\fatu_h},\fatu_h,\bm{\varphi}_h,\bm{\varphi})(t,\cdot)$, $t\in[0,T]$, as well as the estimates (\ref{est_pivh_3})--(\ref{est_pivh_5}) and (\ref{est_pivh_1}), we deduce that
\begin{align*}
    \intt\!\!\inteint \up{\varrho_h\overline{\fatu_h}}{\fatu_h}\cdot\jump{\overline{\bm{\varphi}_h}}\;\dsxdt = \inttinto \varrho_h\overline{\fatu_h}\otimes\fatu_h:\gradx\bm{\varphi}\;\dxdt + \sum_{j\,=\,2}^5 J_{j,h}\,,
\end{align*}
where
\begin{gather*}
    |J_{2,h}| \lesssim h \intt\!\!\intek \big|\jump{\varrho_h\overline{\fatu_h}}\big[\smallgmean{\fatu_h\cdot\nk}\big]^-\big|\,\dsxdt\,, \qquad |J_{3,h}| \lesssim h \intt\!\!\intek |\varrho_h\overline{\fatu_h}|\,\big|\fatu_h-\smallgmean{\fatu_h}\big|\,\dsxdt\,, \\[2mm]
    \qquad |J_{4,h}| \lesssim h \inttinto \big|\varrho_h\overline{\fatu_h}\,\divh(\fatu_h)\big|\,\dxdt\,, \qquad |J_{5,h}| \lesssim h^{1+\varepsilon} \intt\!\!\inteint \big|\jump{\varrho_h\overline{\fatu_h}}\big|\,\dsxdt\,.
\end{gather*}
We continue by estimating the above terms.
\begin{itemize}
    \item{\textbf{Term} $|J_{2,h}|$. We observe that $\jump{\varrho_h\overline{\fatu_h}} = \varrho_h^{\,\mathrm{out}}\jump{\overline{\fatu_h}} + \jump{\varrho_h}\overline{\fatu_h}^{\,\mathrm{in}}$, which implies
    \begin{align}
        |J_{2,h}| &\lesssim h \intt\!\!\intek \big|\varrho_h^{\,\mathrm{out}}\jump{\overline{\fatu_h}}\big[\smallgmean{\fatu_h\cdot\nk}\big]^-\big|\,\dsxdt 
        + h \intt\!\!\intek \big|\jump{\varrho_h}\overline{\fatu_h}\big[\smallgmean{\fatu_h\cdot\nk}\big]^-\big|\,\dsxdt\,. \label{j3h}
    \end{align}
    Employing Hölder's inequality, (\ref{energy_est_8}), (\ref{est_vh_2}), the first estimate in (\ref{trace_neg_lp}), the first and third estimate in (\ref{energy_est_2}), and the second estimate in (\ref{energy_est_3}), we see that
    \begin{align}
        &h \intt\!\!\intek \big|\varrho_h^{\,\mathrm{out}}\jump{\overline{\fatu_h}}\big[\smallgmean{\fatu_h\cdot\nk}\big]^-\big|\,\dsxdt \notag \\[2mm]
        & \quad \lesssim h \left(\intt\!\!\intek -\varrho_h^{\,\mathrm{out}}\jump{\overline{\fatu_h}}^2\big[\smallgmean{\fatu_h\cdot\nk}\big]^-\;\dsxdt\right)^{1/2}\left(\intt\!\!\intek\varrho_h^{\,\mathrm{out}}|\smallgmean{\fatu_h}|\;\dsxdt\right)^{1/2} \notag \\[2mm]
        & \quad = h \left(\intt\!\!\inteint \left(\varrho_h^{\,\mathrm{in}}\big[\smallgmean{\fatu_h\cdot\ngamma}\big]^+ -\varrho_h^{\,\mathrm{out}}\big[\smallgmean{\fatu_h\cdot\ngamma}\big]^-\right)\jump{\overline{\fatu_h}}^2\;\dsxdt\right)^{1/2} \times \notag \\[2mm]
        & \quad \phantom{\,=\,h}\times\left(\intt\!\!\intek\varrho_h^{\,\mathrm{out}}|\smallgmean{\fatu_h}|\;\dsxdt\right)^{1/2} \notag \\[2mm]
        & \quad \lesssim h\big(h^{-1}\,\lnorm{\varrho_h}{2}{(0,T;\Lp{2}{0mm}(\Omega))}{0.55}{-0.5mm}\big(\lnorm{\fatu_h}{2}{(0,T;\Lp{2}{0mm}(\Omega)^d)}{0.55}{-0.5mm}+h\,\lnorm{\gradh\fatu_h}{2}{(0,T;\Lp{2}{0mm}(\Omega)^{d\times d})}{0.55}{-0.5mm}\big)\big)^{1/2} \notag \\[2mm]
        & \quad \lesssim h\big(h^{-1-\delta/2}\,\lnorm{h^{\delta/2}\varrho_h}{\infty}{(0,T;\Lp{2}{0mm}(\Omega))}{0.55}{-0mm}\big(\lnorm{\fatu_h}{2}{(0,T;\Lp{2}{0mm}(\Omega)^d)}{0.55}{-0.5mm}+h\,\lnorm{\gradh\fatu_h}{2}{(0,T;\Lp{2}{0mm}(\Omega)^{d\times d})}{0.55}{-0.5mm}\big)\big)^{1/2} \notag \\[2mm]
        & \quad \lesssim h^{1/2-\delta/4} + h^{1-\delta/4}\,. \label{j3h_1}
    \end{align}
    Next, using Hölder's inequality, the estimates (\ref{trace_neg_lp}), (\ref{est_vh_2}), (\ref{energy_est_5}), the first and third estimate in (\ref{energy_est_2}), the second estimate in (\ref{energy_est_3}), and the fact that $\deltat\approx h$, we deduce that
    \begin{align}
        &h \intt\!\!\intek \big|\jump{\varrho_h}\overline{\fatu_h}\big[\smallgmean{\fatu_h\cdot\nk}\big]^-\big|\;\dsxdt \notag \\[2mm]
        & \quad \lesssim h^{1-\delta/2} \left(\hd\intt\!\!\intek \jump{\varrho_h}^2|\smallgmean{\fatu_h\cdot\nk}|\;\dsxdt\right)^{1/2}\left(\intt\!\!\intek \overline{\fatu_h}^2|\smallgmean{\fatu_h}|\;\dsxdt\right)^{1/2} \notag \\[2mm]
        & \quad \lesssim h^{1-\delta/2}\big(h^{-1}\,\explnorm{\fatu_h}{2}{(0,T;\Lp{6}{0mm}(\Omega)^d)}{0.55}{-0.5mm}{0.5mm}{2}\big(\lnorm{\fatu_h}{\infty}{(0,T;\Lp{3/2}{0mm}(\Omega)^d)}{0.55}{-0mm}+h\,\lnorm{\gradh\fatu_h}{\infty}{(0,T;\Lp{3/2}{0mm}(\Omega)^{d\times d})}{0.55}{-0mm}\big)\big)^{1/2} \notag \\[2mm]
        & \quad \lesssim h^{1-\delta/2}\big(h^{-1}(\deltat)^{-1/2}\,\big(\lnorm{\fatu_h}{2}{(0,T;\Lp{3/2}{0mm}(\Omega)^d)}{0.55}{-0.5mm}+h\,\lnorm{\gradh\fatu_h}{2}{(0,T;\Lp{3/2}{0mm}(\Omega)^{d\times d})}{0.55}{-0.5mm}\big)\big)^{1/2} \notag \\[2mm]
        & \quad \lesssim h^{1/4-\delta/2} + h^{3/4-\delta/2}\,. \label{j3h_2}
    \end{align}
    Consequently, plugging (\ref{j3h_1}) and (\ref{j3h_2}) into (\ref{j3h}), we obtain
    \begin{align*}
        |J_{2,h}| \lesssim h^{1/2-\delta/4} + h^{1-\delta/4} + h^{1/4-\delta/2} + h^{3/4-\delta/2}\,.
    \end{align*}}
    \item{\textbf{Term} $|J_{3,h}|$. Applying Hölder's inequality, the first estimate in (\ref{trace_neg_lp}), the second estimate in (\ref{est_vh_1}), the first estimate in (\ref{energy_est_2}), and the second estimate in (\ref{energy_est_4}), we conclude that
    \begin{align*}
        |J_{3,h}|\lesssim
        h\,\lnorm{\varrho_h\overline{\fatu_h}}{2}{(0,T;\Lp{2}{0mm}(\Omega)^d)}{0.55}{-0.5mm}\,\lnorm{\gradh\fatu_h}{2}{(0,T;\Lp{2}{0mm}(\Omega)^{d\times d})}{0.55}{-0.5mm}\lesssim h^{1-(d+3\delta)/6}\,.
    \end{align*}}
    \item{\textbf{Term} $|J_{4,h}|$. Employing Hölder's inequality, the first estimate in (\ref{energy_est_2}), and the second estimate in (\ref{energy_est_4}), we obtain
    \begin{align*}
        |J_{4,h}|\lesssim h\,\lnorm{\varrho_h\overline{\fatu_h}}{2}{(0,T;\Lp{2}{0mm}(\Omega)^d)}{0.55}{-0.5mm}\,\lnorm{\divh(\fatu_h)}{2}{(0,T;\Lp{2}{0mm}(\Omega))}{0.55}{-0.5mm}\lesssim h^{1-(d+3\delta)/6}\,.
    \end{align*}}
    \item{\textbf{Term} $|J_{5,h}|$.
    Using the first estimate in (\ref{trace_neg_lp}) and the third estimate in (\ref{energy_est_1}), we deduce that
    \begin{align*}
        |J_{5,h}| & \lesssim
        \he\, \lnorm{\varrho_h\overline{\fatu_h}}{1}{(0,T;\Lp{1}{0mm}(\Omega)^d)}{0.55}{-0.5mm}
        \lesssim \he\, \lnorm{\varrho_h\overline{\fatu_h}}{\infty}{(0,T;\Lp{2\gamma/(\gamma+1)}{0mm}(\Omega)^d)}{0.55}{0mm} \lesssim \he\,.
    \end{align*}
    }
\end{itemize}
Consequently, we have
\begin{align}
    &-\into \varrho_h^0\overline{\fatu_h^0}\cdot\bm{\varphi}(0,\cdot)\;\dx + \inttinto\big[\mu\gradh\fatu_h:\gradx\bm{\varphi}+\nu\,\divh(\fatu_h)\,\divx(\bm{\varphi})\big]\,\dxdt + \mathcal{O}(h^{\alpha_2}) \notag \\[2mm]
    & \qquad = \inttinto \big[\varrho_h\overline{\fatu_h}\cdot\pt\bm{\varphi}+\varrho_h\overline{\fatu_h}\otimes\fatu_h:\gradx\bm{\varphi}+\left(\p(\varrho_h\thetah)+\hd\big[\varrho_h^2+(\varrho_h\thetah)^2\big]\right)\divx(\bm{\varphi})\big]\,\dxdt \label{proof_cons_m_method_1}
\end{align}
with $\alpha_2 = \min\left\{\varepsilon,\frac{1-2\delta}{4}\right\}>0$ as $h\downarrow 0$. Then, using Hölder's inequality, the first estimate in (\ref{est_vh_1}), the second estimate in (\ref{energy_est_4}), and the first estimate in (\ref{energy_est_2}), we deduce that
\begin{align*}
    &\left|\,\inttinto \varrho_h\overline{\fatu_h}\otimes\big(\fatu_h-\overline{\fatu_h}\big):\gradx\bm{\varphi}\;\dxdt\,\right| 
    \lesssim \lnorm{\varrho_h\overline{\fatu_h}}{2}{(0,T;\Lp{2}{0mm}(\Omega)^d)}{0.55}{-0.5mm}\,\lnorm{\fatu_h-\overline{\fatu_h}}{2}{(0,T;\Lp{2}{0mm}(\Omega)^d)}{0.55}{-0.5mm} \notag \\[2mm]
    & \qquad \lesssim h\,\lnorm{\varrho_h\overline{\fatu_h}}{2}{(0,T;\Lp{2}{0mm}(\Omega)^d)}{0.55}{-0.5mm}\,\lnorm{\gradh\fatu_h}{2}{(0,T;\Lp{2}{0mm}(\Omega)^{d\times d})}{0.55}{-0.5mm} \lesssim h^{1-(d+3\delta)/6}\,.
\end{align*}
Hence, we may rewrite (\ref{proof_cons_m_method_1}) as
\begin{align*}
    &-\into \varrho_h^0\overline{\fatu_h^0}\cdot\bm{\varphi}(0,\cdot)\;\dx + \inttinto\big[\mu\gradh\fatu_h:\gradx\bm{\varphi}+\nu\,\divh(\fatu_h)\,\divx(\bm{\varphi})\big]\,\dxdt + \mathcal{O}(h^{\alpha_2}) \notag \\[2mm]
    & \qquad = \inttinto \big[\varrho_h\overline{\fatu_h}\cdot\pt\bm{\varphi}+\varrho_h\overline{\fatu_h}\otimes\overline{\fatu_h}:\gradx\bm{\varphi}+\left(\p(\varrho_h\thetah)+\hd\big[\varrho_h^2+(\varrho_h\thetah)^2\big]\right)\divx(\bm{\varphi})\big]\,\dxdt
\end{align*}
as $h\downarrow 0$. \\[4mm]

\noindent \textbf{The entropy inequality.} \\[2mm]
Taking $\chi=\ln$ in Lemma \ref{lem_disc_entropy_inequ}, we deduce that
\begin{align}
    0 &\leq \inttinto \Dt(\varrho_h \ln(\thetah))\,\psi_h\;\dxdt - \inttinteint \upp{\varrho_h \ln(\thetah)}{\fatu_h}\jump{\psi_h}\;\dsxdt  -\sum_{j\,=\,5}^7 H_{j,h}\,, \label{cons_entropy_inequ_1}
\end{align}
where
\begin{gather*}
    H_{5,h} = \hee \inttinteint \jump{\varrho_h}\jump{\psi_h}\;\dsxdt\,,  \qquad H_{6,h} = -\hee \inttinteint \jump{\varrho_h}\jump{\ln(\thetah)\psi_h}\;\dsxdt\,, \notag \\[2mm]
    H_{7,h} = -\hee \inttinteint \jump{\varrho_h \thetah}\Jump{\frac{\psi_h}{\thetah}}\dsxdt\,.
\end{gather*}
Now we may rewrite the first two integrals in (\ref{cons_entropy_inequ_1}) following the procedure used to handle the continuity equation. We arrive at
\begin{align}
    -\into \varrho_h^0\ln(\thetazeroh) \psi(0,\cdot)\;\dx \geq \inttinto \big[\varrho_h\ln(\thetah)\pt\psi+\varrho_h\ln(\thetah)\fatu_h\cdot\gradx\psi\big]\;\dxdt + \sum_{j\,=\,1}^7 H_{j,h}\,, \label{cons_entropy_inequ_2}
\end{align}
where for $j\in\{1,...,4\}$ the error term $H_{j,h}$ equals $I_{j,h}$ with $\varrho_h$ replaced by $\varrho_h\ln(\thetah)$ and $\varphi_h$ replaced by $\psi_h$. Here, it is to be noted that the analogue of the error term $I_{5,h}$ will not be there since (\ref{cons_entropy_inequ_1}) contains the usual upwind operator $\upp{\,\cdot\,}{\,\cdot\,}$ instead of the dissipative upwind operator $\up{\,\cdot\,}{\,\cdot\,}$. Since $H_{5,h}=-I_{5,h}$, $c_\star\leq\thetah\leq c^\star$ and
\begin{align*}
    \big|\jumpG{\varrho^k_h\ln(\thetakh)}\big| &= \big|\jumpG{\varrho^k_h}\meanG{\ln(\thetakh)} + \meanG{\varrho^k_h}\jumpG{\ln(\thetakh)}\big| = \left|\jumpG{\varrho^k_h}\meanG{\ln(\thetakh)} + \frac{1}{\eta_{\theta,k,\sigma}}\, \meanG{\varrho^k_h}\jumpG{\thetakh}\right| \\[2mm]
    &= \left|\jumpG{\varrho^k_h}\meanG{\ln(\thetakh)} + \frac{1}{\eta_{\theta,k,\sigma}}\,\big(\jumpG{\varrho^k_h\thetakh}-\jumpG{\varrho^k_h}\meanG{\thetakh}\big)\right|
    \lesssim \big|\jumpG{\varrho^k_h}\big| + \big|\jumpG{\varrho^k_h\thetakh}\big|
\end{align*}
for every $(k,\sigma)\in\naturals\times\gint$ and suitably chosen values $(\eta_{\theta,k,\sigma})_{\sigma\,\in\,\gint}\subset[c_\star,c^\star]$, it is easy to see that
\begin{align*}
    |H_{j,h}| \lesssim h^{\alpha_1} \qquad \text{for $1\leq j\leq 5$.}
\end{align*}
Moreover, combining $c_\star\leq\thetah\leq c^\star$ with Hölder's inequality, the first estimate in (\ref{trace_neg_lp}), the second estimate in (\ref{energy_est_1}) and the first estimate in (\ref{energy_est_3}), we deduce that
\begin{align*}
    |H_{6,h}|\,,|H_{7,h}| \lesssim h^{\varepsilon-1}\,.
\end{align*}
Finally, seeing that (by a computation similar to that in (\ref{last_step_cons_r_method})) we have
\begin{align*}
    \left|\,\inttinto \varrho_h\ln(\thetah)\big(\fatu_h-\overline{\fatu_h}\big)\cdot\gradx\varphi\;\dxdt\,\right| \lesssim h^{1-\delta/2}\,,
\end{align*}
we may rewrite (\ref{cons_entropy_inequ_2}) as
\begin{align*}
    -\into \varrho_h^0\ln(\thetazeroh) \psi(0,\cdot)\;\dx \geq \inttinto \big[\varrho_h\ln(\thetah)\pt\psi+\varrho_h\ln(\thetah)\,\overline{\fatu_h}\cdot\gradx\psi\big]\;\dxdt + \mathcal{O}(h^{\alpha_3})\,,
\end{align*}
where $\alpha_3=\min\{\alpha_1,\varepsilon-1\}$. In particular, we can choose $\beta=\min\{\alpha_1,\alpha_2,\alpha_3\}=\min\left\{\varepsilon-1,\tfrac{1-2\delta}{4}\right\}$. \\
\vphantom{\Big)}
\end{proof}


\section{Convergence}\label{sec_convergence}


We proceed by proving our main result, namely Theorem \ref{thm_main}.

\begin{proof}[Proof of Theorem $\mathbf{\ref{thm_main}}$]
Let $\{(\varrho_h,\thetah,\fatu_h)\}_{h\,\downarrow\,0}$ be a sequence of solutions to the FE-FV method (\ref{r_method})--(\ref{m_method}) starting from the initial data $\{(\varrho_h^0,\thetazeroh,\fatu_h^0)\}_{h\,\downarrow\,0}$ defined in (\ref{disc_int_1}). Here we suppose that the parameters satisfy (\ref{cons_cond}).

Due to the second estimate in (\ref{energy_est_1}), the first estimate in (\ref{energy_est_3}), the third estimate in (\ref{energy_est_2}), (\ref{bound_piq}), and Corollary \ref{cor_sol}(\hyperlink{cor_ii}{i}), the sequence $\{(\varrho_h,\thetah,\overline{\fatu_h})|_{\Omega_T}\}_{h\,\downarrow\,0}$ generates a Young measure $\mathcal{V}\in\Lp{\infty}{0.2mm}_{\mathrm{weak}^\star}(\Omega_T;\mathcal{P}(\reals^{d+2}))$ that satisfies
\[\mathcal{V}_{(t,\fatx)}\big(\{\bvarrho\geq 0\}\cap\{c_\star\leq \btheta\leq c^\star\}\big)=1 \quad \text{for almost all $(t,\fatx)\in\Omega_T$.}\]
Taking into account the remaining estimates in (\ref{energy_est_1})--(\ref{energy_est_4}) as well as the first estimate in (\ref{est_vh_1})
and passing to a subsequence as the case may be, we obtain that
\begin{gather}
    \varrho_h f(\thetah) \rightharpoonup^\star \omean{\mathcal{V};\bvarrho f(\btheta)} \quad \text{in $\Lp{\infty}{0.2mm}(0,T;\Lp{\gamma}{0.5mm}(\Omega))$ for all $f\in C(0,\infty)$,} \label{conv_1} \\[1mm]
    \varrho_h f(\thetah)\,\overline{\fatu_h} \rightharpoonup^\star \omean{\mathcal{V};\bvarrho f(\btheta)\bfatu} \quad \text{in $\Lp{\infty}{0.2mm}(0,T;\Lp{\tfrac{2\gamma}{\gamma+1}}{0.5mm}(\Omega)^d)$ for all $f\in C(0,\infty)$,} \label{conv_2} \\[2mm]
    \overline{\fatu_h},\fatu_h \rightharpoonup \fatu_{\mathcal{V}} \equiv \omean{\mathcal{V};\bfatu} \quad \text{in $\Lp{2}{0mm}(\Omega_T)^d$} \quad \text{and} \quad
    \gradh\fatu_h \rightharpoonup \mathbb{U} \quad \text{in $\Lp{2}{0mm}(\Omega_T)^{d\times d}$,} \label{conv_3} \\[2mm]
    \hd\big[\varrho_h^2+(\varrho_h\thetah)^2\big] \rightharpoonup^\star \mathfrak{R}(\delta,\varrho,\theta) \quad \text{in $\Lp{\infty}{0.2mm}_{\mathrm{weak}^\star}(0,T;\mathcal{M}^+(\overline{\Omega}))$,} \label{conv_4} \\[2mm]
    \varrho_h\overline{\fatu_h}\otimes\overline{\fatu_h} + \p(\varrho_h\thetah)\mathbb{I} \rightharpoonup^\star \overline{\varrho\fatu\otimes\fatu + \p(\varrho\theta)\mathbb{I}} \quad \text{in $\Lp{\infty}{0.2mm}_{\mathrm{weak}^\star}(0,T;\mathcal{M}(\overline{\Omega})^{d\times d})$,} \label{conv_5} \\[2mm]
    \frac{1}{2}\varrho_h|\overline{\fatu_h}|^2 + P(\varrho_h\thetah) \rightharpoonup^\star \overline{\vphantom{\overline{\frac{1}{2}}}\frac{1}{2}\varrho|\fatu|^2 + P(\varrho\theta)} \quad \text{in $\Lp{\infty}{0.2mm}_{\mathrm{weak}^\star}(0,T;\mathcal{M}(\overline{\Omega}))$,} \label{conv_6} \\[2mm]
    |\gradh\fatu_h|^2 \rightharpoonup^\star \overline{|\gradx\fatu|^2} \quad \text{and} \quad |\divh(\fatu_h)|^2 \rightharpoonup^\star \overline{|\divx(\fatu)|^2} \quad \text{in $\mathcal{M}(\overline{\Omega_T})$} \label{conv_7} 
\end{gather}
as $h\downarrow 0$.
Following the arguments given in \citep[Chapter 10.2.1, pp.\,139 and 142/143]{Feireisl_Karper_Pokorny}, we deduce that
\begin{gather}
    \fatu_{\mathcal{V}} \in \Lp{2}{0mm}(0,T;W^{1,2}_0(\Omega)^d) \qquad \text{and} \qquad \mathbb{U} = \gradx\fatu_{\mathcal{V}}\,. \label{reg_u_limit}
\end{gather}
Moreover, using Hölder's inequality, (\ref{bound_piq}), the first estimate in (\ref{est_vh_1}), the assumption on the initial data, and Lemma \ref{lem_conv_piq_piv}, we easily verify that
\begin{gather}
    \into \varrho_h^0 f(\thetazeroh)\varphi\;\dx \to \into \varrho_0 f(\thetazero)\varphi\;\dx \quad \text{for all $\varphi\in\Cinfty(\overline{\Omega})$ and $f\in\Ck{1}(0,\infty)$,} \label{conv_init_c_z} \\[2mm]
    \into \varrho_h^0\overline{\fatu_h^0}\cdot\bm{\varphi}\;\dx \to \into \varrho_0\fatu_0\cdot\bm{\varphi}\;\dx \quad \text{for all $\bm{\varphi}\in\Cinfty(\overline{\Omega})^d$,} \label{conv_init_m} \\[2mm]
    \text{and} \quad \into \left[\,\frac{1}{2}\,\varrho_h^0|\overline{\fatu_h^0}|^2+P(\varrho_h^0\thetazeroh)+\hd\big((\varrho_h^0)^2+(\varrho_h^0\thetazeroh)^2\big)\right]\dx \to \into \left[\,\frac{1}{2}\,\varrho_0|\fatu_0|^2+P(\varrho_0\thetazero)\right]\dx \label{conv_init_energy}
\end{gather}
as $h\downarrow 0$. \\[2mm]

\noindent \textbf{Energy inequality.} \\[2mm]
From the discrete energy balance (\ref{disc_energy}) we derive that
\begin{align*}
    &\intt\!\!\phi\left[\,\into \left(\frac{1}{2}\,\varrho_h|\overline{\fatu_h}|^2+P(\varrho_h\thetah)+\hd\big[\varrho_h^2+(\varrho_h\thetah)^2\big]\right)\dx + \inttauinto \left[\mu|\gradh\fatu_h|^2+\nu|\divh(\fatu_h)|^2\right]\dxdt\right]\dd\tau \notag \\[2mm]
    & \qquad \leq \intt\!\!\phi \into \left(\frac{1}{2}\,\varrho_h^0|\overline{\fatu_h^0}|^2+P(\varrho_h^0\thetazeroh)+\hd\big[(\varrho_h^0)^2+(\varrho_h^0\thetazeroh)^2\big]\right)\dx\hspace{0.3mm}\dd\tau
\end{align*}
for all $\phi\in\Cinftyc(0,T)$, $\phi\geq 0$. Due to (\ref{conv_init_energy}), (\ref{conv_4}), (\ref{conv_6}),
and (\ref{conv_7}),
we may perform the limit $h\downarrow 0$ to obtain
\begin{align}
    &\intt \phi\left[\,\int_{\,\overline{\Omega}}\dd\!\left(\,\overline{\vphantom{\overline{\frac{1}{2}}}\frac{1}{2}\varrho|\fatu|^2 + P(\varrho\theta)}+\mathfrak{R}(\delta,\varrho,\theta)\right)\!(\tau) + \int_{\,\overline{\Omega_\tau}}\dd\!\left(\,\mu\hspace{0.3mm}\overline{|\gradx\fatu|^2}+\nu\hspace{0.3mm}\overline{|\divx(\fatu)|^2}\,\right)\right]\dd\tau \notag \\[2mm]
    & \qquad \leq \intt \phi \into \left[\,\frac{1}{2}\,\varrho_0|\fatu_0|^2+P(\varrho_0\thetazero)\right]\dx\hspace{0.3mm}\dd\tau \label{vfdvds}
\end{align}
for all $\phi\in\Cinftyc(0,T)$, $\phi\geq 0$. Furthermore, we make the subsequent observations:
\begin{itemize}
    \item{In view of the first estimate in (\ref{energy_est_1}) and the first in (\ref{energy_est_3}), we may apply \citep[Chapter 5, Proposition 5.2]{Feireisl_Lukacova_Book} to see that
    \begin{align}
        \left\langle\mathcal{V};\frac{1}{2}\,\bvarrho\,|\bfatu|^2+P(\bvarrho\,\btheta)\right\rangle\in\Lp{1}{0mm}(\Omega_T). \label{well_defined}
    \end{align}
    Moreover, we may use \citep[Lemma 2.1]{Feireisl_Wiedemann} with $F\equiv 0$ and $G(\bvarrho,\btheta,\bfatu)=\frac{1}{2}\hspace{0.3mm}\bvarrho\,|\bfatu|^2+P(\bvarrho\,\btheta)$ to deduce that
    \begin{align*}
        \mathfrak{E} = \left(\,\overline{\vphantom{\overline{\frac{1}{2}}}\frac{1}{2}\varrho|\fatu|^2 + P(\varrho\theta)} - \left\langle\mathcal{V};\frac{1}{2}\,\bvarrho\,|\bfatu|^2+P(\bvarrho\,\btheta)\right\rangle\right) +\mathfrak{R}(\delta,\varrho,\theta)\in \Lp{\infty}{0.2mm}_{\mathrm{weak}^\star}(0,T;\mathcal{M}^+(\overline{\Omega}))\,.
    \end{align*}}
    \item{
    Applying measure-theoretic arguments to the viscous terms, we conclude that
    \begin{align*}
        \mathfrak{D} = \mu\hspace{0.3mm}\overline{|\gradx\fatu|^2}+\nu\hspace{0.3mm}\overline{|\divx(\fatu)|^2} - \left[\mu|\gradx\fatu_{\mathcal{V}}|^2+\nu|\divx(\fatu_{\mathcal{V}})|^2\right]\in \mathcal{M}^+(\overline{\Omega_T})\,.
    \end{align*}
    }
    \item{Using the density of $\Cinftyc(\Omega)$ in $W^{1,2}_0(\Omega)$ as well as Gauss's theorem, we easily verify that
    \begin{align*}
        \inttauinto \left[\mu|\gradx\fatu_{\mathcal{V}}|^2+\nu|\divx(\fatu_{\mathcal{V}})|^2\right]\dxdt = \inttauinto \mathbb{S}(\gradx\fatu_{\mathcal{V}}):\gradx\fatu_{\mathcal{V}}\;\dxdt\,.
    \end{align*}}
\end{itemize}
In particular, we may rewrite (\ref{vfdvds}) in the form
\begin{align*}
    &\into\left\langle\mathcal{V}_{(\tau,\,\cdot\,)};\frac{1}{2}\,\bvarrho\,|\bfatu|^2+P(\bvarrho\,\btheta)\right\rangle\dx + \inttauinto \mathbb{S}(\gradx\fatu_{\mathcal{V}}):\gradx\fatu_{\mathcal{V}}\;\dxdt \notag \\[2mm]
    &\qquad + \int_{\,\overline{\Omega}}\dd\mathfrak{E}(\tau) + \int_{\,\overline{\Omega_\tau}}\dd\mathfrak{D} \leq \into\left[\,\frac{1}{2}\,\varrho_0|\fatu_0|^2+P(\varrho_0\thetazero)\right]\dx
    \qquad \text{for a.a. $\tau\in(0,T).$}
\end{align*}

\pagebreak
\noindent \textbf{Continuity equation.} \\[2mm]
In view of (\ref{conv_1}), (\ref{conv_2}), and (\ref{conv_init_c_z}), we may perform the limit $h\downarrow 0$ in (\ref{cons_r_method}). We obtain
\begin{align}
    -\into \varrho_0\,\varphi(0,\cdot)\;\dx = \inttinto \big[\omean{\mathcal{V};\bvarrho\,}\,\pt\varphi+\omean{\mathcal{V};\bvarrho\,\bfatu}\cdot\gradx\varphi\big]\,\dxdt \label{conv_r_method_1}
\end{align}
for all $\varphi\in \Cinftyc([0,T)\times\overline{\Omega})$. Following the arguments presented in \citep[Chapter 2.1.3]{Feireisl_Lukacova_Book}, we deduce from (\ref{conv_r_method_1}) that $\omean{\mathcal{V};\bvarrho\,}\in C_{\mathrm{weak}}([0,T];\Lp{\gamma}{0.5mm}(\Omega))$. Consequently, (\ref{conv_r_method_1}) can be rewritten in the form
\begin{align}
    \left[\,\into \omean{\mathcal{V}_{(t,\,\cdot\,)};\bvarrho\,}\,\varphi(t,\cdot)\;\dx\right]_{t\,=\,0}^{t\,=\,\tau} = \inttauinto \big[\omean{\mathcal{V};\bvarrho\,}\,\pt\varphi+\omean{\mathcal{V};\bvarrho\,\bfatu}\cdot\gradx\varphi\big]\,\dxdt \label{conv_r_method_2}
\end{align}
for all $\tau\in[0,T]$ and $\varphi\in\Cinfty(\overline{\Omega_T})$. Here, we have set
$
    \mathcal{V}_{(0,\fatx)} = \delta_{(\varrho_0(\fatx),\thetazero(\fatx),\fatu_0(\fatx))}\,\footnotemark.
$
\footnotetext{$\delta_{\faty}$ denotes the Dirac measure centered on $\faty\in\reals^{d+2}$.}Due to the integrability properties of $\omean{\mathcal{V};\bvarrho\,}$ and $\omean{\mathcal{V};\bvarrho\,\bfatu}$, the boundedness of $\Omega_T$, the fact that $\Cinfty(\overline{\Omega_T})$ is dense in $W^{1,p}(\Omega_T)$ for every $p\in[1,\infty)$, and the Sobolev embedding $W^{1,q}(\Omega_T)\hookrightarrow C(\overline{\Omega_T})$ for $q>d+1$,
we may extend the validity of (\ref{conv_r_method_2}) to test functions $\varphi$ of the class $W^{1,\infty}(\Omega_T)$. \\[2mm]

\noindent \textbf{Potential temperature equation.} \\[2mm]
The potential temperature equation can be handled in the same manner as the continuity equation. \\[2mm]

\noindent \textbf{Momentum equation.} \\[2mm]
Thanks to (\ref{conv_3})--(\ref{conv_5}), (\ref{reg_u_limit}), and (\ref{conv_init_m}), we can take the limit $h\downarrow 0$ in (\ref{cons_m_method}). We obtain
\begin{align}
    &-\into \varrho_0\fatu_0\cdot\bm{\varphi}(0,\cdot)\;\dx + \inttinto\big[\mu\gradx\fatu_{\mathcal{V}}:\gradx\bm{\varphi}+\nu\,\divx(\fatu_{\mathcal{V}})\,\divx(\bm{\varphi})\big]\,\dxdt \notag \\[2mm]
    & \qquad = \inttinto \omean{\mathcal{V};\bvarrho\,\bfatu}\cdot\pt\bm{\varphi}\;\dxdt + \inttinto \gradx\bm{\varphi}:\dd\!\left(\,\overline{\varrho\fatu\otimes\fatu + \p(\varrho\theta)\mathbb{I}}+ \mathfrak{R}(\delta,\varrho,\theta)\mathbb{I}\right)\dt \label{conv_m_method_1}
\end{align}
for all $\bm{\varphi}\in \Cinftyc([0,T)\times\Omega)^d$. Next, we make the following observations:
\begin{itemize}
    \item{Analogous to above, it follows from (\ref{conv_m_method_1}) that $\omean{\mathcal{V};\bvarrho\,\bfatu}\in C_{\mathrm{weak}}([0,T];\Lp{\tfrac{2\gamma}{\gamma+1}}{0.5mm}(\Omega)^d)$.}
    \item{Using Gauss's theorem, we conclude that
    \begin{align*}
        \inttinto\big[\mu\gradx\fatu_{\mathcal{V}}:\gradx\bm{\varphi}+\nu\,\divx(\fatu_{\mathcal{V}})\,\divx(\bm{\varphi})\big]\,\dxdt = \inttinto \mathbb{S}(\gradx\fatu_{\mathcal{V}}):\gradx\bm{\varphi}\;\dxdt\,.
    \end{align*}}
    \item{Due to (\ref{well_defined}), $\omean{\mathcal{V};\bvarrho\bfatu\otimes\bfatu + \p(\bvarrho\,\btheta\,)\mathbb{I}}\in \Lp{1}{0mm}(\Omega_T)$. Moreover, applying \citep[Lemma 2.1]{Feireisl_Wiedemann} with $F\equiv 0$ and $G(\bvarrho,\btheta,\bfatu) = (\xi\otimes\xi):(\bvarrho\,\bfatu\otimes\bfatu + \p(\bvarrho\,\btheta\,)\mathbb{I})$, $\xi\in\reals^d$, we deduce that
    \begin{align*}
        \bm{\mathfrak{R}} = \left(\,\overline{\varrho\fatu\otimes\fatu + \p(\varrho\theta)\mathbb{I}} - \omean{\mathcal{V};\bvarrho\,\bfatu\otimes\bfatu + \p(\bvarrho\,\btheta\,)\mathbb{I}}\right) + \mathfrak{R}(\delta,\varrho,\theta)\mathbb{I} \in \Lp{\infty}{0.2mm}_{\mathrm{weak}^\star}(0,T;\mathcal{M}(\overline{\Omega})^{d\times d}_{\mathrm{sym},+})\,.
    \end{align*}
    In particular,
    \begin{align*}
        \underline{d}\mathfrak{E} \leq \mathrm{tr}[\bm{\mathfrak{R}}] \leq \overline{d}\mathfrak{E}\,, \quad \text{where} \quad \underline{d} = \min\{2,d(\gamma-1)\} \quad \text{and} \quad \overline{d} = \max\{d,d(\gamma-1)\}\,.
    \end{align*}}
\end{itemize}
Consequently, (\ref{conv_m_method_1}) can be rewritten as
\begin{align}
    &\left[\,\into \omean{\mathcal{V}_{(t,\,\cdot\,)};\bvarrho\,\bfatu}\cdot\bm{\varphi}(t,\cdot)\;\dx\right]^{t\,=\,\tau}_{t\,=\,0} + \inttauinto \mathbb{S}(\gradx\fatu_{\mathcal{V}}):\gradx\bm{\varphi}\;\dxdt \notag \\[2mm]
    & \qquad = \inttauinto \big[\omean{\mathcal{V};\bvarrho\,\bfatu}\cdot\pt\bm{\varphi} + \omean{\mathcal{V};\bvarrho\,\bfatu\otimes\bfatu + \p(\bvarrho\,\btheta\,)\mathbb{I}}:\gradx\bm{\varphi}\big]\,\dxdt + \inttauinto \gradx\bm{\varphi}:\dd\bm{\mathfrak{R}}(t)\hspace{0.3mm}\dt \label{conv_m_method_2}
\end{align}
for all $\tau\in[0,T]$ and all $\bm{\varphi}\in \Cinftyc([0,T]\times\Omega)^d$. It is easy to see that (\ref{conv_m_method_2}) also holds for test functions $\bm{\varphi}$ of the class $\Ckc{1}([0,T]\times\Omega)^d$. Moreover, for every $\bm{\varphi}\in \Ck{1}(\overline{\Omega_T})^d$ satisfying $\bm{\varphi}|_{[0,T]\times\po}=\bm{0}$ we can construct a sequence $\{\bm{\varphi}_n\}_{n\,\in\,\naturals}\subset \Ckc{1}([0,T]\times\Omega)^d$ of smoothed truncations of $\bm{\varphi}$ such that
\begin{center}
    $(\bm{\varphi}_n,\gradx\bm{\varphi}_n,\pt\bm{\varphi}_n)\to(\bm{\varphi},\gradx\bm{\varphi},\pt\bm{\varphi})$ pointwise in $\Omega_T$ \qquad and  \qquad
    {${\displaystyle\sup_{n\,\in\,\naturals}\left\{\maxnormk{\bm{\varphi}_n}{1}{(\overline{\Omega_T})}\right\}} < \infty$.}
\end{center}
Accordingly, we may use the dominated convergence theorem to extend the validity of (\ref{conv_m_method_2}) to test functions $\bm{\varphi}\in \Ck{1}(\overline{\Omega_T})^d$ satisfying $\bm{\varphi}|_{[0,T]\times\po}=\bm{0}$. \\[2mm]

\noindent \textbf{Poincaré's inequality.} \\[2mm]
Let $\tau\in[0,T]$ and $\barepsilon>0$ be arbitrary. Further, let $\{\vartheta_k\}_{k\,\in\,\naturals}\subset C_c(\reals^{d+2})$ be the sequence of functions defined by
\begin{align*}
    \vartheta_k(\faty)=\left\{{\arraycolsep=0pt
    \begin{array}{cl}
        1 & \quad \text{if $|\faty|<k$,} \\[2mm]
        1+k-|\faty| & \quad \text{if $|\faty|\in[k,k+1)$, and} \\[2mm]
        0 & \quad \text{else,}
    \end{array}}\right.
\end{align*}
and $C\geq 1$ a constant such that
\begin{align}
    \lnorm{v_h}{2}{(\Omega)}{0.55}{-0.5mm}&\leq C\,\lnorm{\gradh v_h}{2}{(\Omega)^d}{0.55}{-0.5mm} \quad \text{for all $h\in(0,H]$ and $v_h\in V_{0,h}$} \label{sobolev_C}\\[2mm]
    \text{and} \qquad \lnorm{v}{2}{(\Omega)}{0.55}{-0.5mm} &\leq C\,\lnorm{\gradx v}{2}{(\Omega)^d}{0.55}{-0.5mm} \quad \text{for all $v\in W^{1,2}_0(\Omega)$.} \label{poincare_C}
\end{align}
Clearly, such a constant exists due to (\ref{sobolev}) and the usual Poincaré inequality.
Due to (\ref{poincare_C}), we observe that
\begin{align*}
    \inttauinto \omean{\mathcal{V};|\bfatu-\bm{U}|^2}\;\dxdt &\leq 2\inttauinto \omean{\mathcal{V};|\bfatu-\fatu_{\mathcal{V}}|^2}\;\dxdt + 2\inttauinto |\fatu_{\mathcal{V}}-\bm{U}|^2\;\dxdt \notag \\[2mm]
    & \leq 2\inttauinto \omean{\mathcal{V};|\bfatu-\fatu_{\mathcal{V}}|^2}\;\dxdt + 2C^{\,2}\inttauinto |\gradx(\fatu_{\mathcal{V}}-\bm{U})|^2\;\dxdt\,.
\end{align*}
Using the monotone convergence theorem, Lemma \ref{lem_aux_piv}, Lemma \ref{lem_conv_piq_piv}(\hyperlink{lem_conv_piq_piv_ii}{ii}), (\hyperlink{lem_conv_piq_piv_iii}{iii}), the first estimate in (\ref{est_vh_1}), the first estimate in (\ref{energy_est_2}), and (\ref{conv_3}), we deduce that
\begin{align}
    & \inttauinto \omean{\mathcal{V};|\bfatu-\fatu_{\mathcal{V}}|^2}\;\dxdt = \lim_{k\,\to\,\infty}\inttauinto \omean{\mathcal{V};|\bfatu-\fatu_{\mathcal{V}}|^2\vartheta_k(\bvarrho,\btheta,\bfatu)}\;\dxdt \notag \\[2mm]
    & \quad = \lim_{k\,\to\,\infty}\left(\,\lim_{h\,\downarrow\,0}\inttauinto |\overline{\fatu_h}-\fatu_{\mathcal{V}}|^2\vartheta_k(\varrho_h,\thetah,\overline{\fatu_h})\;\dxdt\right) \leq \liminf_{h\,\downarrow\,0}\inttauinto |\overline{\fatu_h}-\fatu_{\mathcal{V}}|^2\;\dxdt \notag \\[2mm]
    & \quad \leq 3\liminf_{h\,\downarrow\,0}\inttauinto \left(|\overline{\fatu_h}-\fatu_h|^2 + |\fatu_h-\pivh\fatu_{\mathcal{V}}|^2 + |\pivh\fatu_{\mathcal{V}}-\fatu_{\mathcal{V}}|^2\right)\dxdt \notag \\[2mm]
    & \quad \leq 12C^{\,2}\liminf_{h\,\downarrow\,0}\inttauinto |\gradh\fatu_h-\gradh\pivh\fatu_{\mathcal{V}}|^2\;\dxdt \notag \\[2mm]
    & \quad \leq 24C^{\,2}\liminf_{h\,\downarrow\,0}\inttauinto \left(|\gradh\fatu_h-\gradx\fatu_{\mathcal{V}}|^2 + |\gradx\fatu_{\mathcal{V}}-\gradh\pivh\fatu_{\mathcal{V}}|^2\right)\dxdt \notag \\[2mm]
    & \quad = 24C^{\,2}\left(\liminf_{h\,\downarrow\,0}\inttauinto |\gradh\fatu_h|^2\;\dxdt - \inttauinto |\gradx\fatu_{\mathcal{V}}|^2\;\dxdt\right) \leq \frac{24C^{\,2}}{\mu} \int_{\,\overline{\Omega_\tau}}\dd\mathfrak{D} \label{conv_comp_poincare}
\end{align}
for almost all $\tau\in(0,T)$. Consequently, choosing $C_P=48C^{\,2}/\mu$ we obtain
\begin{align*}
    \inttauinto \omean{\mathcal{V};|\bfatu-\bm{U}|^2}\;\dxdt \leq C_P\left(\,\inttauinto |\gradx(\fatu_{\mathcal{V}}-\bm{U})|^2\;\dxdt + \int_{\,\overline{\Omega_\tau}}\dd\mathfrak{D}\right).
\end{align*}

\noindent \textbf{Entropy inequality.} \\[2mm]
Due to (\ref{conv_1}), (\ref{conv_2}), and (\ref{conv_init_c_z}), we may take the limit $h\downarrow 0$ in (\ref{cons_entropy_inequ}). We obtain
\begin{align}
    -\into \varrho_0\ln(\thetazero)\,\psi(0,\cdot)\;\dx \geq \inttinto \big[\omean{\mathcal{V};\bvarrho\hspace{0.3mm}\ln(\btheta)}\,\pt\psi + \omean{\mathcal{V};\bvarrho\hspace{0.3mm}\ln(\btheta)\bfatu}\cdot\gradx\psi\big]\,\dxdt \label{conv_entropy_inequ_1}
\end{align}
for all $\psi\in\Cinftyc([0,T)\times\overline{\Omega})$, $\psi\geq 0$. By an approximation argument similar to that in the case of the continuity equation, the validity of (\ref{conv_entropy_inequ_1}) can be extended to test functions $\psi\geq0$ of the class $C_c([0,T)\times\overline{\Omega})\cap W^{1,\infty}(\Omega_T)$. In particular, we may consider test functions of the form $\psi=\phi_{\tau,\overline{\delta}}\,\eta$, where $\eta\in W^{1,\infty}(\Omega_T)$, $\eta\geq 0$, $\tau\in(0,T)$, $\bardelta\in(0,\min\{T-\tau,\tau\})$, and $\phi_{\tau,\overline{\delta}}\in C_c([0,T))$,
\begin{align*}
    \phi_{\tau,\overline{\delta}}(t) = \left\{{\arraycolsep=0pt\begin{array}{cl}
        1 & \quad \text{if $t<\tau-\bardelta$,} \\[2mm]
        \tfrac{1}{2}\!\left(1+(\tau-t)/\bardelta\,\right) & \quad \text{if $t\in[\tau-\bardelta,\tau+\bardelta]$, and} \\[2mm]
        0 & \quad \text{if $t>\tau+\bardelta$.}
    \end{array}}\right.
\end{align*}
Consequently,
\begin{align*}
    &\frac{1}{2\bardelta\vphantom{\scalebox{1.1}{$\bardelta$}}}\int_{\tau-\overline{\delta}}^{\,\tau+\overline{\delta}}\!\!\into\omean{\mathcal{V};\bvarrho\hspace{0.3mm}\ln(\btheta)}\,\eta\;\dxdt - \into \varrho_0\ln(\thetazero)\,\eta(0,\cdot)\;\dx \\[2mm]
    &\qquad\qquad\qquad\qquad \leq \inttinto \phi_{\tau,\overline{\delta}}\big[\omean{\mathcal{V};\bvarrho\hspace{0.3mm}\ln(\btheta)}\,\pt\eta + \omean{\mathcal{V};\bvarrho\hspace{0.3mm}\ln(\btheta)\bfatu}\cdot\gradx\eta\big]\,\dxdt
\end{align*}
for all $\eta\in W^{1,\infty}(\Omega_T)$, $\eta\geq 0$, $\tau\in(0,T)$, $\bardelta\in(0,\min\{T-\tau,\tau\})$. The entropy inequality (\ref{entropy_inequ}) follows by performing the limit $\bardelta\downarrow 0$ in the above inequality. For the limit process, we rely on Lebesgue's differentiation theorem as well as the dominated convergence theorem. This completes the proof of Theorem~\ref{thm_main}.
\end{proof}

From the proof of Theorem \ref{thm_main} it follows that any Young measure generated by a sequence $\{(\varrho_h,\thetah,\overline{\fatu_h})|_{\Omega_T}\}_{h\,\downarrow\,0}$ obtained from a sequence of solutions $\{(\varrho_h,\thetah,\fatu_h)\}_{h\,\downarrow\,0}$ to our FE-FV method (\ref{r_method})--(\ref{m_method}) represents a DMV solution to the Navier-Stokes system with potential temperature transport (\ref{cequation})--(\ref{isen_press}). Moreover, 
$$
\varrho_h \rightharpoonup^\star \omean{\mathcal{V};\bvarrho\,}\  \text{ in } \Lp{\infty}{0.2mm}(0,T;\Lp{\gamma}{0.5mm}(\Omega)),   \
 \thetah \rightharpoonup^\star \omean{\mathcal{V}; \btheta} \  \text {in } \Lp{\infty}{0.2mm}(\Omega_T),\ \mbox{ and }\
\fatu_h \rightharpoonup \omean{\mathcal{V};\bfatu } \ \text{ in }  \Lp{2}{0mm}(\Omega_T)^d.
$$

If there is a strong solution to system (\ref{cequation})--(\ref{isen_press}) for given initial data $(\varrho_0,\thetazero,\fatu_0)$, then we may use the DMV-strong uniqueness result established in \citep{Lukacova_Schoemer} to strengthen the aforementioned convergence statement as follows.

\begin{theorem}\label{thm_strong_conv}
Let the assumptions of Theorem $\mathrm{\ref{thm_main}}$ be satisfied and suppose there is a strong solution $(\varrho,\theta,\fatu)$ to system $\mathrm{(\ref{cequation})}$--$\mathrm{(\ref{isen_press})}$ from the regularity class
\begin{align*}
    \varrho,\theta\in \Ck{1}(\overline{\Omega_T})\,, \quad \varrho,\theta>0\,, \quad \fatu\in\Ck{1}(\overline{\Omega_T})\cap \Lp{2}{0mm}(0,T;W^{2,\infty}(\Omega))\,, \quad \fatu|_{[0,T]\times\po}=\bm{0},
\end{align*}
emanating from the chosen initial data. Further, let $(\varrho_h,\thetah,\fatu_h)_{h\,\downarrow\,0}$ be a sequence of solutions to the FE-FV method $\mathrm{(\ref{r_method})}$--$\mathrm{(\ref{m_method})}$ starting from the corresponding discrete initial data defined in $\mathrm{(\ref{disc_int_1})}$ and suppose the parameters satisfy $\mathrm{(\ref{cons_cond})}$. Let $p\in[1,\infty)$ and $q\in[1,2)$ be arbitrary. Then
\begin{align*}
    \varrho_h \rightarrow \varrho \;\;\text{in $\Lp{\gamma}{0mm}(\Omega_T)$,} \quad \thetah\rightarrow\theta \;\; \text{in $\Lp{p}{0.5mm}(\Omega_T)$,} \quad \text{and} \quad \fatu_h\rightarrow\fatu \;\; \text{in $\Lp{q}{0.5mm}(\Omega_T)^d$} \qquad \text{as $h\downarrow 0$.}
\end{align*}
\end{theorem}

\begin{proof}
Let $(\varrho_h,\thetah,\fatu_h)_{h\,\downarrow\,0}$ be a sequence as described above. To prove Theorem \ref{thm_strong_conv}, it suffices to show that every subsequence $(\varrho_{h_\star},\theta_{h_\star},\fatu_{h_\star})_{h_\star\,\downarrow\,0}$ of $(\varrho_h,\thetah,\fatu_h)_{h\,\downarrow\,0}$ possesses a subsequence $(\varrho_{h^{\prime}},\theta_{h^{\prime}},\fatu_{h^{\prime}})_{h^{\prime}\,\downarrow\,0}$ such that
\begin{align*}
    \varrho_{h^{\prime}} \rightarrow \varrho \;\;\text{in $\Lp{\gamma}{0mm}(\Omega_T)$,} \quad \theta_{h^{\prime}}\rightarrow\theta \;\; \text{in $\Lp{p}{0.5mm}(\Omega_T)$,} \quad \text{and} \quad \fatu_{h^{\prime}}\rightarrow\fatu \;\; \text{in $\Lp{q}{0.5mm}(\Omega_T)^d$}
\end{align*}
as ${h^{\prime}}\downarrow 0$. Thus, let $(\varrho_{h_\star},\theta_{h_\star},\fatu_{h_\star})_{h_\star\,\downarrow\,0}$ be an arbitrary subsequence of $(\varrho_h,\thetah,\fatu_h)_{h\,\downarrow\,0}$. From the proof of Theorem \ref{thm_main} and the DMV-strong uniqueness principle established in \citep{Lukacova_Schoemer} we deduce that there is a subsequence $(\varrho_{h^{\prime}},\theta_{h^{\prime}},\fatu_{h^{\prime}})_{{h^{\prime}}\,\downarrow\,0}$ of $(\varrho_{h_\star},\theta_{h_\star},\fatu_{h_\star})_{h_\star\,\downarrow\,0}$ such that
\begin{gather*}
    \varrho_{h^{\prime}}\theta_{h^{\prime}}^m \rightharpoonup \varrho\theta^m \;\;\text{in $\Lp{\gamma}{0mm}(\Omega_T)$ for all $m\in\naturals_0$,} \qquad \lnorm{\varrho_{h^{\prime}}|\overline{\fatu_{h^{\prime}}}|^2}{1}{(\Omega_T)}{0.55}{-0.5mm} \to \lnorm{\varrho|\fatu|^2}{1}{(\Omega_T)}{0.55}{-0.5mm}\,,  \\[2mm]
    \varrho_{h^{\prime}}\overline{\fatu_{h^{\prime}}} \rightharpoonup \varrho\fatu \;\;\text{in $\Lp{\tfrac{2\gamma}{\gamma+1}}{0mm}(\Omega_T)^d$,} \qquad \fatu_{h^{\prime}} \rightharpoonup \fatu \;\;\text{in $\Lp{2}{0mm}(\Omega_T)^d$} 
    \qquad \text{and} \qquad \varrho_{h^{\prime}}\theta_{h^{\prime}} \to \varrho\theta \;\;\text{in $\Lp{\gamma}{0mm}(\Omega_T)$}
\end{gather*}
as ${h^{\prime}}\downarrow 0$. Consequently,
\begin{align*}
    \lnorm{\varrho_{h^{\prime}}-\varrho}{1}{(\Omega_T)}{0.55}{-0.5mm} &\lesssim \lnorm{\varrho_{h^{\prime}}\theta-\varrho\theta}{1}{(\Omega_T)}{0.55}{-0.5mm} \lesssim \lnorm{\varrho_{h^{\prime}}\theta_{h^{\prime}}-\varrho\theta}{1}{(\Omega_T)}{0.55}{-0.5mm} + \lnorm{\varrho_{h^{\prime}}\theta-\varrho_{h^{\prime}}\theta_{h^{\prime}}}{1}{(\Omega_T)}{0.55}{-0.5mm} \notag \\[2mm]
    &\lesssim \lnorm{\varrho_{h^{\prime}}\theta_{h^{\prime}}-\varrho\theta}{\gamma}{(\Omega_T)}{0.55}{-0.5mm} + \left(\lnorm{\varrho_{h^{\prime}}}{1}{(\Omega_T)}{0.55}{-0.5mm}\,\lnorm{\varrho_{h^{\prime}}(\theta-\theta_{h^{\prime}})^2}{1}{(\Omega_T)}{0.55}{-0.5mm}\right)^{1/2} \notag \\[2mm]
    &\lesssim \lnorm{\varrho_{h^{\prime}}\theta_{h^{\prime}}-\varrho\theta}{\gamma}{(\Omega_T)}{0.55}{-0.5mm} + \explnorm{\varrho_{h^{\prime}}(\theta-\theta_{h^{\prime}})^2}{1}{(\Omega_T)}{0.55}{-0.5mm}{0.5mm}{1/2} \xrightarrow{\,{h^{\prime}}\,\downarrow\,0\,} 0\,,
\end{align*}
i.e., $\varrho_{h^{\prime}}\to\varrho$ in $\Lp{1}{0mm}(\Omega_T)$ as ${h^{\prime}}\downarrow 0$. Therefore,
\begin{align*}
    \lnorm{(\theta_{h^{\prime}}-\theta)^{2m}}{1}{(\Omega_T)}{0.55}{-0.5mm} &\lesssim \lnorm{\mathds{1}_{\{\varrho_{h^{\prime}}\,\leq\, \underline{\varrho}/2\}}(\theta_{h^{\prime}}-\theta)^{2m}}{1}{(\Omega_T)}{0.55}{-0.5mm} + \lnorm{\varrho_{h^{\prime}}(\theta_{h^{\prime}}-\theta)^{2m}}{1}{(\Omega_T)}{0.55}{-0.5mm} \\[2mm]
    &\lesssim \lnorm{\mathds{1}_{\{\varrho_{h^{\prime}}\,\leq\, \underline{\varrho}/2\}}}{1}{(\Omega_T)}{0.55}{-0.5mm} + \lnorm{\varrho_{h^{\prime}}(\theta_{h^{\prime}}-\theta)^{2m}}{1}{(\Omega_T)}{0.55}{-0.5mm} \xrightarrow{\,{h^{\prime}}\,\downarrow\,0\,} 0
\end{align*}
for all $m\in\naturals$, where $\underline{\varrho}=\inf_{(t,\fatx)\,\in\,\Omega_T}\{\varrho(t,\fatx)\}>0$. That is, $\theta_{h^{\prime}}\to\theta$ in $\Lp{p}{0.5mm}(\Omega_T)$ as ${h^{\prime}}\downarrow 0$. This in turn implies
\begin{align*}
    \lnorm{\varrho_{h^{\prime}}-\varrho}{\gamma}{(\Omega_T)}{0.55}{-0mm} &\lesssim \lnorm{\varrho_{h^{\prime}}\theta_{h^{\prime}}-\varrho\theta_{h^{\prime}}}{\gamma}{(\Omega_T)}{0.55}{-0mm} \lesssim \lnorm{\varrho_{h^{\prime}}\theta_{h^{\prime}}-\varrho\theta}{\gamma}{(\Omega_T)}{0.55}{-0mm} + \lnorm{\varrho\theta-\varrho\theta_{h^{\prime}}}{\gamma}{(\Omega_T)}{0.55}{-0mm} \notag \\[2mm]
    &\lesssim \lnorm{\varrho_{h^{\prime}}\theta_{h^{\prime}}-\varrho\theta}{\gamma}{(\Omega_T)}{0.55}{-0mm} + \lnorm{\theta-\theta_{h^{\prime}}}{\gamma}{(\Omega_T)}{0.55}{-0mm} \xrightarrow{\,{h^{\prime}}\,\downarrow\,0\,} 0\,,
\end{align*}
i.e., $\varrho_{h^{\prime}}\to\varrho$ in $\Lp{\gamma}{0.2mm}(\Omega_T)$ as ${h^{\prime}}\downarrow 0$. 
Finally, if $q\in[1,2)$, then
\begin{align*}
    \lnorm{|\fatu_{h^{\prime}}-\fatu|^{q}}{1}{(\Omega_T)}{0.55}{-0.5mm} &\lesssim \lnorm{|\fatu_{h^{\prime}}-\overline{\fatu_{h^{\prime}}}|^{q}}{1}{(\Omega_T)}{0.55}{-0.5mm} + \lnorm{|\overline{\fatu_{h^{\prime}}}-\fatu|^{q}}{1}{(\Omega_T)}{0.55}{-0.5mm} \\[2mm]
    &\lesssim \explnorm{\fatu_{h^{\prime}}-\overline{\fatu_{h^{\prime}}}}{2}{(\Omega_T)^d}{0.55}{-0.5mm}{0.5mm}{\hspace{0.3mm}q} + \lnorm{|\overline{\fatu_{h^{\prime}}}-\fatu|^{q}}{1}{(\Omega_T)}{0.55}{-0.5mm} \\[2mm]
    &\lesssim (h^{\prime})^q\,\explnorm{\nabla_{\!h^\prime}\fatu_{h^{\prime}}}{2}{(\Omega_T)^{d\times d}}{0.55}{-0.5mm}{0.5mm}{q} + \lnorm{|\overline{\fatu_{h^{\prime}}}-\fatu|^{q}}{1}{(\Omega_T)}{0.55}{-0.5mm} \\[2mm]
    &\lesssim (h^{\prime})^q + \lnorm{\mathds{1}_{\{\varrho_{h^{\prime}}\,\leq\, \underline{\varrho}/2\}}|\overline{\fatu_{h^{\prime}}}-\fatu|^{q}}{1}{(\Omega_T)}{0.55}{-0.5mm} + \lnorm{\varrho_{h^{\prime}}^{q/2}|\overline{\fatu_{h^{\prime}}}-\fatu|^{q}}{1}{(\Omega_T)}{0.55}{-0.5mm} \\[2mm]
    &\lesssim (h^{\prime})^q + \Lnorm{\mathds{1}_{\{\varrho_{h^{\prime}}\,\leq\, \underline{\varrho}/2\}}}{\frac{2}{2-q}}{(\Omega_T)}{0.75}{-0.5mm}\explnorm{\overline{\fatu_{h^{\prime}}}-\fatu}{2}{(\Omega_T)^d}{0.55}{-0.5mm}{0.5mm}{\hspace{0.3mm}q} + \explnorm{\varrho_{h^{\prime}}|\overline{\fatu_{h^{\prime}}}-\fatu|^{2}}{1}{(\Omega_T)}{0.55}{-0.5mm}{0.5mm}{\hspace{0.3mm}q/2} \\[2mm]
    &\phantom{\leq\;} \xrightarrow{\,{h^{\prime}}\,\downarrow\,0\,} 0\,, 
\end{align*}
i.e., $\fatu_{h^{\prime}}\to\fatu$ in $\Lp{q}{0.5mm}(\Omega_T)^d$ as ${h^{\prime}}\downarrow 0$. 
\end{proof}

\section{Conclusions}\label{sec_conlusions}
In the present paper, we introduced DMV solutions to the Navier-Stokes system with potential temperature transport (\ref{cequation})--(\ref{isen_press}) and proved their existence. For the existence proof we examined the convergence properties of solutions to a mixed FE-FV method that is a generalization of the method developed for the barotropic Navier-Stokes equations; see  \citep{Feireisl_Lukacova}, \citep[Chapter 13]{Feireisl_Lukacova_Book}, \citep[Chapter 7]{Feireisl_Karper_Pokorny}. In particular, we showed that any Young measure generated by a sequence $\{(\varrho_h,\thetah,\overline{\fatu_h})|_{\Omega_T}\}_{h\,\downarrow\,0}$ obtained from a sequence of solutions $\{(\varrho_h,\thetah,\fatu_h)\}_{h\,\downarrow\,0}$ to our FE-FV method (\ref{r_method})--(\ref{m_method}) represents a DMV solution to the Navier-Stokes system with potential temperature transport (\ref{cequation})--(\ref{isen_press}).

In order to ensure the validity of our existence result for all physically relevant values of the adiabatic index $\gamma$ -- that is, $\gamma\in(1,2]$ if $d=2$ and $\gamma\in(1,5/3]$ if $d=3$ -- we added two artificial pressure terms to our method. In the case of values of $\gamma$ close to $1$, these terms provided us with sufficiently good stability estimates for the limit process. In the limit process itself, we profited from the generality of DMV solutions that allowed us to hide the terms arising from the artificial pressure terms in the energy concentration defect and the Reynolds concentration defect, respectively. The strategy of adding artificial pressure terms  points out a flexibility of the DMV concept. Indeed, it would  not work in the framework of weak solutions.

In spite of the generality of DMV solutions to system (\ref{cequation})--(\ref{isen_press}), we can show DMV-strong uniqueness, i.e., provided there is a strong solution, we can show that in a suitable sense any DMV solution on the same time interval coincides with it. We will present the detailed result in our upcoming paper \citep{Lukacova_Schoemer}. Here, we made use of this result to prove the strong convergence of the solutions to our FE-FV method to the strong solution of the system.

\medskip
\noindent \textbf{Conflict of interests:} On behalf of all authors, the corresponding author states that there is no conflict of interest.
\medskip
\bibliographystyle{unsrt}
\bibliography{Bibliography_new}

\appendix
\section{Appendix}\label{sec_appendix}
\subsection{Mesh-related estimates}\label{subsec_mesh_est}
We summarize several important mesh-related estimates; see, e.g., \citep{Feireisl_Lukacova_Book} and the references therein.
We begin with the discrete trace and inverse inequalities. We have
\begin{align}
    \lnorm{r_K}{p}{(\sigma)}{0.55}{-0.0mm}\lesssim h^{-1/p}\,\lnorm{r}{p}{(K)}{0.55}{-0.0mm} \qquad \text{and} \qquad \lnorm{r}{p}{(\Omega_h)}{0.55}{-0.0mm}\lesssim h^{d\left(\frac{1}{p}-\frac{1}{q}\right)}\,\lnorm{r}{q}{(\Omega_h)}{0.55}{-0.0mm} \label{trace_neg_lp}
\end{align}
for all $r\in Q_h$, all $K\in\mathcal{T}_h$, all $\sigma\in\mathcal{E}_h(K)$, and all $1\leq q\leq p\leq\infty$. In addition,
\begin{gather}
    \lnorm{v-\overline{v}}{p}{(K)}{0.55}{-0.0mm}\lesssim h\,\lnorm{\gradh v}{p}{(K)^d}{0.55}{-0.0mm}\,, \qquad \lnorm{v-\smallgmean{v}}{p}{(\sigma)}{0.55}{-0.0mm}\lesssim h\,\lnorm{\gradh v}{p}{(\sigma)^d}{0.55}{-0.0mm}\,, \label{est_vh_1} \\[2mm]
    \text{and} \qquad \lnorm{\smallgmean{v}}{p}{(\sigma)}{0.55}{-0.0mm}\lesssim h^{-1/p}\left(\lnorm{v}{p}{(K)}{0.55}{-0.0mm}+h\,\lnorm{\gradh v}{p}{(K)^d}{0.55}{-0.0mm}\right) \label{est_vh_2}
\end{gather}
are valid for all $p\in[1,\infty]$, all $v\in V_{0,h}$, all $K\in\mathcal{T}_h$, and all $\sigma\in\mathcal{E}_h(K)$. Moreover, given $\phi\in\Ck{1}(\overline{\Omega})$, an application of Taylor's theorem yields
\begin{align}
    \Lnorm{\Jump{\,\overline{\phi}\,}}{\infty}{(\sigma)}{0.55}{-0.0mm} & \lesssim h\,\maxnormk{\phi}{1}{(\overline{\Omega})} \qquad \text{for all $\sigma\in\mathcal{E}_{h,\mathrm{int}}$,} \label{est_piqh_1} \\[2mm]
    \Lnorm{\phi-\overline{\phi}_K}{\infty}{(\sigma)}{0.55}{-0.0mm} & \lesssim h\,\maxnormk{\phi}{1}{(\overline{\Omega})} \qquad \text{for all $K\in\mathcal{T}_h$ and all $\sigma\in\mathcal{E}_h(K)$,} \label{est_piqh_2} \\[3mm]
    \qquad \Lnorm{\phi-\smallgmean{\phi}}{\infty}{(\sigma)}{0.55}{-0.0mm} & \lesssim h\,\maxnormk{\phi}{1}{(\overline{\Omega})} \qquad \text{for all $K\in\mathcal{T}_h$ and all $\sigma\in\mathcal{E}_h(K)$,} \label{est_piqh_3} \\[2mm]
    \Lnorm{\Jump{\,\overline{\pivh\phi}\,}}{\infty}{(\sigma)}{0.55}{-0.0mm} & \lesssim h\,\maxnormk{\phi}{1}{(\overline{\Omega})} \qquad \text{for all $\sigma\in\mathcal{E}_{h,\mathrm{int}}$,} \label{est_pivh_3} \\[2mm]
    \Lnorm{\phi-\overline{\pivh\phi}_K}{\infty}{(\sigma)}{0.55}{-0.0mm} & \lesssim h\,\maxnormk{\phi}{1}{(\overline{\Omega})} \qquad \text{for all $K\in\mathcal{T}_h$ and all $\sigma\in\mathcal{E}_h(K)$,} \label{est_pivh_4} \\[3mm]
    \text{and} \qquad \Lnorm{\phi-\overline{\pivh\phi}}{\infty}{(\Omega_h)}{0.55}{-0.0mm} & \lesssim h\,\maxnormk{\phi}{1}{(\overline{\Omega})}\,. \label{est_pivh_5}
\end{align}
Next, combining \citep[Theorem 6.1]{Pietro_Ern} with \citep[Lemma 2.2]{Gallouet_Herbin_Latche} we obtain a \textit{discrete version of Poincaré's inequality}, namely
\begin{align}
    \lnorm{v}{q}{(\Omega_h)}{0.55}{-0.0mm} \lesssim \lnorm{\gradh v}{2}{(\Omega_h)^d}{0.55}{-0.5mm} \label{sobolev}
\end{align}
for all $v\in V_{0,h}$, where $q\in[1,\infty)$ if $d=2$ and $q\in[1,6]$ if $d=3$. Due to \citep[Theorem 5]{Ciarlet_Raviart}, we have the following estimates for the projection operators $\piqh$ and $\pivh$:
\begin{align}
    \lnorm{\phi-\overline{\phi}}{q}{(\Omega_h)}{0.55}{0.0mm} \equiv \lnorm{\phi-\piqh\phi}{q}{(\Omega_h)}{0.55}{0.0mm} &\lesssim h\,\wnorm{\phi}{1,q}{(\Omega)}{0.55}{-0.5mm}\,, \label{est_piqh_4} \\[2mm]
    \lnorm{\phi-\pivh\phi}{q}{(\Omega_h)}{0.55}{0.0mm} + h\,\lnorm{\gradx \phi-\gradh\pivh\phi}{q}{(\Omega_h)^d}{0.55}{0.0mm} &\lesssim h\,\wnorm{\phi}{1,q}{(\Omega)}{0.55}{-0.5mm}\,, \label{est_pivh_1} \\[2mm]
    \lnorm{\psi-\pivh\psi}{q}{(\Omega_h)}{0.55}{0.0mm} + h\,\lnorm{\gradx \psi-\gradh\pivh\psi}{q}{(\Omega_h)^d}{0.55}{0.0mm} &\lesssim h^2\,\wnorm{\psi}{2,q}{(\Omega)}{0.55}{-0.5mm} \label{est_pivh_2}
\end{align}
for all $q\in[1,\infty]$, all $\phi\in W^{1,q}(\Omega)$, and all $\psi\in W^{2,q}(\Omega)$. The latter estimates are also known as the \textit{Crouzeix-Raviart estimates}. 

\begin{remark}\label{rem_proj_cont}
Clearly, the operators $\piqh$ and $\pivh$ are linear. Furthermore, we may use (\ref{est_pivh_1}) and the triangle inequality to deduce that
there exists an $h$-independent constant $C>0$ such that
\begin{align}
    \lnorm{\pivh v}{q}{(\Omega)}{0.55}{-0.5mm}\leq (1+Ch)\,\wnorm{v}{1,q}{(\Omega)}{0.55}{-0.5mm} \qquad \text{for all $v\in W^{1,q}(\Omega)$, $q\in[1,\infty]$.}
\end{align}
Consequently, $\pivh$ is continuous. The continuity of $\piqh$ is a consequence of Jensen's inequality which yields (cf. \citep[p.90]{Feireisl_Karper_Pokorny})
\begin{align}
    \lnorm{\piqh v}{q}{(\Omega)}{0.55}{-0.5mm}\leq \lnorm{v}{q}{(\Omega)}{0.55}{-0.5mm} \qquad \text{for all $v\in \Lp{q}{0mm}(\Omega)$, $q\in[1,\infty]$.} \label{bound_piq}
\end{align}
\end{remark}

\subsection{Auxiliary results for the projections \texorpdfstring{$\piqh$}{\textbackslash Pi\_\{Q,h\}} and \texorpdfstring{$\pivh$}{\textbackslash Pi\_\{V,h\}}}
This section is devoted to some important auxiliary results concerning the projections $\piqh$ and $\pivh$. Combining suitable density arguments with the estimates (\ref{est_piqh_4})--(\ref{est_pivh_2}), one easily establishes the subsequent lemma.
\begin{lemma}\label{lem_conv_piq_piv}
Let $p\in[1,\infty)$ be arbitrary. Then \hypertarget{lem_conv_piq_piv_i}{}
\begin{enumerate}
    \item[$\mathrm{(i)}$]{$\lnorm{\piqh v-v}{p}{(\Omega)}{0.55}{0mm} \xrightarrow{\,h\,\downarrow\,0\,} 0 \quad \text{for all $v\in\Lp{p}{0.5mm}(\Omega)$,}$ \hypertarget{lem_conv_piq_piv_ii}{}}
    \item[$\mathrm{(ii)}$]{$\lnorm{\pivh v-v}{p}{(\Omega)}{0.55}{0mm} \xrightarrow{\,h\,\downarrow\,0\,} 0 \quad \text{for all $v\in W^{1,p}(\Omega)$, and}$ \hypertarget{lem_conv_piq_piv_iii}{}}
    \item[$\mathrm{(iii)}$]{$\lnorm{\gradh\pivh v-\gradx v}{p}{(\Omega)^d}{0.55}{0mm} \xrightarrow{\,h\,\downarrow\,0\,} 0 \quad \text{for all $v\in W^{1,p}(\Omega)$.}$}
\end{enumerate}
\end{lemma}

Next, we prove the following auxiliary result that is needed in the proof of Theorem \ref{thm_main}.

\begin{lemma}\label{lem_aux_piv}
Let $\tau\in[0,T]$ be arbitrary. Further, let $\fatu_{\mathcal{V}}$, $\{\fatu_h\}_{h\,\downarrow\,0}$, and $C$ be the same as in the proof of Theorem $\mathrm{\ref{thm_main}}$. Then
\begin{align*}
    \liminf_{h\,\downarrow\,0}\left(\,\inttinto |\fatu_h-\pivh\fatu_{\mathcal{V}}|^2\;\dxdt\right) \leq 4C^{\,2} \liminf_{h\,\downarrow\,0}\left(\,\inttinto |\gradh\fatu_h-\gradh\pivh\fatu_{\mathcal{V}}|^2\;\dxdt\right)\,.
\end{align*}
\end{lemma}

\begin{proof}
Let $\barepsilon>0$ be arbitrary. Further, let $\bm{\phi}\in\Lp{2}{0mm}(0,T;\Cinftyc(\Omega)^d)\subset \Lp{2}{0mm}(0,T;W^{1,2}_0(\Omega)^d)$ be a function satisfying
\begin{align*}
    4C^{\,2}\,\explnorm{\bm{\phi}-\fatu_{\mathcal{V}}}{2}{(0,T;W^{1,2}_0(\Omega)^d)}{0.55}{-0.5mm}{0.5mm}{2} \leq \barepsilon\,.
\end{align*}
Using (\ref{sobolev_C}), we deduce that
\begin{align*}
    & \inttinto |\fatu_h-\pivh\fatu_{\mathcal{V}}|^2\;\dxdt \leq 2\inttinto \left(|\fatu_h-\pivh\bm{\phi}|^2 + |\pivh(\bm{\phi}-\fatu_{\mathcal{V}})|^2\right)\dxdt \notag \\[2mm]
    & \leq 2C^{\,2}\inttinto \left(|\gradh\fatu_h-\gradh\pivh\bm{\phi}|^2 + |\pivh(\bm{\phi}-\fatu_{\mathcal{V}})|^2\right)\dxdt \notag \\[2mm]
    & \leq 4C^{\,2}\inttinto \left(|\gradh\fatu_h-\gradh\pivh\fatu_{\mathcal{V}}|^2 + |\gradh\pivh(\fatu_{\mathcal{V}}-\bm{\phi})|^2 + |\pivh(\bm{\phi}-\fatu_{\mathcal{V}})|^2\right)\dxdt\,,
\end{align*}
provided $h$ is sufficiently small. Therefore, an application of Lemma \ref{lem_conv_piq_piv}(\hyperlink{lem_conv_piq_piv_ii}{ii}), (\hyperlink{lem_conv_piq_piv_iii}{iii}) yields
\begin{align*}
    &\liminf_{h\,\downarrow\,0}\left(\,\inttinto |\fatu_h-\pivh\fatu_{\mathcal{V}}|^2\;\dxdt\right) \notag \\[2mm]
    & \quad \leq 4C^{\,2} \left[\liminf_{h\,\downarrow\,0}\left(\,\inttinto |\gradh\fatu_h-\gradh\pivh\fatu_{\mathcal{V}}|^2\;\dxdt\right) + \explnorm{\bm{\phi}-\fatu_{\mathcal{V}}}{2}{(0,T;W^{1,2}_0(\Omega)^d)}{0.55}{-0.5mm}{0.5mm}{2}\right] \notag \\[2mm]
    & \quad \leq 4C^{\,2} \liminf_{h\,\downarrow\,0}\left(\,\inttinto |\gradh\fatu_h-\gradh\pivh\fatu_{\mathcal{V}}|^2\;\dxdt\right) + \barepsilon\,.
\end{align*}
Since $\barepsilon>0$ was chosen arbitrarily, the desired result follows.
\end{proof}

\subsection{Properties of the numerical scheme}\label{app_solv}
In this section, we present a proof of Lemma \ref{lem_sol} that is based on the following lemma.

\begin{lemma}[{\citep[Theorem A.1]{Gallouet_Maltese_Novotny}}]\label{lem_fix}
Let $M,N$ be natural numbers, $C_1>\alpha>0$ and $C_2>0$ real numbers, and
\begin{align*}
    V &= \big\{(\fatr,\fatv)\in\reals^N\times\reals^M\,\big|\,\fatr > 0 \;\; \mathrm{componentwise}\big\}\,, \\[2mm]
    W &= \big\{(\fatr,\fatv)\in\reals^N\times\reals^M\,\big|\,\alpha < \fatr < C_1 \;\;\mathrm{componentwise \; and} \;\; |\fatv| < C_2\big\}\,.
\end{align*}
Further, let $\bm{F}:V\times[0,1]\to\reals^N\times\reals^M$ be a continuous function that complies with the following conditions: \hypertarget{lem_i}{}
\begin{enumerate}
    \item[$\mathrm{(i)}$]{If $\fatf\in V$ satisfies $\bm{F}(\fatf,\zeta)=(\bm{0},\bm{0})$ for some $\zeta\in[0,1]$, then $\fatf\in W$.} \hypertarget{lem_ii}{}
    \item[$\mathrm{(ii)}$]{The equation $\bm{F}(\fatf,0)=(\bm{0},\bm{0})$ is a linear system with respect to $\fatf$ and admits a solution in $W$.}
\end{enumerate}
Then there is $\fatf\in W$ such that $\bm{F}(\fatf,1)=(\bm{0},\bm{0})$.
\end{lemma}



The proof of Lemma \ref{lem_sol} is done in two steps.
\begin{proof}[Proof of Lemma $\mathbf{\ref{lem_sol}}(\hyperlink{lem_sol_i}{\mathbf{i}})$]\label{proof_lem_sol_i}
We start by showing that, given
$(\varrho_h^{k-1},Z_h^{k-1},\fatu_h^{k-1})\in Q_h^+\times Q_h^+\times \bm{V}_{h}$, there is $(\varrho_h^{k},Z_h^{k},\fatu_h^{k})\in Q_h^+\times Q_h^+\times \bm{V}_{0,h}$ such that
\begin{align}
    \intoh (\Dt\varrho^k_h)\,\phi_h\;\dx - \int_{\mathcal{E}_{\mathrm{int}}}\up{\varrho^k_h}{\fatu_h^k}\jump{\phi_h}\;\dsx &= 0\,, \label{r_method_z} \\[2mm]
    \intoh (\Dt Z^k_h)\,\phi_h\;\dx - \int_{\mathcal{E}_{\mathrm{int}}}\up{Z^k_h}{\fatu_h^k}\jump{\phi_h}\;\dsx &= 0\,, \label{z_method_z} \\[2mm]
    \intoh \Dt\!\left(\varrho_h^k\overline{\fatu_h^k}\,\right)\cdot\bm{\phi}_h\;\dx - \int_{\mathcal{E}_{\mathrm{int}}}\up{\varrho^k_h\overline{\fatu_h^k}}{\fatu_h^k}\cdot\Jump{\overline{\bm{\phi}_h}\vphantom{\overline{\overline{\phi}}}}\dsx + \mu\intoh \gradh\fatu_h^k:\gradh\bm{\phi}_h\;\dx & \phantom{.} \notag \\[2mm]
    {}+\nu\intoh \divh(\fatu_h^k)\,\divh(\bm{\phi}_h)\;\dx -\intoh \left(\p(Z^k_h)+\hd\!\left[(\varrho^k_h)^2+(Z^k_h)^2\right]\right)\divh(\bm{\phi}_h)\;\dx &= 0 \label{m_method_z}
\end{align}
for all $\phi_h\in Q_h$ and $\bm{\phi}_h\in\bm{V}_{0,h}$. The proof of this fact is essentially identical to that of \citep[Lemma 11.3]{Feireisl_Lukacova_Book}. In order to be able to apply Lemma \ref{lem_fix}, we set
\begin{align*}
    V &= \big\{((\varrho_h^k,Z_h^k),\fatu_h^k)\in Q_h^2\times \bm{V}_{0,h}\,\big|\,\varrho_h^k,Z_h^k>0 \;\;\text{in $\Omega_h$}\big\} \\[2mm]
    \text{and} \qquad  W &= \big\{((\varrho_h^k,Z_h^k),\fatu_h^k)\in Q_h^2\times \bm{V}_{0,h}\,\big|\,\alpha<\varrho_h^k,Z_h^k<C_1 \;\;\text{in $\Omega_h$ and}\;\; ||\fatu^k_h||<C_2\big\}\,,
\end{align*}
where $||\fatu^k_h||\equiv\lnorm{\gradh\fatu^k_h}{2}{(\Omega_h)^{d\times d}}{0.55}{-0.5mm}$ and the numbers $\alpha,C_1,C_2$ are yet to be determined. Clearly, we can construe $Q_h^2$ as a subset of $\reals^{2N}$ and $\bm{V}_{0,h}$ as a subset of $\reals^{dM}$, where $N$ is the number of tetrahedra (triangles) and $M$ the number of inner faces (edges) of the mesh $\mathcal{T}_h$. Next, we define the continuous map
\begin{align*}
    \bm{F}:V\times[0,1]\to Q_h^2\times\bm{V}_{0,h} \qquad \text{via} \qquad (((\varrho_h^k,Z_h^k),\fatu^k_h),\zeta)\mapsto ((\varrho^\star,Z^\star),\fatu^\star)\,,
\end{align*}
where $((\varrho^\star,Z^\star),\fatu^\star)$ is the uniquely determined element of $Q_h^2\times\bm{V}_{0,h}$ satisfying
\begin{align*}
    \intoh \varrho^\star\phi_h\;\dx &= \intoh (\Dt\varrho^k_h)\,\phi_h\;\dx - \zeta\inteint \up{\varrho^k_h}{\fatu^k_h}\,\jump{\phi_h}\;\dsx\,,  \\[2mm]
    \intoh Z^\star\phi_h\;\dx &= \intoh (\Dt Z_h^k)\,\phi_h\;\dx - \zeta\inteint \up{Z^k_h}{\fatu_h^k}\,\jump{\phi_h}\;\dsx\,, \\[2mm]
    \intoh \fatu^\star\cdot\bm{\phi}_h\;\dx &= \mu\intoh\gradh\fatu_h^k:\gradh\bm{\phi}_h\;\dx + \zeta\,\nu\intoh\divh(\fatu_h^k)\,\divh(\bm{\phi}_h)\;\dx \notag \\[2mm]
    &\phantom{\;=\;} + \intoh\Dt\big(\varrho_h^k\overline{\fatu_h^k}\,\big)\cdot\bm{\phi}_h\;\dx - \zeta\inteint\up{\varrho^k_h\overline{\fatu_h^k}}{\fatu_h^k}\cdot\Jump{\overline{\bm{\phi}_h}\vphantom{\overline{\overline{\phi}}}}\dsx \notag \\[2mm]
    &\phantom{\;=\;}
     - \zeta\intoh\left[\p(Z^k_h)+\hd\big((\varrho_h^k)^2+(Z_h^k)^2\big)\right]\divh(\bm{\phi}_h)\;\dx
\end{align*}
for all $\phi_h\in Q_h$ and $\bm{\phi}_h\in\bm{V}_{0,h}$. To show that $\bm{F}$ satisfies assumption (\hyperlink{lem_i}{i}) of Lemma \ref{lem_fix}, we suppose that $((\varrho^k_h,Z_h^k),\fatu_h^k)\in V$ solves $\bm{F}(((\varrho_h^k,Z_h^k),\fatu_h^k),\zeta)=(\bm{0},\bm{0})$ for some $\zeta\in[0,1]$, i.e.
\begin{align}
    0 &= \intoh (\Dt\varrho^k_h)\,\phi_h\;\dx - \zeta\inteint \up{\varrho^k_h}{\fatu^k_h}\,\jump{\phi_h}\;\dsx\,, \label{z0c_method} \\[2mm]
    0 &= \intoh (\Dt Z_h^k)\,\phi_h\;\dx - \zeta\inteint \up{Z^k_h}{\fatu_h^k}\,\jump{\phi_h}\;\dsx\,, \label{z0z_method} \\[2mm]
    0 &= \mu\intoh\gradh\fatu_h^k:\gradh\bm{\phi}_h\;\dx + \zeta\,\nu\intoh\divh(\fatu_h^k)\,\divh(\bm{\phi}_h)\;\dx + \intoh\Dt\big(\varrho_h^k\overline{\fatu_h^k}\,\big)\cdot\bm{\phi}_h\;\dx \notag \\[2mm]
    &\phantom{\;=\;} - \zeta\inteint\up{\varrho^k_h\overline{\fatu_h^k}}{\fatu_h^k}\cdot\Jump{\overline{\bm{\phi}_h}\vphantom{\overline{\overline{\phi}}}}\dsx - \zeta\intoh\left[\p(Z^k_h)+\hd\big((\varrho_h^k)^2+(Z_h^k)^2\big)\right]\divh(\bm{\phi}_h)\;\dx \label{z0m_method}
\end{align}
for all $\phi_h\in Q_h$ and $\bm{\phi}_h\in\bm{V}_{0,h}$. Adapting and repeating the arguments from Section \ref{sec_stability} to derive the energy estimates,
we deduce that
\begin{align}
    ||\fatu_h^k|| < C_2 \equiv C_2(\varrho_h^{k-1},Z_h^{k-1},\fatu_h^{k-1})\,. \label{bound_grad}
\end{align}
Next, we choose $K\in\mathcal{T}_h$ such that $(\varrho_h^k)_K = \min_{R\,\in\,\mathcal{T}_h}\{(\varrho_h^k)_R\}$. Taking $\phi_h=\mathds{1}_K$ in (\ref{z0c_method}), leads to
\begin{align*}
    &|K|\big((\varrho_h^k)_K-(\varrho_h^{k-1})_K\big) \notag \\[2mm]
    & \quad = \zeta\deltat\sum_{\sigma\,\in\,\mathcal{E}_{\mathrm{int}}(K)}\intg \left((\varrho_h^k)^{\mathrm{out}}\big[\smallgmean{\fatu_h^k\cdot\ngamma}\big]^- + (\varrho_h^k)^{\mathrm{in}}\big[\smallgmean{\fatu_h^k\cdot\ngamma}\big]^+ -\frac{\he}{2}\,\jump{\varrho_h^k}\right)\jump{\mathds{1}_K}\;\dsx \notag \\[2mm]
    & \quad \geq \zeta\deltat\sum_{\sigma\,\in\,\mathcal{E}(K)}\intg\left(\big((\varrho_h^k)^{\mathrm{out}}-(\varrho_h^k)_K\big)\big[\smallgmean{\fatu_h^k\cdot\ngamma}\big]^- + \big((\varrho_h^k)^{\mathrm{in}}-(\varrho_h^k)_K\big)\big[\smallgmean{\fatu_h^k\cdot\ngamma}\big]^+\right)\jump{\mathds{1}_K}\;\dsx \notag \\[2mm]
    & \quad\phantom{\;=\;} + \zeta\deltat \sum_{\sigma\,\in\,\mathcal{E}(K)}\intg (\varrho_h^k)_K\,\smallgmean{\fatu_h^k\cdot\ngamma}\,\jump{\mathds{1}_K}\;\dsx \notag \\[2mm]
    & \quad \geq \zeta\deltat \sum_{\sigma\,\in\,\mathcal{E}(K)}\intg (\varrho_h^k)_K\,(\fatu_h^k\cdot\fatn_K)\;\dsx = -|K|\zeta\deltat \big(\varrho_h^k\,\divh(\fatu_h^k)\big)_{\!K} \geq -|K|\zeta\deltat \big(\varrho_h^k\,|\divh(\fatu^k_h)|\big)_{\!K}\,. 
\end{align*}
Consequently,
 $\varrho_h^k\geq (\varrho_h^k)_K\geq \frac{(\varrho_h^{k-1})_K}{1+\zeta\deltat\,|(\div_{h}(\fatu^k_h))_K|}$ in $\Omega_h$
and, similarly,
$Z_h^k\geq (Z_h^k)_L\geq \frac{(Z_h^{k-1})_L}{1+\zeta\deltat\,|(\div_h(\fatu_h^k))_L|} $ in $\Omega_h$,
where $L\in\mathcal{T}_h$ is chosen in such a way that $(Z_h^k)_L = \min_{R\,\in\,\mathcal{T}_h}\{(Z_h^k)_R\}$. In view of (\ref{bound_grad}), we can find a constant $\alpha\equiv\alpha(\varrho_h^{k-1},Z_h^{k-1},\fatu_h^{k-1})>0$ such that $\varrho_h^{k-1},Z_h^{k-1},\varrho_h^k,Z_h^k>\alpha$ in $\Omega_h$. Finally,
taking $\phi_h=\mathds{1}_{\overline{\Omega_h}}$ in (\ref{z0c_method}) yields
\begin{align*}
    \lnorm{\varrho_h^k}{1}{(\Omega_h)}{0.55}{-0.5mm} &= \intoh \varrho_h^k\;\dx = \intoh \varrho_h^{k-1}\;\dx \equiv M_{0,\varrho} > 0.
\end{align*}
Thus, we have
\begin{align*}
    \varrho_h^k \leq \frac{M_{0,\varrho}}{\min_{R\,\in\,\mathcal{T}_h}\{|R|\}}\quad \mbox{ in $\Omega_h$ and, analogously, }\quad
      Z_h^k \leq \frac{M_{0,Z}}{\min_{R\,\in\,\mathcal{T}_h}\{|R|\}}
    \quad \text{in $\Omega_h$.}
\end{align*}
Consequently, there is a constant $C_1\equiv C_1(\varrho_h^{k-1},Z_h^{k-1},\fatu_h^{k-1})>0$ such that $\varrho_h^{k-1},Z_h^{k-1},\varrho_h^k,Z_h^k<C_1$ in $\Omega_h$. Therefore, $\bm{F}$ fulfills assumption (\hyperlink{lem_i}{i}) of Lemma \ref{lem_fix}. We proceed by proving that $\bm{F}$ satisfies assumption (\hyperlink{lem_ii}{ii}) of Lemma \ref{lem_fix}. To this end, we consider the equation $\bm{F}(((\varrho_h^k,Z_h^k),\fatu_h^k),0)=(\bm{0},\bm{0})$ that can be written as
\begin{align*}
    (\varrho_h^k,Z_h^k) &= (\varrho_h^{k-1},Z_h^{k-1})\,, \\[2mm]
    0 &= \mu\intoh \gradh\fatu_h^k:\gradh\bm{\phi}_h\;\dx + \intoh \varrho_h^{k-1}\,\frac{\overline{\fatu_h^k}-\overline{\fatu_h^{k-1}}}{\deltat}\cdot\bm{\phi}_h\;\dx \qquad \text{for all $\bm{\phi}_h\in\bm{V}_{0,h}$.}
\end{align*}
Obviously, this is a linear system for $((\varrho_h^k,Z_h^k),\fatu_h^k)$ with 
a positive definite matrix. Thus, the equation $\bm{F}(((\varrho_h^k,Z_h^k),\fatu_h^k),0)=(\bm{0},\bm{0})$ has a unique solution.
Therefore, $\bm{F}$ also satisfies assumption (\hyperlink{lem_ii}{ii}) of Lemma \ref{lem_fix} and the existence of a solution $(\varrho_h^k,Z_h^k,\fatu_h^k)\in Q_h^+\times Q_h^+\times \bm{V}_{0,h}$ to (\ref{r_method_z})--(\ref{m_method_z}) follows from Lemma \ref{lem_fix}.

Finally, given $(\varrho_h^{k-1},\thetakmh,\fatu_h^{k-1})\in Q_h^+\times Q_h^+\times \bm{V}_h$, we set $Z_h^{k-1}=\varrho_h^{k-1}\thetakmh\in Q_h^+$, find a solution $(\varrho_h^k,Z_h^k,\fatu_h^k)\in Q_h^+\times Q_h^+\times \bm{V}_{0,h}$ to (\ref{r_method_z})--(\ref{m_method_z}), and observe that
$(\varrho_h^k,\thetakh,\fatu_h^k)\in Q_h^+\times Q_h^+\times \bm{V}_{h,0}$, where
\begin{align*}
    (\thetakh)_R = \frac{(Z_h^k)_R}{(\varrho_h^k)_R} \quad \text{for all $R\in\mathcal{T}_h$,}
\end{align*}
is a solution to (\ref{r_method})--(\ref{m_method}).
\end{proof}

\begin{proof}[Proof of Lemma $\mathbf{\ref{lem_sol}}(\hyperlink{lem_sol_ii}{\mathbf{ii}})$]\label{proof_lem_sol_ii}
Suppose the triplet $(r_h^{k-1},r_h^k,\fatu_h^k)\in Q_h^+\times Q_h\times \bm{V}_{0,h}$ satisfies
\begin{align*}
    \intoh (\Dt r_h^k)\,\phi_h\;\dx - \inteint \up{r^k_h}{\fatu_h^k}\,\jump{\phi_h}\;\dsx = 0 \qquad \text{for all $\phi_h\in Q_h$.}
\end{align*}
Then \citep[Chapter 7.6, Lemma 6]{Feireisl_Karper_Pokorny} shows that $r^k_h\in Q_h^+$. The desired conclusions follow by applying this observation with
\begin{align*}
    (r_h^{k-1},r_h^k) &\in \big\{(\varrho_h^{k-1},\varrho_h^k),(\varrho_h^{k-1}\thetakmh-\underline{c}\varrho_h^{k-1},\varrho_h^k\thetakh-\underline{c}\varrho^k_h)\big\} \notag \\[2mm]
    &\quad \cup \big\{(\varrho_h^{k-1}\thetakmh,\varrho_h^k\thetakh),(\overline{c}\varrho_h^{k-1}-\varrho_h^{k-1}\thetakmh,\overline{c}\varrho_h^k-\varrho_h^k\thetakh)\big\}\,.
\end{align*}
\end{proof}

\subsection{Stability estimates}\label{app_stab}
The aim of this section is to provide the reader with a proof of Corollary \ref{cor_stab}.
\begin{proof}[Proof of Corollary $\mathbf{\ref{cor_stab}}$]
To begin with, we observe that $0\leq E_h^k\leq E_h^{k-1}$ for all $k\in\naturals$. This follows from the fact that the second term on the left-hand side of (\ref{disc_energy}) is nonnegative and all terms on the right-hand side are nonpositive. Here, the nonpositivity of the terms on the right-hand side is ensured by the convexity of the pressure potential $P$. Moreover, employing Hölder's inequality and Remark \ref{rem_proj_cont}, we see that
\begin{align*}
    E_h^0 &= \intoh \left[\,\frac{1}{2}\,\varrho_h^0|\overline{\fatu_h^0}|^2+P(\varrho_h^0\thetazeroh)+\hd\big((\varrho_h^0)^2+(\varrho_h^0\thetazeroh)^2\big)\right]\dx \notag \\[2mm]
    &\lesssim \lnorm{\varrho_h^0}{\infty}{(\Omega_h)}{0.55}{-0.0mm}\,\explnorm{\overline{\fatu_h^0}}{2}{(\Omega_h)^d}{0.55}{-0.5mm}{0.5mm}{2} + \explnorm{\varrho_h^0}{\infty}{(\Omega_h)}{0.55}{0.0mm}{0.5mm}{\gamma}\,\explnorm{\thetazeroh}{\gamma}{(\Omega_h)}{0.55}{0.0mm}{0.5mm}{\gamma} + \hd\big(\explnorm{\varrho_h^0}{2}{(\Omega_h)}{0.55}{-0.5mm}{0.5mm}{2}+\explnorm{\varrho_h^0}{\infty}{(\Omega_h)}{0.55}{-0.0mm}{0.5mm}{2}\,\explnorm{\thetazeroh}{2}{(\Omega_h)}{0.55}{-0.5mm}{0.5mm}{2}\big) \notag \\[2mm]
    &\lesssim \lnorm{\varrho_0}{\infty}{(\Omega)}{0.55}{-0.0mm}\,\expwnorm{\fatu_0}{1,\!2}{(\Omega)^d}{0.55}{-0.5mm}{0.5mm}{2} + \explnorm{\varrho_0}{\infty}{(\Omega)}{0.55}{0.0mm}{0.5mm}{\gamma}\,\explnorm{\thetazero}{\gamma}{(\Omega)}{0.55}{0.0mm}{0.5mm}{\gamma} + \hd\big(\explnorm{\varrho_0}{2}{(\Omega)}{0.55}{-0.5mm}{0.5mm}{2}+\explnorm{\varrho_0}{\infty}{(\Omega)}{0.55}{-0.0mm}{0.5mm}{2}\,\explnorm{\thetazero}{2}{(\Omega)}{0.55}{-0.5mm}{0.5mm}{2}\big) 
    \lesssim 1.
\end{align*}
Using this observation, it is easy to establish the first estimate in (\ref{energy_est_1}), the first two estimates in (\ref{energy_est_2}), the estimates in (\ref{energy_est_3}), and the estimates (\ref{energy_est_5})--(\ref{energy_est_8}). Then, due to Corollary \ref{cor_sol}(\hyperlink{cor_ii}{i}), the second estimate in (\ref{energy_est_1}) follows from the first estimate in (\ref{energy_est_3}). Next, applying Hölder's inequality, we observe that
\begin{align*}
    & \lnorm{\varrho_h\overline{\fatu_h}}{\infty}{(0,T;\Lp{2\gamma/(\gamma+1)}{0.0mm}(\Omega_h)^d)}{0.55}{0.0mm} \notag \\[2mm]
    & \qquad = \sup_{k\,\in\,\{1,\dots,N_T\}}\left\{\lnorm{\varrho_h^k\overline{\fatu_h^k}}{2\gamma/(\gamma+1)}{(\Omega_h)^d}{0.55}{-0.5mm}\right\} 
    \leq \sup_{k\,\in\,\{1,\dots,N_T\}}\left\{\explnorm{\varrho_h^k}{\gamma}{(\Omega_h)}{0.55}{-0.2mm}{0.5mm}{1/2}\,\explnorm{\varrho_h^k|\overline{\fatu_h^k}|^2}{1}{(\Omega_h)}{0.55}{-0.5mm}{0.5mm}{1/2}\right\} \notag \\[2mm]
    & \qquad \leq \left(\,\sup_{k\,\in\,\{1,\dots,N_T\}}\left\{\lnorm{\varrho_h^k}{\gamma}{(\Omega_h)}{0.55}{0.0mm}\right\}\sup_{k\,\in\,\{1,\dots,N_T\}}\left\{\lnorm{\varrho_h^k|\overline{\fatu_h^k}|^2}{1}{(\Omega_h)}{0.55}{-0.5mm}\right\}\right)^{1/2} \notag \\[2mm]
    & \qquad \leq \left(\lnorm{\varrho_h}{\infty}{(0,T;\Lp{\gamma}{0.5mm}(\Omega_h))}{0.55}{0.0mm}\,\lnorm{\varrho_h|\overline{\fatu_h}|^2}{\infty}{(0,T;\Lp{1}{0.0mm}(\Omega_h))}{0.55}{0.0mm}\right)^{1/2}\,.
\end{align*}
Consequently, the last estimate in (\ref{energy_est_1}) follows from the first two. Furthermore, an application of Poincaré's inequality (\ref{sobolev}) reveals that the last estimate in (\ref{energy_est_2}) is a consequence of the first. Due to Corollary \ref{cor_sol}(\hyperlink{cor_ii}{i}), the validity of the first estimate in (\ref{energy_est_4}) results from the third estimate in (\ref{energy_est_1}). Using Hölder's inequality and the second estimate in (\ref{trace_neg_lp}), we deduce that
\begin{align*}
    & \lnorm{\varrho_h\overline{\fatu_h}}{2}{(0,T;\Lp{2}{0.0mm}(\Omega_h)^d)}{0.55}{-0.5mm} \notag \\[2mm]
    & \qquad = \left(\intt\lnorm{(\varrho_h^2|\overline{\fatu_h}|^2)(t,\cdot)}{1}{(\Omega_h)}{0.55}{-0.5mm}\;\dt\right)^{1/2} \lesssim \left(\intt \explnorm{\varrho_h(t,\cdot)}{3}{(\Omega_h)}{0.55}{-0.5mm}{0.5mm}{2}\,\explnorm{\overline{\fatu_h}(t,\cdot)}{6}{(\Omega_h)^d}{0.55}{-0.5mm}{0.5mm}{2}\;\dt\right)^{1/2} \notag \\[2mm]
    & \qquad \lesssim h^{-\frac{d+3\delta}{6}}\left(\intt \explnorm{h^{\delta/2}\varrho_h(t,\cdot)}{2}{(\Omega_h)}{0.55}{-0.5mm}{0.5mm}{2}\,\explnorm{\overline{\fatu_h}(t,\cdot)}{6}{(\Omega_h)^d}{0.55}{-0.5mm}{0.5mm}{2}\;\dt\right)^{1/2} \notag \\[2mm]
    & \qquad \lesssim h^{-\frac{d+3\delta}{6}}\,\lnorm{h^{\delta/2}\varrho_h}{\infty}{(0,T;\Lp{2}{0.0mm}(\Omega_h))}{0.55}{0.0mm}\,\lnorm{\overline{\fatu_h}}{2}{(0,T;\Lp{6}{0.0mm}(\Omega_h)^d)}{0.55}{-0.5mm}\,.
\end{align*}
Therefore, the second estimate in (\ref{energy_est_4}) follows from the third estimate in (\ref{energy_est_2}), the second estimate in (\ref{energy_est_3}), and (\ref{bound_piq}). Finally, we combine Hölder's inequality, the estimates (\ref{est_vh_2}) and (\ref{energy_est_5}), the first estimate in (\ref{trace_neg_lp}), and the first and third estimate in (\ref{energy_est_2}) to conclude that
\begin{align*}
    &\intt\!\!\inteint \big|\jump{\varrho_h}\smallgmean{\fatu_h\cdot\ngamma}\big|\;\dsxdt \notag \\[2mm]
    & \qquad \lesssim h^{-\delta/2}\left(\hd\intt\!\!\inteint \jump{\varrho_h}^2\,|\smallgmean{\fatu_h\cdot\ngamma}|\;\dsxdt\right)^{1/2}\left(\intt\!\!\inteint |\smallgmean{\fatu_h}|\;\dsxdt\right)^{1/2} \notag \\[2mm]
    & \qquad \lesssim h^{-\delta/2}\left(\intt h^{-1}\left(\lnorm{\fatu_h(t,\cdot)}{1}{(\Omega_h)^d}{0.55}{-0.5mm}+h\,\lnorm{\gradh\fatu_h(t,\cdot)}{1}{(\Omega_h)^{d\times d}}{0.55}{-0.5mm}\right)\dt\right)^{1/2} \notag \\[2mm]
    & \qquad \lesssim h^{-\delta/2}\left(h^{-1}\lnorm{\fatu_h}{2}{(0,T;\Lp{2}{0.0mm}(\Omega_h)^d)}{0.55}{-0.5mm}+\lnorm{\gradh\fatu_h}{2}{(0,T;\Lp{2}{0.0mm}(\Omega_h)^{d\times d})}{0.55}{-0.5mm}\right)^{1/2} \lesssim h^{-\delta/2}(1+h^{-1/2})\,.
\end{align*}
We note in passing that estimate (\ref{energy_est_10}) can be proven in the same way.
\end{proof}

\end{document}